\documentclass[10pt,reqno]{amsart}

\usepackage[latin2]{inputenc}
\usepackage{amsmath}
\usepackage{graphicx}
\usepackage{amssymb}
\usepackage{esint}
\usepackage{xcolor}
\usepackage{amsthm}
\usepackage{epsfig}
\usepackage{enumitem}
\usepackage{mathtools}
\usepackage{esvect}
\usepackage{dsfont}
\usepackage{tikz}
\usetikzlibrary{calc,intersections,through,backgrounds,patterns,arrows.meta, math,through}
\usepackage{graphicx}
\usepackage{caption}
\usepackage{subcaption}
\usepackage[english]{babel}

\newtheorem{theorem}{Theorem}
\newtheorem{proposition}[theorem]{Proposition}
\newtheorem{lemma}[theorem]{Lemma}

\newtheorem{definition}[theorem]{Definition}
\newtheorem{construction}[theorem]{Construction}

\newtheorem*{theorem*}{Theorem}

\theoremstyle{definition}

\allowdisplaybreaks[4]

\newcommand{\BV}{\operatorname{BV}} 
\newcommand{\supp}{\operatorname{supp}}

\newcommand{\dist}{\operatorname{dist}} 
\newcommand{\Id}{\operatorname{Id}}
 
\newcommand{\sdist}{\operatorname{sdist}}

\newcommand{\Om}{\Omega} 
\newcommand{\no}{n_{\partial\Omega}} 
\newcommand{\tano}{\tau_{\partial\Omega}} 
\newcommand{\R}{\mathbb{R}}

\newcommand{\h}{\operatorname{H}}
\newcommand{\haus}{\mathcal{H}^{d-1}}
\newcommand{\pt}{\partial_t}

\newcommand{\fract}{\frac{1}{2}}
\newcommand{\fracmu}{\frac{\mu}{2}}
\newcommand{\into}{\int_{\Om}}
\newcommand{\intpo}{\int_{\partial \Om}}

\newcommand{\C}{C}
\newcommand{\Ss}{\mathbb{S}^{d-1}}

\newcommand{\dH}{\,\mathrm{d}\mathcal{H}^{d-1}}

\newcommand{\dS}{\,\mathrm{d}S}

\newcommand{\dx}{\,\mathrm{d}x}
\newcommand{\dy}{\,\mathrm{d}y}
\newcommand{\dV}{\,\mathrm{d}V_t(x,s)}
\newcommand{\dVm}{\,\mathrm{d}|V_t|_{\Ss}}
\newcommand{\dchi}{\,\mathrm{d}|\nabla \chi|}
\newcommand{\dchiu}{\,\mathrm{d}|\nabla \chi_u|}
\newcommand{\dchiut}{\,\mathrm{d}|\nabla \chi_u(\cdot, T)|}
\newcommand{\dchiuo}{\,\mathrm{d}|\nabla \chi_u(\cdot, 0)|}
\newcommand{\dt}{\,\mathrm{d}t}

\newcommand{\cupdot}{\mathbin{\mathaccent\cdot\cup}}

\definecolor{Yellow}{rgb}{0.95,0.9,0.0} 
\definecolor{Red}{rgb}{0.8,0.1,0.1}
\definecolor{Green}{rgb}{0.1,0.65,0.2}
\definecolor{Blue}{rgb}{0.1,0.1,0.8}
\definecolor{Purple}{rgb}{0.7,0.1,0.7}
\definecolor{Grey}{rgb}{0.6,0.6,0.6}

\definecolor{YELLOW}{rgb}{0.95,0.9,0.0} 
\definecolor{RED}{rgb}{0.8,0.1,0.1}
\definecolor{GREEN}{rgb}{0.25,0.65,0.1}
\definecolor{BLUE}{rgb}{0.1,0.1,0.8}
\definecolor{PURPLE}{rgb}{0.7,0.1,0.7}

\usepackage[linktocpage=true,colorlinks=true,linkcolor=Blue,citecolor=Green]{hyperref}

\begin{document}

\title[Stability for two-phase fluid flow with $90^\circ$ contact angle]
{Weak-strong uniqueness for the Navier--Stokes equation for two fluids with 
ninety degree contact angle and same viscosities}
    %

\author{Sebastian Hensel}
\address[Sebastian Hensel]{Institute of Science and Technology Austria (IST Austria), Am~Campus~1, 
3400 Klosterneuburg, Austria}
\email{sebastian.hensel@ist.ac.at}
\curraddr{Hausdorff Center for Mathematics, Universit{\"a}t Bonn, Endenicher Allee 62, 53115 Bonn, Germany
(\texttt{sebastian.hensel@hcm.uni-bonn.de})}
    
\author{Alice Marveggio}
\address[Alice Marveggio]{Institute of Science and Technology Austria (IST Austria), Am Campus 1, 3400 Klosterneuburg, Austria} 
\email{alice.marveggio@ist.ac.at}

	
\begin{abstract}
We consider the flow of two viscous and incompressible fluids
within a bounded domain modeled by means of a two-phase Navier--Stokes system. 
The two fluids are assumed to be immiscible, meaning that they are separated 
by an interface. With respect to the motion of the interface, we consider
pure transport by the fluid flow. Along the boundary of the domain,
a complete slip boundary condition for the fluid velocities and
a constant ninety degree contact angle condition for the interface are assumed.

The main result of the present work establishes in~2D a weak-strong uniqueness
result in terms of a varifold solution concept \`a la Abels (Interfaces Free Bound.~9, 2007).
The proof is based on a relative entropy argument. More precisely, we extend
ideas from the recent work of Fischer and the first author (Arch.\ Ration.\ Mech.\ Anal.~236, 2020)
to incorporate the contact angle condition. To focus on the effects of the 
necessarily singular geometry, we work for simplicity in the regime of
same viscosities for the two fluids.


\end{abstract}	
	
\maketitle 

\section{Introduction}	
\subsection{Context}
The question of uniqueness or non-uniqueness of weak solution
concepts in the context of classical fluid mechanics models has
seen a series of intriguing breakthroughs throughout the last three decades.
In case of the Euler equations, the journey started with the
seminal works of Scheffer~\cite{Scheffer} and Shnirelman~\cite{Shnirelman}
providing the construction of compactly supported nonzero
weak solutions. The first example of an energy dissipating weak solution
to the Euler equations is again due to Shnirelman~\cite{Shnirelman2000}.
Later, De~Lellis and Sz\'ekelyhidi Jr.\ not only strengthened these results 
in their groundbreaking works (see, e.g., \cite{DeLellisSzekelyhidi} and~\cite{Lellis2012}),  
but in retrospect even more importantly introduced a novel perspective on the problem:
their proofs are based on a nontrivial transfer of convex integration techniques from
typically geometric PDEs to the framework of the Euler equations. Indeed,
their ideas eventually culminated in the resolution of Onsager's conjecture
by Isett~\cite{Isett}; see also the work of 
Buckmaster, De~Lellis, Sz\'ekelyhidi Jr.\ and Vicol~\cite{Buckmaster2018a}.

By now, these developments also generated spectacular results for
the Navier--Stokes equations. For instance, Buckmaster and Vicol~\cite{Buckmaster2019}
as well as Buckmaster, Colombo and Vicol~\cite{Buckmaster2021}
establish that mild solutions in the energy class are non-unique. The constructed
solutions are not Leray--Hopf solutions, i.e., it is not proven that they
are subject to the energy dissipation inequality. However,
Albritton, Bru\'{e} and Colombo~\cite{Albritton2021} 
even show in a very recent preprint that one can construct an external force
such that there exists a finite time horizon so that
one may construct at least two distinct Leray--Hopf solutions
for the associated forced full-space Navier--Stokes equations in~3D
(both starting from zero initial data). 

Hence, in terms of uniqueness of weak solutions the best one
can expect in general is essentially a weak-strong uniqueness principle.
Roughly speaking, this refers to uniqueness of weak solutions
within a class of sufficiently regular solutions. In the
context of the incompressible Navier--Stokes equations, such results are
classical and can be traced back to the works of Leray~\cite{Leray1934},
Prodi~\cite{Prodi1959} and Serrin~\cite{Serrin2}. In the 
case of the compressible Navier--Stokes equations, we mention the
works of Germain~\cite{Germain}, Feireisl, Jin and Novotn{\'y}~\cite{FeireislJinNovotny}, 
as well as Feireisl and Novotn{\'y}~\cite{FeireislNovotny}.
The usual strategy to establish these results is based on
a by now widely used method which infers weak-strong uniqueness
from a quantitative stability estimate for a suitable distance measure
between two solutions, the so-called relative entropy (or relative energy).
We refer to the survey article by Wiedemann~\cite{Wiedemann2018}
for an overview on the relative entropy method in the context of
mathematical fluid mechanics.

In the present work, we are concerned with the question of weak-strong
uniqueness with respect to a two-phase free boundary
fluid problem within a physical domain~$\Omega\subset\R^d,\,d\in\{2,3\}$.
More precisely, we study this question in terms of varifold solutions \`a la Abels~\cite{Abels} 
for the specific evolution problem of the flow of
two incompressible Navier--Stokes fluids separated by a sharp interface.
Along the boundary of the domain,
a complete slip boundary condition for the fluid velocities 
as well as a constant ninety degree contact angle condition for the interface are assumed.
For the precise PDE~formulation of the model, we refer to Subsection~\ref{subsec:PDEformulation}.
For a discussion of the weak solution concept and its precise definition,
we instead refer to Subsection~\ref{subsec:varifoldSol} and Definition~\ref{Def_varsol}, 
respectively. Even when neglecting the fluid mechanics, uniqueness of weak solutions in form
of a weak-strong uniqueness principle is in general the best one can expect
also for interface evolution problems. In this context, this is due to the formation
of singularities and topology changes; see already, for instance,
the work of Brakke~\cite{Brakke} for mean curvature flow of networks
of interfaces in~$\R^2$ or the work of Angenent, Ilmanen and Chopp~\cite{AngenentIlmanenChopp}
for mean curvature flow of surfaces in~$\R^3$.

When restricting to the full-space setting $\Omega=\R^d$,
Fischer and the first author~\cite{Fischer2020} recently established
a weak-strong uniqueness principle up to the first topology change
for the corresponding two-phase free boundary fluid problem considered in this work.
Their approach relies on a suitable extension of the relative
entropy method to get control on the difference in the underlying 
geometries of two solutions; cf.\ Subsection~\ref{subsec:weakStrongUniqueness}
for a discussion in this direction. Their ideas were later generalized
by Fischer, Laux, Simon and the first author~\cite{Fischer2020a} to derive
a weak-strong uniqueness principle for BV~solutions of Laux and Otto~\cite{Laux2016} 
to mean curvature flow of networks of interfaces in~$\R^2$,
or even for canonical multiphase Brakke flows of Stuvard and Tonegawa~\cite{Stuvard2021}
(cf.\ also~\cite{Hensel2021e}).

The main goal of the present work is to extend parts of the
analysis of~\cite{Fischer2020} to include the nontrivial boundary effects.
More precisely, in our main result we establish a weak-strong uniqueness principle
in the framework of varifold solutions to the two-phase free boundary fluid problem 
specified in Subsection~\ref{subsec:PDEformulation} below. We refer to Theorem~\ref{theo:mainResult} 
for the precise mathematical formulation of our result.
In the spirit of~\cite{Fischer2020a}, we also derive a conditional
weak-strong uniqueness result in the three-dimensional setting;
cf.\ Proposition~\ref{prop:conditionalWeakStrong} for the precise statement.
%

\subsection{Strong PDE formulation of the two-phase fluid model}
\label{subsec:PDEformulation}
We start with a description of the underlying evolving geometry.
Denoting by~\(\Om\) a bounded domain in \( \R^d\) 
with smooth and orientable boundary~\(\partial \Om\), $d\in\{2,3\} $,
each of the two fluids is contained within a time-evolving domain $\Om^+(t) \subset \Omega$ 
resp.\ $\Om^-(t) \subset \Omega$, $t \in  [0, T)$.
The interface separating both fluids is given as the common boundary between the two fluid domains.
Denoting it at time $t \in [0,T)$ by~$I(t) \subset\overline{\Omega}$, 
we then have a disjoint decomposition of~$\overline{\Omega}$
in form of $\overline{\Om} = \Om^+(t) \cup \Om^-(t) \cup (I(t) \cap \Omega) \cup \partial \Om$
for every $t \in  [0, T)$. We write~$\no$ to refer to the inner pointing unit normal vector field 
of~\(\partial \Om\), as well as~$n_I(\cdot,t)$ to denote the unit normal vector field along~$I(t)$
pointing towards $\Omega^+(t)$, $t \in [0,T)$.

With respect to internal boundary conditions along the separating interface, 
first, a no-slip boundary condition is assumed.
This in fact allows to
represent the two fluid velocity fields by a single continuous vector field~\(v\). 
We also consider a single scalar field~\(p\) as the pressure, 
which in contrast may jump across the interface. 
Second, along the interface the internal forces of the fluids 
have to match a surface tension force. Denoting by~$\chi(\cdot,t)$ the characteristic function 
associated with the domain~$\Omega^+(t)$, $t \in [0,T)$, and
defining $\mu(\chi) := \mu^+\chi + \mu^{-}(1{-}\chi)$ with~\(\mu^+\) and~\(\mu^-\) 
being the viscosities of the two fluids, the stress tensor 
$\mathbb{T} := \mu(\chi)  (\nabla v {+} \nabla v^\mathsf{T}  ) - p \Id$
is required to satisfy
\begin{align} \label{kineticjump}
[[\mathbb{T} n_{I}  ]](\cdot,t) = \sigma \h_I(\cdot,t)
\quad\text{along } I(t)
\end{align}
for all $t \in [0,T)$, where moreover $[[ \cdot ]] $ denotes the jump in normal 
direction, $\sigma > 0$ is the fixed surface tension 
coefficient of the interface, and \(\h_I(\cdot,t)\) represents the mean curvature 
vector field along the interface \(I(t)\), $t \in [0,T)$.

With respect to boundary conditions along~$\partial\Omega$,
we assume in terms of the two fluids a complete slip boundary conditions. In terms of
the evolving geometry, a ninety degree contact angle condition at the contact set 
of the fluid-fluid interface with the boundary of the domain is imposed.
Mathematically, this amounts to
\begin{align}
v(\cdot,t) \cdot \no  &= 0  \label{bcv} 
&& \text{along } \partial\Omega,
\\ \label{bcT}
\big(\no \cdot \mu(\chi) (\nabla v + \nabla v ^\mathsf{T} )(\cdot,t)B\big) &= 0
&& \text{along } \partial\Omega
\end{align}
for all $t \in [0,T)$ and all tangential vector fields~$B$ along~$\partial\Omega$,
as well as
\begin{align} \label{90contactangle}
n_{I} (\cdot,t)	\cdot \no &= 0 
&& \text{along } I(t) \cap \partial\Omega
\end{align}
for all $t \in [0,T)$. These boundary conditions not only prescribe that the fluid cannot exit from the domain 
and that it can move only tangentially to its boundary, but they also exclude any external 
contribution to the viscous stress and any friction effect with the boundary.
Observe also that the ninety degree contact angle condition is consistent with the complete slip boundary conditions~\eqref{bcv} and~\eqref{bcT}, in the sense that~\eqref{90contactangle} together with~\eqref{bcv} implies~\eqref{bcT}. 
Furthermore, the ninety degree contact angle may be imposed only as an initial condition: 
for later times it can be deduced using~\eqref{bcv} and~\eqref{bcT} and a Gronwall-type argument.
For details, see the remark after Definition~\ref{Def_strongsol}.

Now, defining $\rho(\chi) := \rho^+\chi + \rho^{-}(1{-}\chi)$ with~\(\rho^+\) and~\(\rho^-\) 
representing the densities of the two fluids, the fluid motion is given 
by the incompressible Navier--Stokes equation, which by~\eqref{kineticjump}
and~\eqref{bcT} can be formulated as  
\begin{align}
\pt\big(\rho (\chi) v\big) + \nabla\cdot\big(\rho (\chi) v \otimes v\big) 
&= - \nabla p + \nabla \cdot \big(\mu(\chi) (\nabla v + \nabla v^\mathsf{T} )\big) 
+ \sigma \h_I |\nabla \chi|\llcorner\Omega, \label{NS} \\
\nabla \cdot v &= 0, \label{inc}
\end{align}
where \(|\nabla \chi|(\cdot,t)\llcorner\Omega\) represents the surface measure 
\(\haus\llcorner(I(t) \cap \Omega)\), $t \in [0,T)$.
Second, the interface is assumed to be transported along the fluid flow.
In other words, the associated normal velocity of the interface is given
by the normal component of the fluid velocity~$v$. Thanks to~\eqref{bcv},
\eqref{90contactangle} and~\eqref{inc}, this is formally equivalent to
\begin{align}
\pt \chi + (v \cdot \nabla ) \chi &= 0. \label{traneq}
\end{align}

Finally, from a modeling perspective, the total energy of the 
PDE system~\eqref{NS}--\eqref{traneq} is given by the sum of kinetic and surface tension energies
\begin{equation} \label{enintro}
E[\chi, v] := \into \frac12 \rho(\chi) |v|^2 \dx + \sigma \into 1 \dchi + \sigma^+ \intpo \chi \dS + \sigma^- \intpo (1-\chi) \dS,
\end{equation}
where $\sigma^+$ and $ \sigma^-$ are the surface tension coefficients 
of $\partial \Om \cap \overline{\Om^+_t}$ and $\partial \Om \cap \overline{\Om^-_t}$, 
respectively. Note that the ninety 
degree contact angle condition~\eqref{90contactangle} corresponds to $\sigma^-= \sigma^+$. 
Indeed, a general constant contact angle~$\alpha \in (0,\pi)$ is prescribed by Young's equation which
in our notation reads as follows 
 \[
 \sigma \cos\alpha = {\sigma^+ - \sigma^-}.
 \]
In particular, by subtracting the constant $\int_{\partial\Omega} 1 \,\mathrm{d}S$
from~\eqref{enintro} we see that the relevant part of the total energy
does not contain a surface energy contribution along~$\partial\Omega$
in our special case of a constant ninety degree contact angle.
By formal computations, one finally observes that this energy
satisfies an energy dissipation inequality 
\begin{equation} \label{eq:dissipIntro}
E[\chi, v](T') + \int_{0}^{T'} \into \frac{\mu(\chi)}{2}  
|\nabla v + \nabla v ^T|^2 \dx \,\mathrm{d}t \leq E[\chi, v](0),
\quad T' \in  [0, {T}).
\end{equation}  

\subsection{Varifold solutions for two-phase fluid flow with $90^\circ$ contact angle}
\label{subsec:varifoldSol}
In terms of weak solution theories for the 
evolution problem~\eqref{NS}--\eqref{traneq},
the energy dissipation inequality suggests to consider velocity
fields in the space $L^\infty(0,T;L^2(\Omega;\R^d)) \cap L^2(0,T;H^1(\Omega;\R^d))$,
and the evolving geometry may be modeled based on a
time-evolving set of finite perimeter so that the associated characteristic
function~$\chi$ is an element of $L^\infty(0,T;BV(\Omega;\{0,1\}))$.

However, a well-known problem arises when considering limit points
of a sequence of pairs~$(\chi_k,v_k)_{k\in\mathbb{N}}$ representing solutions originating
from an approximation scheme for~\eqref{NS}--\eqref{traneq}. 
Ignoring the time variable for the sake of the discussion, the main point is that a uniform bound 
of the form $\sup_{k \in \mathbb{N}} \|\chi_k\|_{BV(\Omega)} < \infty$ in
general does not suffice to pass to the limit (not even subsequentially)
in the surface tension force $\sigma\h_{I_k} |\nabla \chi_k|\llcorner\Omega$. 
Recalling that we work in a setting with a ninety degree angle condition, 
this term is represented in distributional form by
\begin{align}
\label{eq:meanCurvatureDistributional}
\int_{\Omega} \h_{I_k} \cdot B \,\mathrm{d}|\nabla \chi_k|
= - \int_{\Omega} (\mathrm{Id} - n_k \otimes n_k) : \nabla B \,\mathrm{d}|\nabla \chi_k|
\end{align}
for all smooth vector fields~$B$ which are tangential along~$\partial\Omega$,
where $n_k = \frac{\nabla\chi_k}{|\nabla\chi_k|}$ denotes the
measure-theoretic interface unit normal. 
One may pass to the limit on the right hand side of the previous display 
provided $|\nabla\chi_k|(\Omega) \to |\nabla\chi|(\Omega)$.
However, for standard approximation schemes 
there is in general no reason why this should be true. For instance, 
hidden boundaries may be generated within~$\Omega$ in the limit. Furthermore, but now specific
to the setting of a bounded domain, nontrivial parts of the approximating 
interfaces may converge towards the boundary~$\partial\Omega$.
  
The upshot is that one has to pass to an even weaker
representation of the surface tension force than~\eqref{eq:meanCurvatureDistributional}. 
A popular workaround 
is based on the concept of (oriented) varifolds.
In the setting of the present work and in view of the preceding discussion,
this in fact amounts to consider the space of finite Radon measures
on the product space $\overline{\Omega}{\times}{\Ss}$. Indeed, introducing
the varifold lift $V_k := |\nabla\chi_k|\llcorner\Omega \otimes (\delta_{n_k(x)})_{x\in\Omega}$
one may equivalently express the right hand side of~\eqref{eq:meanCurvatureDistributional}
in terms of the functional $B \mapsto -\int_{\overline{\Omega}{\times}{\Ss}}
(\mathrm{Id}{-}s\otimes s):\nabla B \,\mathrm{d}V_k(x,s)$ which is now
stable with respect to weak$^*$ convergence in the space of
finite Radon measures on $\overline{\Omega}{\times}{\Ss}$. 
Note also that by the choice of working
in a varifold setting, one expects $\sigma\int_{\overline{\Omega}} 1 \,\mathrm{d}|V|_{\Ss}$
instead of $\sigma \into 1 \dchi$ as the interfacial energy contribution in~\eqref{enintro},
where the finite Radon measure $|V|_{\Ss}$ denotes the mass of the varifold~$V$.

Motivated by the previous discussion, we give a
full formulation of a varifold solution concept to two-phase fluid
flow with surface tension and constant ninety degree contact angle
in Definition~\ref{Def_varsol} below. This definition is nothing else
but the suitable analogue of the definition by Abels~\cite{Abels}, 
who provides for the full-space setting a global-in-time existence theory for 
such varifold solutions with respect to rather general initial data.
Unfortunately, in the bounded domain case
with non-zero interfacial surface tension, to the best of
our knowledge a global-in-time existence result for varifold solutions is missing. 
In particular, such a result is not contained in the work of Abels~\cite{Abels}.
For this reason, we include in this work at least a sketch of an existence proof. 
To this end, one may follow on one side the higher-level structure of the argument 
given by Abels~\cite{Abels} for the full-space setting. On the other side,
additional arguments are of course necessary due to the specified boundary
conditions for the geometry and the fluids, respectively. These additional arguments
are outlined in Appendix~\ref{appendix}.

\subsection{Weak-strong uniqueness for varifold solutions of two-phase fluid flow}
\label{subsec:weakStrongUniqueness}
In case 
the two fluids occupy the full space~\(\R^d\), $d\in \{2,3\} $, a weak-strong uniqueness result for
Abels'~\cite{Abels} varifold solutions of the system~\eqref{NS}--\eqref{traneq}
was recently established by Fischer and the first author~\cite{Fischer2020}.
Given sufficiently regular 
initial data, it is shown that on the time horizon 
of existence of the associated unique strong solution, 
any varifold solution in the sense of Abels~\cite{Abels} 
starting from the same initial data has to coincide with this strong solution.

This result is achieved by extending a by now several decades old idea
in the analysis of classical PDE models from continuum mechanics to a 
previously not covered class of problems: a relative entropy method
for surface tension driven interface evolution.
The gist of this method can be described as follows. 
Based on a dissipated energy functional, 
one first tries to build an error functional --- the relative entropy ---
which penalizes the difference between two solutions 
in a sufficiently strong sense. A minimum requirement is to 
ensure that the error functional vanishes if and only if the two
solutions coincide. In a second step, one proceeds by computing the time
evolution of this error functional. In a third step, one
tries to identify all the terms appearing in this computation as
contributions which either are controlled by the error functional
itself or otherwise may be absorbed into a residual quadratic term
represented essentially by the difference of the dissipation energies.
One finally concludes by an application of Gronwall's lemma.

The novelty of the work~\cite{Fischer2020} consists of an implementation
of this strategy for the full-space version of the energy functional~\eqref{enintro}.
More precisely, the relative entropy as it was originally constructed 
in the full-space setting in~\cite{Fischer2020} essentially consists of
two contributions. The first aims for a penalization of the difference
of the underlying geometries of the two solutions. This in fact is
performed at the level of the interfaces by introducing a tilt-excess
type error functional with respect to the two associated unit normal
vector fields. To this end, the construction of a suitable extension of the unit normal vector field 
of the interface of the strong solution in the vicinity of its space-time 
trajectory is required. Furthermore, the length of this vector field
is required to decrease quadratically fast as one moves away from the 
interface of the strong solution. The merit of this
is that 
one also obtains a measure of the interface error in terms of the distance between them.

Due to the inclusion of contact point dynamics in form of a constant ninety degree contact angle,
some additional ingredients are needed for the present work. We refer to Subsection~\ref{subsec:relEntropy} 
below for a detailed and mathematical account on the geometric part of the relative entropy functional. 
There are however two notable additional difficulties
in comparison to~\cite{Fischer2020} which are worth
emphasizing already at this point. Both are related to
the required extension~$\xi$ of the unit normal vector
field associated with the evolving interface of the
strong solution. The first is concerned with the correct
boundary condition for the extension~$\xi$ along~$\partial\Omega$.
Since along the contact set the interface intersects the boundary of the 
domain orthogonally, it is natural to enforce~$\xi$ to be tangential
along~$\partial\Omega$. This indeed turns out to be the right
condition as it allows by an integration by parts to rewrite the interfacial part of
the relative entropy as the sum of interfacial energy of the weak solution and
a linear functional with respect to the characteristic function~$\chi$
of the weak solution. This is crucial to even attempt computing the time evolution.

The second difference concerns the actual construction of the
extension~$\xi$. In contrast to~\cite{Fischer2020}, where only a
finite number of sufficiently regular closed curves ($d=2$) or closed surfaces ($d=3$)
are allowed at the level of the strong solution, this results in
a nontrivial and subtle task in the context of the present work
due to the necessarily singular geometry in contact angle problems.
The main difficulty roughly speaking is to provide a construction which
on one side respects the required boundary condition and on the
other side is regular enough to support the computations and estimates
in the Gronwall-type argument. For a complete list of the
required conditions for the extension~$\xi$, we refer to 
Definition~\ref{def:calibrationTwoPhaseFluidFlow} below.

We finally turn to a brief discussion of the second contribution in the
total relative entropy functional from~\cite{Fischer2020}. In principle,
this term on first sight should be nothing else than 
the relative entropy analogue to the kinetic part of the energy of the system,
thus controlling the squared $L^2$-distance between the fluid velocities
of the two solutions. However, as recognized in~\cite{Fischer2020}
a major problem arises for the two-phase fluid problem in the regime of 
different viscosities $\mu^+\neq\mu^-$: without performing a very careful
(and in its implementation highly technical) perturbation of this naive ansatz 
for the fluid velocity error, a Gronwall-type argument will not be realizable;
cf.\ for more details the discussion in~\cite[Subsection~3.4]{Fischer2020}.
Since the main focus of the present work lies on the inclusion
of the ninety degree contact angle condition, we do not delve into these issues
and simply assume for the rest of this work that the viscosities of the two fluids coincide:
$\mu := \mu^+ = \mu^-$.
We emphasize, however, that at least for the construction of the
extension~$\xi$ and the verification of its properties 
we in fact do not rely on this assumption.

\section{Main results}	
	
\subsection{Weak-strong uniqueness and stability of evolutions}
The main result of this work reads as follows.
	
\begin{theorem}
\label{theo:mainResult}
Let $d=2$, and let $\Omega\subset\R^2$ be a bounded domain with orientable and smooth boundary.
Let $(\chi_u,u,V)$ be a varifold solution to the incompressible
Navier--Stokes equation for two fluids in the sense of Definition~\ref{Def_varsol}
on a time interval $[0,T_w)$. Let $(\chi_v,v)$ be a strong solution
to the incompressible Navier--Stokes equation for two fluids in the sense
of Definition~\ref{Def_strongsol} on a time interval $[0,T_s)$ where $T_s\leq T_w$.

Then, for every $T\in (0,T_s)$ there exists a constant $C=C(\chi_v,v,T)>0$
such that the relative entropy functional~\eqref{relent} and the bulk error
functional~\eqref{eq:bulkErrorFunctional} satisfy stability estimates of the form 
\begin{align}
\label{eq:stabilityRelEntropy}
E[\chi_u,u,V|\chi_v,v](t) &\leq Ce^{Ct}E[\chi_u,u,V|\chi_v,v](0), 
\\
\label{eq:stabilityBulkError}
E_{\mathrm{vol}}[\chi_u|\chi_v](t) &\leq Ce^{Ct}\big(
E[\chi_u,u,V|\chi_v,v](0) + E_{\mathrm{vol}}[\chi_u|\chi_v](0)\big)
\end{align}
for almost every $t\in [0,T]$. 

In particular, in case the initial data for the varifold solution and strong solution
coincide, it follows that
\begin{align}
\label{eq:BVisStrong}
\chi_u(\cdot,t) &= \chi_v(\cdot,t),\quad
u(\cdot,t) = v(\cdot,t)
&&\text{a.e.\ in } \Omega \text{ for a.e.\ } t\in [0,T_s),
\\
\label{eq:varifoldIsBV}
V_t &= (|\nabla\chi_u(\cdot,t)|\llcorner\Omega) \otimes
\big(\delta_{\frac{\nabla\chi_u(\cdot,t)}{|\nabla\chi_u(\cdot,t)|}(x)}\big)_{x\in\Omega}
&&\text{for a.e.\ } t\in [0,T_s).
\end{align}
\end{theorem}

The proof of Theorem~\ref{theo:mainResult} may be divided into two steps
as explained in the following two subsections.
	
\subsection{Quantitative stability by a relative entropy approach}
\label{subsec:relEntropy}
Following the general strategy of~\cite{Fischer2020}, 
our weak-strong uniqueness result essentially relies on two
ingredients: \textit{i)} the construction of a suitable extension~$\xi$ of the unit normal
vector field of the interface of a strong solution, and \textit{ii)} 
based on this extension, the introduction of a suitably defined error functional penalizing the
interface error between a varifold and a strong solution in a sufficiently strong sense.
In comparison to~\cite{Fischer2020}, the extension of the unit normal
has to be carefully constructed in the sense that the vector field~$\xi$
is required to be tangent to the domain boundary~$\partial\Omega$
(which is the natural boundary condition in case of a $90^\circ$ contact angle). 
Due to the singular nature of the geometry at the contact set, this
is a nontrivial task. The precise conditions on the extension~$\xi$ 
are summarized as follows.

\begin{definition}[Boundary adapted extension of the interface unit normal]
\label{def:calibrationTwoPhaseFluidFlow}
Let $d\in\{2,3\}$, and let $\Omega\subset\R^d$ be a bounded domain with orientable and smooth boundary.
Let $T\in (0,\infty)$ be a finite time horizon. Let $(\chi_v,v)$ be a strong solution to the incompressible
Navier--Stokes equation for two fluids in the sense of Definition~\ref{Def_strongsol}
on the time interval $[0,T]$.

In this setting, we call a vector field
$\xi\colon \overline{\Omega} \times [0,T] \to \R^d$
a \emph{boundary adapted extension of~$n_{I_v}$ for two-phase fluid flow~$(\chi_v,v)$
with~$90^\circ$ contact angle} if the following conditions are satisfied:

\begin{itemize}[leftmargin=0.4cm]
\begin{subequations}
\item In terms of regularity, it holds
			$\xi \in \big(C^0_tC^2_x\cap C^1_tC^0_x \big)\big(\overline{\Omega{\times}[0,T]}
			\setminus(I_v\cap(\partial\Omega{\times}[0,T]))\big)$.
\item The vector field~$\xi$ extends the unit normal vector field~$n_{I_v}$ (pointing inside~$\Omega^+_v$) 
			of the interface~$I_v$ subject to the conditions 
			\begin{align}
			\label{eq:coercivityByModulationOfLength}
			|\xi| &\leq \max\big\{0,1{-}C\dist^2(\cdot,I_v)\big\}
			&&\text{in } \Omega\times[0,T],
			\\
			\label{eq:boundaryValueXi}
			\xi \cdot \no &= 0
			&&\text{on } \partial\Omega\times[0,T],
			\\
			\label{eq:divConstraintXi}
			\nabla\cdot \xi &= - H_{I_v} 
			&&\text{on } I_v,
			\end{align}
			for some $C>0$. Here, $H_{I_v}$ denotes the scalar mean curvature of the interface~$I_v$
			(oriented with respect to the normal~$n_{I_v}$).
\item The fluid velocity approximately transports the vector field~\(\xi\) in form of
			\begin{align}
			\label{eq:timeEvolutionXi}
			\pt \xi + (v \cdot \nabla)\xi + (\Id - \xi \otimes \xi)(\nabla v)^\mathsf{T}\xi 
			&= O(\dist(\cdot, I_v) \wedge 1)
			&&\text{in } \Omega\times[0,T],
			\\
			\label{eq:timeEvolutionLengthXi}
			\pt |\xi|^2 + (v \cdot \nabla )|\xi|^2 
			&= O(\dist^2(\cdot, I_v) \wedge 1)
			&&\text{in } \Omega\times[0,T].
			\end{align}
\end{subequations}
\end{itemize}
\end{definition}	

Let us comment on the motivation behind this definition.
Given a vector field~$\xi$
with respect to a fixed strong solution~$(\chi_v,v)$
as in the previous definition, 
we may introduce for any varifold solution~$(\chi_u,u,V)$ 
and for all~$t \in [0,T]$ a functional
\begin{align}
\label{def:interfaceRelEntropy}
E[\chi_u,V|\chi_v](t) := \sigma\int_{\overline{\Omega}} 1 \,\mathrm{d} |V_t|_{\Ss}
- \sigma\int_{I_u(t)} \frac{\nabla\chi_u(\cdot,t)}{|\nabla\chi_u(\cdot,t)|} 
\cdot  \xi (\cdot,t)  \dH,
\end{align}
where~\(I_u(t):=\mathrm{supp}|\nabla\chi_u(\cdot,t)|\cap\Omega\) 
denotes the interface associated to the varifold solution. 
The functional~$E[\chi_u, V|\chi_v]$ is a measure for the
interfacial error between the two solutions for the following reasons.
First of all, it is a consequence of the definition of a varifold
solution, cf.\ the compatibility condition~\eqref{compcond},
that for almost every $t \in [0,T]$ it holds $|\nabla \chi_u(\cdot,t)|\llcorner\Omega
\leq |V_t|_{\Ss}\llcorner\Omega$ in the sense of measures on~$\Omega$. In particular,
it follows that the functional~$E[\chi_u, V|\chi_v]$ controls its ``$\BV$-analogue''
\begin{align}
\label{eq:BVrelEntropy}
0 \leq E[\chi_u|\chi_v](t) := \sigma\int_{I_u(t)} 
1 - \frac{\nabla\chi_u(\cdot,t)}{|\nabla\chi_u(\cdot,t)|} 
\cdot  \xi (\cdot,t)  \dH \leq E[\chi_u,V|\chi_v](t).
\end{align}
Introducing the Radon--Nikod\'{y}m derivative $\theta_t := 
\frac{\mathrm{d} |\nabla \chi_u (\cdot, t)|\llcorner\Omega}
{ \mathrm{d} |V_t|_{\Ss}\llcorner\Omega}$, 
one can be even more precise in the sense that
\begin{align}
\label{eq:VarifoldrelEntropy}
E[\chi_u,V|\chi_v](t) = \sigma \int_{\partial\Omega} 1 \,\mathrm{d} |V_t|_{\Ss}
+ \sigma \int_{\Omega} 1- \theta_t \,\mathrm{d} |V_t|_{\Ss} + E[\chi_u|\chi_v](t).
\end{align}
This representation of the functional~$E[\chi_u,V|\chi_v]$
as well as the length constraint~\eqref{eq:coercivityByModulationOfLength}
for the vector field~$\xi$ lead to the following two observations.
First, the functional~$E[\chi_u,V|\chi_v]$ controls the mass
of hidden boundaries and higher multiplicity interfaces
(i.e., where $\theta_t \in [0,1)$) in the sense of 
\begin{align}
\label{eq:multiplicityControl}
\sigma \int_{\partial\Omega} 1 \,\mathrm{d} |V_t|_{\Ss} + 
\sigma \int_{\Omega} 1- \theta_t \,\mathrm{d} |V_t|_{\Ss} \leq E[\chi_u,V|\chi_v](t).
\end{align}
Second, because of~\eqref{eq:coercivityByModulationOfLength}
it measures the interface error in the sense that
\begin{align}
\label{eq:BVtiltExcessControl}
\sigma \int_{I_{u}(t)} \frac{1}{2}\bigg| 
\frac{\nabla\chi_u(\cdot,t)}{|\nabla\chi_u(\cdot,t)|} - \xi \bigg|^2 \dH
&\leq E[\chi_u|\chi_v](t),
\\ \label{eq:BVdistControl}
\sigma \int_{I_{u}(t)} \min\big\{1,C\dist^2(\cdot,I_v(t))\big\} \dH
&\leq E[\chi_u|\chi_v](t).
\end{align}
On a different note, the compatibility condition~\eqref{compcond}
satisfied by a varifold solution together with the
boundary condition~\eqref{eq:boundaryValueXi} also allows to represent the
error functional~$E[\chi_u,V|\chi_v]$ in the alternative form
\begin{align}
\label{eq:auxRepVarifoldRelEntropy}
E[\chi_u,V|\chi_v](t) = \sigma\int_{\overline{\Omega}{\times}\Ss} 
1 - s\cdot\xi \,\mathrm{d} V_t,
\end{align}
which then entails as a consequence of~\eqref{eq:coercivityByModulationOfLength}
\begin{align}
\label{eq:tiltExcessVarifoldControl}
\sigma\int_{\overline{\Omega}{\times}\Ss} 
\frac{1}{2} |s - \xi|^2 \,\mathrm{d} V_t
&\leq E[\chi_u,V|\chi_v](t),
\\
\sigma\int_{\overline{\Omega}} \min\big\{1,C\dist^2(\cdot,I_v(t))\big\} \,\mathrm{d} |V_t|_{\Ss}
&\leq E[\chi_u,V|\chi_v](t).
\end{align}
Finally, let us quickly discuss what is implied by $E[\chi_u,V|\chi_v](t) = 0$.
We claim that~\eqref{eq:varifoldIsBV} and $I_u(t) \subset I_v(t)$
up to $\mathcal{H}^{d-1}$-negligible sets have to
be satisfied. Indeed, the latter follows directly from~\eqref{eq:BVrelEntropy}
and~\eqref{eq:BVdistControl}. The former is best seen when 
representing the varifold~$V_t\llcorner(\Omega{\times}{\Ss})$ by its 
disintegration~$(|V_t|_{\Ss}\llcorner\Omega)\otimes(\nu_{x,t})_{x\in\Omega}$.
Then, it follows on one side from~\eqref{eq:multiplicityControl} that
$|V_t|_{\Ss}\llcorner\partial\Omega=0$ and
$|V_t|_{\Ss}\llcorner\Omega = |\nabla\chi_u(\cdot,t)|\llcorner\Omega$ as measures 
on~$\partial\Omega$ and~$\Omega$, respectively, and then
on the other side that~$\nu_{x,t}=\delta_{\frac{\nabla\chi_u(\cdot,t)}{|\nabla\chi_u(\cdot,t)|}(x)}$ 
for $|\nabla\chi_u(\cdot,t)|$-a.e.\ $x\in\Omega$ due to
\begin{align*}
&\int_{\Omega} \int_{\Ss} \frac{1}{2} \bigg|s - \frac{\nabla\chi_u(\cdot,t)}{|\nabla\chi_u(\cdot,t)|}(x)\bigg|^2
\,\mathrm{d}\nu_{x,t}(s) \,\mathrm{d}(|\nabla\chi_u(\cdot,t)|\llcorner\Omega)(x)
\\&
= \int_{\Omega} \int_{\Ss} 1 - s \cdot \frac{\nabla\chi_u(\cdot,t)}{|\nabla\chi_u(\cdot,t)|}(x)
\,\mathrm{d}\nu_{x,t}(s) \,\mathrm{d}(|\nabla\chi_u(\cdot,t)|\llcorner\Omega)(x)
= 0,
\end{align*}
where for the last equality we simply plugged in the compatibility condition~\eqref{compcond}
and again $|V_t|_{\Ss}\llcorner\partial\Omega=0$ as well as
$|V_t|_{\Ss}\llcorner\Omega = |\nabla\chi_u(\cdot,t)|\llcorner\Omega$.

Apart from these coercivity conditions, it is
equally important to be able to estimate the time evolution
of the error functional~$E[\chi_u,V|\chi_v]$. The main
observation in this regard is that the functional
can be rewritten as a perturbation of the interface 
energy~$E[\chi_u,V](t) := \sigma\int_{\overline{\Omega}} 1 \,\mathrm{d} |V_t|_{\Ss}$
which is linear in the dependence on the indicator function~$\chi_u$. 
Indeed, thanks to the boundary condition~\eqref{eq:boundaryValueXi}
for the extension~$\xi$, a simple integration by parts readily reveals
\begin{align}
\label{eq:repRelEntropyTimeEvol}
E[\chi_u,V|\chi_v](t) = E[\chi_u,V](t) 
+ \sigma \int_{\Omega} \chi_u(\cdot,t)(\nabla\cdot\xi)(\cdot,t) \dx.
\end{align}
This structure is in fact the very reason why we call~$E[\chi_u,V|\chi_v]$ a relative entropy.
Computing the time evolution of~$E[\chi_u|\chi_v]$ then
only requires to exploit the dissipation of energy and
using $\nabla\cdot\xi$ as a test function in the evolution
equation of the phase indicator~$\chi_u$ of the varifold solution.
The latter in turn requires knowledge on the time evolution
of~$\xi$ itself, which is encoded in terms of the fluid velocity~$v$
through the equations~\eqref{eq:timeEvolutionXi} and~\eqref{eq:timeEvolutionLengthXi}.
The condition~\eqref{eq:divConstraintXi} is natural in view of the interpretation of~$\xi$ 
as an extension of the unit normal~$n_{I_v}$ away from the
interface~$I_v$.

Even though all of this may already be quite promising,
there is one small caveat: obviously, one can not deduce 
from~$E[\chi_u,V|\chi_v]=0$ that $\chi_u=\chi_v$ (e.g., $\chi_u$
representing an empty phase is consistent with having vanishing relative entropy). This
lack of coercivity in the regime of vanishing interface measure
motivates to introduce a second
error functional which directly
controls the deviation of~$\chi_u$ from~$\chi_v$.
The main input to such a functional is captured in the following definition.

\begin{definition}[Transported weight]
\label{def:transportedWeight}
Let $d\in\{2,3\}$, and let $\Omega\subset\R^d$ be a bounded domain with orientable and smooth boundary.
Let $T\in (0,\infty)$ be a finite time horizon, consider a solenoidal vector field 
$v\in L^2([0,T];H^1(\Omega;\R^d))$ with $(v\cdot n_{\partial\Omega})|_{\partial\Omega}=0$, 
and let $(\Omega_v^+(t))_{t\in [0,T]}$ be a family of
sets of finite perimeter in $\Omega$. Denote by $I_v(t)$, $t\in [0,T]$,
the reduced boundary of $\Omega_v^+(t)$ in $\Omega$. Writing $\chi_v(\cdot,t)$
for the indicator function associated to $\Omega_v^+(t)$, assume that
$\partial_t\chi_v = -\nabla\cdot (\chi_vv)$ in a weak sense.

In this setting, we call a map 
$
\vartheta\colon\overline{\Omega}\times [0,T] \to [-1,1]
$
a \emph{transported weight with respect to $(\chi_v,v)$}
if the following conditions are satisfied:
\begin{itemize}[leftmargin=0.7cm]
\item (Regularity) It holds $\vartheta\in W^{1,\infty}_{x,t}(\Omega\times [0,T])$.
\item (Coercivity) Throughout the essential interior of $\Omega_v^+$ (relative to $\Omega$) it holds
			$\vartheta < 0$, throughout the essential exterior of $\Omega_v^+$ (relative to $\Omega$)
			it holds $\vartheta > 0$, and along~$I_v \cup \partial\Omega$ we have $\vartheta = 0$.
			There also exists $C>0$ such that
			\begin{align}
			\label{eq:lowerBoundTransportedWeight}
			\dist(\cdot,\partial\Omega) \wedge \dist(\cdot,I_v) \wedge 1
			\leq C|\vartheta| \quad\text{in } \Omega\times [0,T].
			\end{align}
\item (Transport equation) There exists~$C>0$ such that
			\begin{align}
			\label{eq:advDerivTransportedWeight}
			|\partial_t \vartheta + (v\cdot\nabla)\vartheta| \leq C|\vartheta|
			 \quad\text{in } \Omega\times [0,T].
			\end{align}
\end{itemize}
\end{definition}

The merit of the previous two definitions is now the following result.
It reduces the proof of Theorem~\ref{theo:mainResult}
to the existence of a boundary adapted extension~$\xi$ of the interface unit normal 
and a transported weight~$\vartheta$ with respect to a strong solution~$(\chi_v,v)$, respectively.

\begin{proposition}[Conditional weak-strong uniqueness principle]
\label{prop:conditionalWeakStrong}
Let $d\in\{2,3\}$, and let $\Omega\subset\R^d$ be a bounded domain with orientable and smooth boundary.
Let $(\chi_u,u,V)$ be a varifold solution to the incompressible
Navier--Stokes equation for two fluids in the sense of Definition~\ref{Def_varsol}
on a time interval $[0,T]$. Consider in addition a strong solution $(\chi_v,v)$ to the incompressible
Navier--Stokes equation for two fluids in the sense of Definition~\ref{Def_strongsol}
on a time interval $[0,T]$.

Assume there exists a boundary adapted extension~$\xi$ of the unit normal~$n_{I_v}$
as well as a transported weight~$\vartheta$
with respect to~$(\chi_v,v)$ in the sense of Definition~\ref{def:calibrationTwoPhaseFluidFlow}
and Definition~\ref{def:transportedWeight}, respectively.
Then the stability estimates~\eqref{eq:stabilityRelEntropy} and~\eqref{eq:stabilityBulkError}
for the relative entropy functional~\eqref{relent} and the bulk error
functional~\eqref{eq:bulkErrorFunctional} are satisfied, respectively.
Moreover, if the initial data of the varifold solution and the strong solution
coincide, we may conclude that
\begin{align*}
\chi_u(\cdot,t) &= \chi_v(\cdot,t),\quad
u(\cdot,t) = v(\cdot,t)
&&\text{a.e.\ in } \Omega \text{ for a.e.\ } t\in [0,T],
\\
V_t &= (|\nabla\chi_u(\cdot,t)|\llcorner\Omega) \otimes
\big(\delta_{\frac{\nabla\chi_u(\cdot,t)}{|\nabla\chi_u(\cdot,t)|}(x)}\big)_{x\in\Omega}
&&\text{for a.e.\ } t\in [0,T].
\end{align*}
\end{proposition}	

A proof of this conditional weak-strong uniqueness principle
is presented in Subsection~\ref{subsec:proofCondWeakStrongUniqueness} below.
We emphasize again that it is valid for~$d\in\{2,3\}$. The key ingredient
to the stability estimate~\eqref{eq:stabilityRelEntropy} is the following
relative entropy inequality. We refer to Subsection~\ref{subsec:timeEvolRelEntropy} for a proof.

\begin{proposition}[Relative entropy inequality in case of a $90^\circ$ contact angle] 
\label{proprelent90}
Let $d\in\{2,3\}$, and let $\Omega\subset\R^d$ be a smooth and bounded domain.
Let $(\chi_u,u,V)$ be a varifold solution to the incompressible
Navier--Stokes equation for two fluids in the sense of Definition~\ref{Def_varsol}
on a time interval $[0,T]$. In particular, let~\(\theta\) be the density 
$\theta_t := \frac{\mathrm{d} |\nabla \chi_u (\cdot, t)|\llcorner\Omega}
{ \mathrm{d} |V_t|_{\Ss}\llcorner\Omega}$ 
as defined in~\eqref{thetat}. Furthermore, let $(\chi_v,v)$ be a strong solution 
in the sense of Definition~\ref{Def_strongsol} on the same time interval $[0,T]$, and assume there exists a 
boundary adapted extension~$\xi$ of the interface unit normal~$n_{I_v}$ 
with respect to $(\chi_v,v)$ as in Definition~\ref{def:calibrationTwoPhaseFluidFlow}.
		
Then, the total relative entropy defined by
(recall the definition~\eqref{def:interfaceRelEntropy} of the interface 
contribution $E[\chi_u,V|\chi_v]$)
\begin{align}
\label{relent}
E[\chi_u, u, V|\chi_v, v](t) &:= 
\into \frac{1}{2}\rho(\chi_u(\cdot, t)) |u(\cdot, t)-v(\cdot, t)|^2 \dx  
+ E[\chi_u,V|\chi_v](t)
\end{align} 
satisfies the relative entropy inequality
\begin{align}
		& 	E[\chi_u, u, V|\chi_v, v](T')
		+ \int_{0}^{T'} \into \fracmu |\nabla (u-v) + \nabla (u-v)^\mathsf{T} |^2 \dx \dt  \nonumber \\
		&\leq  	E[\chi_u, u, V|\chi_v, v](0) + R_{dt} +R_{adv} + R_{surTen}, 
		\label{relentineq}
		\end{align} 
		for almost every $T' \in [0,T]$, where we made use of the abbreviations
		(denote by $n_u:=\frac{\nabla\chi_u}{|\nabla\chi_u|}$ the measure-theoretic unit normal)
		\begin{align}
		R_{dt}  = - \int_{0}^{T'} \into (\rho(\chi_v) - \rho(\chi_u)) (u-v) \cdot \pt v  \dx \dt , \nonumber 
		\end{align}
		\begin{align}
		R_{adv} = &- \int_{0}^{T'} \into (\rho(\chi_u) - \rho(\chi_v)) (u-v) \cdot (v \cdot \nabla) v   \dx \dt  \nonumber \\
		&	- \int_{0}^{T'} \into \rho(\chi_u) (u-v) \cdot ((u-v) \cdot \nabla) v   \dx \dt , \nonumber 
		\end{align}
		as well as
		\begin{align*}
		R_{surTen} = 
		& -  \sigma \int_{0}^{T'} \int_{\overline{\Om} \times \Ss} 
		(s - \xi ) \cdot ((s - \xi) \cdot \nabla) v \dV \dt   \\
		&+\sigma \int_{0}^{T'} \into (1 - \theta_t) \xi \cdot  ( \xi \cdot \nabla) v \dVm \dt 	\\
		&+ \sigma \int_{0}^{T'} \int_{\partial\Omega}  \xi \cdot  ( \xi \cdot \nabla) v \dVm \dt \\
		&+ \sigma \int_{0}^{T'} \into (\chi_u -\chi_v) ((u-v)\cdot \nabla ) (\nabla \cdot \xi) \dx \dt  \\
		& - \sigma \int_{0}^{T'} \into (n_u - \xi) \cdot (\pt \xi + (v \cdot \nabla )\xi 
			+ (\mathrm{Id}{-}\xi\otimes\xi)(\nabla v)^\mathsf{T} \xi) \dchiu \dt \\
		& - \sigma \int_0^{T'} \into ((n_u - \xi) \cdot \xi ) (\xi\otimes\xi:\nabla v)\dchiu \dt \\
		& - \sigma \int_{0}^{T'} \into \Big(\pt \frac{1}{2}|\xi|^2 + (v \cdot \nabla )\frac{1}{2}|\xi|^2 \Big) \dchiu \dt \\
		& + \sigma \int_{0}^{T'} \into  (1 - n_u \cdot \xi) (\nabla \cdot v) \dchiu \dt. 
\end{align*}
\end{proposition}

The stability estimate~\eqref{eq:stabilityBulkError} for the bulk error functional 
is in turn based on the following auxiliary result; see Subsection~\ref{subsec:timeEvolBulkError}
for a proof.

\begin{lemma}[Time evolution of the bulk error]
\label{lem:stabilityBulkError}
Let $d\in\{2,3\}$, and let $\Omega\subset\R^d$ be a smooth and bounded domain.
Let $T\in (0,\infty)$ be a finite time horizon, and let $(\chi_v,v)$ be as 
in Definition~\ref{def:transportedWeight} of a transported weight. Let 
$(\chi_u,u,V)$ be a varifold solution to the incompressible
Navier--Stokes equation for two fluids in the sense of Definition~\ref{Def_varsol}
on $[0,T]$. Assume there exists a transported weight $\vartheta$ with 
respect to~$(\chi_v,v)$ in the sense of Definition~\ref{def:transportedWeight},
and define the bulk error functional
\begin{align}
\label{eq:bulkErrorFunctional}
E_{\mathrm{vol}}[\chi_u | \chi_v](t) := 
\int_{\Omega} |\chi_u(\cdot,t) - \chi_v(\cdot,t)| |\vartheta(\cdot,t)| \dx.
\end{align}
Then the following identity holds true
for almost every $T'\in [0,T]$
\begin{align}
\label{eq:stabilityBulkErrorFunctional}
E_{\mathrm{vol}}[\chi_u | \chi_v](T')
&= E_{\mathrm{vol}}[\chi_u | \chi_v](0)
+ \int_0^{T'}\int_\Omega (\chi_u-\chi_v)(\partial_t\vartheta + (v\cdot\nabla)\vartheta) \dx\dt
\\&~~~ \nonumber
+ \int_0^{T'}\int_\Omega (\chi_u-\chi_v)\big((u-v)\cdot\nabla\big)\vartheta \dx\dt.
\end{align}
\end{lemma}
	
\subsection{Existence of boundary adapted extensions of the interface unit normal 
and transported weights in planar case}
\label{subsec:existenceCalibrations}
To upgrade the conditional weak-strong uniqueness principle
of Proposition~\ref{prop:conditionalWeakStrong} to the
statement of Theorem~\ref{theo:mainResult}, it remains to
construct a boundary adapted extension~$\xi$ of~$n_{I_v}$ and a transported weight~$\vartheta$
associated to a given strong solution~$(\chi_v,v)$. In the context of the present
work, we perform this task for simplicity in the planar regime~$d=2$. 
However, it is expected that the principles of the construction
carry over to the case~$d=3$ involving contact lines.

\begin{proposition}
\label{prop:existenceCalibration}
Let $d=2$, and let $\Omega\subset\R^2$ be a bounded domain with orientable and smooth boundary.
Let $(\chi_v,v)$ be a strong solution to the incompressible
Navier--Stokes equation for two fluids in the sense of Definition~\ref{Def_strongsol}
on a time interval $[0,T]$. Then there exists a boundary adapted extension $\xi$
of~$n_{I_v}$ w.r.t.\ $(\chi_v,v)$ in the sense of Definition~\ref{def:calibrationTwoPhaseFluidFlow}.
\end{proposition}

A proof of this result is presented in Subsection~\ref{eq:proofExistenceCalibration} below.
One major step in the proof consists of reducing the global construction to certain
local constructions being supported in the bulk~$\Omega$ or
in the vicinity of contact points along~$\partial\Omega$, respectively.
The main ingredients for this reduction argument are provided in Subsection~\ref{subsec:partitionOfUnity}.
The construction of suitable local vector fields subject to
conditions as in Definition~\ref{def:calibrationTwoPhaseFluidFlow}
is in turn relegated to Section~\ref{sec:localCalibrationBulk}
(bulk construction) and Section~\ref{sec:localCalibrationContactPoint}
(construction near contact points). 
We finally provide the construction of a transported weight
in Section~\ref{sec:constructionWeight}.

\begin{lemma}
\label{lem:existenceTransportedWeight}
Let $d=2$, and let $\Omega\subset\R^2$ be a bounded domain with orientable and smooth boundary.
Let $(\chi_v,v)$ be a strong solution to the incompressible
Navier--Stokes equation for two fluids in the sense of Definition~\ref{Def_strongsol}
on a time interval $[0,T]$. Then there exists a transported weight $\vartheta$
w.r.t.\ $(\chi_v,v)$ in the sense of Definition~\ref{def:transportedWeight}.
\end{lemma}

	
\subsection{Definition of varifold and strong solutions}
In this subsection, we present definitions of strong and 
varifold solutions for the free-boundary problem
of the evolution of two immiscible, incompressible, viscous fluids separated 
by a sharp interface with surface tension inside a bounded domain $\Omega\subset\R^d$, $d\in \{2,3\}$,
with smooth and orientable boundary. Recall in this context that we
restrict ourselves to the case of a~$90^\circ$ contact 
angle between the interface and the boundary of the domain~\(\Om\).
In order to define a notion of strong solutions, we first  
introduce the notion of a smoothly evolving domain
within~$\Omega$.

\begin{definition}[Smoothly evolving domains and smoothly evolving interfaces with $90^\circ$ contact angle]
\label{Def_smdom}
Let $d\in\{2,3\}$, and let $\Omega\subset\R^d$ be a bounded domain with orientable and smooth boundary.
Let $T\in (0,\infty)$ be a finite time horizon. Consider an open subset $\Omega^+_0\subset\Omega$
subject to the following regularity conditions:
\begin{itemize}[leftmargin=0.7cm]
\item Denoting by $I_0$ the closure of $\partial\Omega^+_0\cap\Omega$ in $\overline{\Omega}$,
			we require $I_0$ to be a $(d{-}1)$-dimensional uniform $C^3_x$ submanifold of $\overline{\Omega}$
			with or without boundary. Moreover, $I_0$ is compact and consists of finitely many connected components.
\item Interior points of $I_0$ are contained in $\Omega$, whereas boundary points of $I_0$
			are contained in $\partial\Omega$. In particular, $I_0\cap\partial\Omega$ is a
			$(d{-}2)$-dimensional uniform $C^3_x$ submanifold of $\partial\Omega$.
\item Whenever $I_0$ intersects with $\partial\Omega$, it does so by forming an angle of $90^\circ$.
\end{itemize}
Now, consider a set $\Omega^+=\bigcup_{t\in [0,T]} \Omega^+(t) {\times}\{t\}$
represented in terms of open subsets $\Omega^+(t) \subset \Omega$ for all $t\in [0,T]$.
Denote by $I(t)$ the closure of $\partial\Omega^+(t)\cap\Omega$ in $\overline{\Omega}$,
$t\in [0,T]$. We call $\Omega^+$ a \emph{smoothly evolving domain in $\Omega$}, and 
$I=\bigcup_{t\in [0,T]} I(t){\times}\{t\}$ a \emph{smoothly evolving interface with $90^\circ$ contact angle},
if there exists a flow map $\psi\colon \overline{\Omega} \times [0,T] \to \overline{\Omega}$
such that the following requirements are satisfied:
\begin{itemize}[leftmargin=0.7cm]
\item $\psi(\cdot,0) = \mathrm{Id}$. For any $t\in [0,T]$, the map $\psi_t:=\psi(\cdot,t)\colon \overline{\Omega}
			\to \overline{\Omega}$ is a $C^3_x$ diffeomorphism such that $\psi_t(\Omega)=\Omega$,
			$\psi_t(\partial\Omega)=\partial\Omega$ and 
			$\sup_{t\in[0,T]}\|\psi_t\|_{W^{3,\infty}_x(\overline{\Omega})} < \infty$.
\item For all $t\in [0,T]$, it holds $\Omega^+(t) = \psi_t(\Omega^+_0)$ and
			$I(t) = \psi_t(I_0)$.
\item $\partial_t\psi\in C([0,T];C^1(\overline{\Omega}))$ such that
			$\sup_{t\in[0,T]}\|\partial_t\psi(\cdot,t)\|_{W^{1,\infty}_x(\overline{\Omega})} < \infty$.
\item Whenever $I(t)$, $t\in [0,T]$, intersects $\partial\Omega$ it does so by forming
			an angle of $90^\circ$.
\end{itemize}
\end{definition}

With the geometric setup in place, we can proceed with our notion of
strong solutions to two-phase Navier--Stokes flow with $90^\circ$ contact angle.

\begin{definition}[Strong solution]  
\label{Def_strongsol}
Let $d\in\{2,3\}$, and let $\Omega\subset\R^d$ be a bounded domain with orientable and smooth boundary.
Let a surface tension constant $\sigma > 0$, the densities and shear viscosity 
of the two fluids $\rho^\pm, \mu> 0$, and a finite time $T_s > 0$ be given.
Let $\chi_0$ denote the indicator function of an open subset $\Omega^+_0\subset\Omega$
subject to the conditions of Definition~\ref{Def_smdom}. Denoting the associated
initial interface by~$I_v(0)$, let a solenoidal initial 
velocity profile $v_0 \in L^2(\Om;\R^d)$ be given 
such that it holds $v_0 \in C^2(\overline{\Om}\setminus I_v(0))$.
(Of course, additional compatibility conditions in terms of an initial pressure~$p_0$ 
have to be satisfied by~$v_0$ to allow for the below required regularity of the solution.)

A pair $(\chi_v,v)$ consisting of a velocity field 
$
v\colon \overline{\Omega} \times [0,T_s) \to \R^d
$
and an indicator function
$
\chi_v\colon \overline{\Omega} \times [0,T_s) \to \{0,1\}
$
is called a \emph{strong solution to the free boundary problem for the 
Navier--Stokes equation for two fluids with $90^\circ$ contact angle 
and initial data $(\chi_0,v_0)$} if for all $T\in (0,T_s)$
it is a \emph{strong solution on} $[0,T]$ in the following sense:
\begin{itemize}[leftmargin=0.7cm]
\item It holds
			\begin{align}
			&v \in 
			W^{1,\infty}([0, T]; W^{1,\infty}(\Om; \R^d)), \nonumber \\
			&\nabla v \in L^{1}([0, T]; \BV(\Om; \R^{d \times d})) , \nonumber \\
			&\chi_v \in L^\infty ([0, T]; \BV(\Om; \{0,1\} )).\nonumber 
			\end{align}
\item Define $\Omega^+_v(t) := \{x \in \Om: \chi_v(x,t) = 1\}$. Then, 
			$\Omega^+_v=\bigcup_{t\in[0,T]} \Omega^+_v(t) {\times} \{t\}$ is a smoothly evolving domain in $\Omega$
			in the sense of Definition~\ref{Def_smdom} with $\Omega^+_v(0)=\Omega^+_0$.
			Denoting by $I_v(t)$ the closure of $\partial\Omega^+_v(t)\cap\Omega$
			in $\overline{\Omega}$ for all $t\in [0,T]$, the set $I_v=\bigcup_{t\in [0,T]} I_v(t){\times}\{t\}$
			is a smoothly evolving interface with $90^\circ$ contact angle in the sense of Definition~\ref{Def_smdom}.
			In particular, for every $t\in [0,T]$ and every contact point $c(t)\in I_v(t)\cap\partial\Omega$
			\begin{align}
			\label{sangcond}
   		n_{\partial \Omega} (c(t))\cdot n_{I_v}(c(t),t) = 0.
   		\end{align}
			Moreover, for every $t\in [0,T]$ and every $c(t)\in I_v(t)\cap\partial\Omega$
			the following higher-order compatibility condition is required to hold:
			\begin{align}
			\label{eq:higherOrderCompContactPoints}
   			-\big((n_{\partial \Omega}\cdot\nabla) (n_{I_v}\cdot v)\big)(c(t),t) 
				= H_{\partial \Omega}(c(t)) (n_{I_v}\cdot v) (c(t),t), 
   		\end{align}
   		where~$H_{\partial \Omega}$ denotes the scalar mean curvature of~\(\partial \Omega\)
			(with respect to the inward pointing unit normal~$n_{\partial\Omega}$).
\item The velocity field $v$ has vanishing divergence $\nabla \cdot  v = 0$,
			and it satisfies the boundary conditions
			\begin{align}
			\label{eq:boundaryCondFluid1}
			v(\cdot,t) \cdot \no  &= 0
			&&\text{along } \partial\Omega,
			\\
			\label{eq:boundaryCondFluid2}
			\big(\no \cdot \mu (\nabla v + \nabla v ^\mathsf{T} )(\cdot,t)B\big) &= 0
			&& \text{along } \partial\Omega
			\end{align}
			for all $t \in [0,T]$ and all tangential vector fields~$B$ along~$\partial\Omega$.
			Moreover, the equation for the momentum balance
   		\begin{align} 
			\nonumber 
   		\into &\rho(\chi_v(\cdot, T')) v(\cdot, T') \cdot \eta (\cdot, T') \dx 
			- \into \rho(\chi_0)) v_0 \cdot \eta (\cdot, 0) \dx
			\\ \label{sbalmom}
   		=& \int_{0}^{T'} \into \rho(\chi_v) v \cdot \partial_t \eta \dx \dt 
			+ \int_{0}^{T'} \into \rho(\chi_v) v \otimes v : \nabla \eta \dx \dt  
			\\ \nonumber
   		&- \int_{0}^{T'} \into \mu (\nabla v + \nabla v^\mathsf{T}) : \nabla \eta \dx \dt 
			+ \sigma \int_{0}^{T'} \int_{I_v(t)} \mathrm{H}_{I_v} \cdot \eta \dS \dt 
   		\end{align}
   		holds true for almost every $T' \in  [0,T]$ and	every
			$\eta \in C^\infty(\overline{\Omega} \times [0,T];\R^d)$ such 
			that $\nabla \cdot \eta = 0$ as well as $(\eta \cdot \no )|_{\partial \Om} = 0$. 
			Here, $\mathrm{H}_{I_v}(\cdot,t)$ denotes 
			the mean curvature vector of the interface $I_v(t)$. 
			For the sake of brevity, we have used the abbreviation $\rho(\chi) := \rho^+\chi + \rho^-(1 - \chi)$.
\item The indicator function $\chi_v$ is transported by the fluid velocity $v$ in form of
   		\begin{equation} \label{straneq}
   		\into \chi_v(\cdot, T') \varphi(\cdot, T') \dx 
			- \into \chi_0 \varphi(\cdot, 0) \dx 
			= \int_{0}^{T'} \into \chi_v (\pt \varphi {+} (v \cdot \nabla) \varphi ) \dx \dt 
   		\end{equation}
   		for almost every \(T' \in  [0, T]\) and all 
			$\varphi \in C^\infty(\overline{\Omega} \times [0, T])$.
\item It holds $v\in C^1_tC^0_x(\overline{\Omega{\times}[0,T]}\setminus I_v) 
			\cap C^0_tC^2_x(\overline{\Omega{\times}[0,T]}\setminus I_v)$.
\end{itemize}
\end{definition}

We conclude the discussion on strong solutions with a series of remarks.
First, by standard arguments one may deduce from~\eqref{straneq},
the solenoidality of~$v$, and the boundary 
condition~$(v \cdot \no )|_{\partial \Om} = 0$
that $V_{I_v}=v\cdot n_{I_v}$ holds true along the interface~$I_v$
for the normal speed~$V_{I_v}$ of~$I_v$ (oriented with respect to~$n_{I_v}$).
Second, as a consequence of the contact point condition~\eqref{sangcond} it holds
for all $t\in [0,T_s)$
\begin{align*}
\int_{I_v(t)} \mathrm{H}_{I_v} \cdot \eta \dS
= - \int_{I_v(t)} \big(\mathrm{Id}{-}n_{I_v}(\cdot,t)
\otimes n_{I_v}(\cdot,t)\big):\nabla\eta \dS
\end{align*}
for all test fields $\eta \in C^\infty(\overline{\Om};\R^d)$ 
subject to $\nabla \cdot \eta = 0$ and $(\eta \cdot \no )|_{\partial \Om} = 0$. 
Third, note that Definition~\ref{Def_strongsol} implies that
all pairs of two distinct contact points at the initial time remain distinct
at all later times within a finite time horizon. This in fact is a consequence
of the regularity of the velocity field and the evolving interface.
Indeed, denoting by $t\mapsto c(t)\in I_v(t)\cap\partial\Omega$
resp.\ $t\mapsto \hat c(t)\in I_v(t)\cap\partial\Omega$ the trajectories
of two distinct contact points, we may estimate the time evolution
of their squared distance $\alpha(t):= \frac{1}{2}|c(t){-}\hat c(t)|^2$ by means of
\begin{align*}
\frac{\mathrm{d}}{\mathrm{d}t} \alpha(t)
= \big(c(t){-}\hat c(t)\big) \cdot \big(v(c(t),t){-}v(\hat c(t),t)\big)
\geq -2\|\nabla v\|_{L^\infty_{x,t}} \alpha(t).
\end{align*}
Using Gronwall's Lemma, we can conclude that
$\alpha(t) \geq \alpha(0) \exp(- 2\|\nabla v\|_{L^\infty_{x,t}}t)$.

Fourth, we remark that it actually suffices to require the compatibility 
conditions~\eqref{sangcond} and~\eqref{eq:higherOrderCompContactPoints}
at the initial time~$t=0$ only. For later times~$t\in (0,T]$, they
are in fact consequences of the regularity of a strong solution, which
can be seen as follows. For the sake fo simplicity, consider the case $d=2$. By means of the chain rule, the fact that $v \cdot \no=0$ along $\partial \Om$, 
and the formulas for~$\nabla \no$ and~$\nabla \tau_{\partial\Omega}$ from Lemma~\ref{lem:gradientsFrames},
we may rewrite the boundary condition $(\mu (\nabla v + \nabla v ^\mathsf{T} ) : 
\no \otimes \tau_{\partial \Om})=0$ along $\partial \Om$ as
\[H_{\partial\Omega} (v \cdot \tau_{\partial \Omega} )+ (\no \cdot \nabla )
(v \cdot \tau_{\partial\Omega} ) =0 \quad \text{along } \partial \Om,
\]
which holds in particular at a contact point \(c(t)\) for any $t \in [0,T]$.
Then, since the quantities
$
|\tau_{\partial \Omega}  \cdot \tau_{I_v}  |=|n_{I_v}  \cdot n_{\partial \Omega}  |
$,
$|\tau_{\partial \Omega}  - n_{I_v}  |$ ,
$|n_{\partial \Omega}  + \tau_{I_v}  | $ evaluated at a contact point can all 
be bounded from above by $\sqrt{1-n_{I_v}  \cdot \tau_{\partial \Omega} }$, 
we may compute by adding zeros (see also the formulas for~$\nabla \no$ 
and~$\nabla \tau_{\partial\Omega}$ as well as the expressions 
for~$\frac{\mathrm{d}}{\mathrm{d}t}  \tau_{\partial\Omega}(c(t))$
and~$\frac{\mathrm{d}}{\mathrm{d}t} n_{I_v}(c(t),t)$ from Lemma~\ref{lem:gradientsFrames} 
and Lemma~\ref{lem:higherOrderCompConditions}, respectively)
\begin{align*}
&\frac{\mathrm{d}}{\mathrm{d}t}  [1- n_{I_v}(c(t),t) \cdot \tau_{\partial \Omega}(c(t)) ] \\
&\quad = - \big((n_{I_v} \cdot n_{\partial \Omega})
( (\no \cdot \nabla )(v \cdot \tau_{\partial \Omega}) + ( \tau_{I_v} \cdot \nabla )(v \cdot n_{I_v})
) \big)\big|_{(c(t),t)} \\
&\quad = - \big((n_{I_v} \cdot n_{\partial \Omega})
( \nabla v :  (\tau_{\partial \Omega} -  n_{I_v}) \otimes \no
+  \nabla v : n_{I_v} \otimes (\no + \tau_{I_v} ) \\ 
&\quad \quad \quad 
- H_{I_v}(v \cdot \tau_{I_v}) (\tau_{\partial \Om} \cdot \tau_{I_v}))
\big) \big|_{(c(t),t)} \\
&\quad \leq C \|\nabla v\|_{L^\infty_{x,t}} [1- n_{I_v}(c(t),t) \cdot \tau_{\partial \Omega}(c(t)) ] 
\end{align*}	
for some $C>0$ and any $t \in [0,T]$. From an application of a Gronwall-type argument 
and the validity of the contact angle condition~\eqref{sangcond} at the initial time $t=0$, 
we may conclude that~\eqref{sangcond} is indeed satisfied for any $t \in [0,T]$. 
The compatibility condition~\eqref{eq:higherOrderCompContactPoints}
in turn follows from differentiating in time the angle condition~\eqref{sangcond} along
a smooth trajectory $t\mapsto c(t)\in I_v(t)\cap\partial\Omega$ of a contact point,
see for details the proof of Lemma~\ref{lem:higherOrderCompConditions}.

We proceed with the notion of a varifold solution.

\begin{definition}[Varifold solution in case of $90^\circ$ contact angle condition]  \label{Def_varsol}
Let a surface tension constant $\sigma > 0$, the densities and shear viscosity of the 
two fluids $\rho^\pm, \mu > 0$, a finite time $T_w > 0$, a solenoidal initial velocity 
profile $u_0 \in L^2(\Om;\R^d)$, and an indicator function~$\chi_0 \in \BV(\Om)$ be given.
   	
A triple $(\chi_u,u, V)$ consisting of a velocity field \(u\), an indicator function \(\chi_u\), 
and an oriented varifold \(V\) with
	\begin{align}
	&u \in L^{2}([0, T_w]; H^1(\Om; \R^d)) \cap L^{\infty}([0, T_w]; L^{2}(\Om; \R^d)), \nonumber \\
	&\chi_u \in L^\infty ([0, T_w]; \BV(\Om; \{0,1\} )) ,\nonumber \\
	& V \in L_w^\infty ([0, T_w]; \mathcal{M}(\overline{\Om} \times \Ss)), \nonumber 
	\end{align}
is called a \emph{varifold solution to the free boundary problem for the Navier-Stokes 
equation for two fluids with $90^\circ$ contact angle and initial data $(\chi_0,u_0)$} if the following conditions are satisfied:
\begin{itemize}[leftmargin=0.7cm]
		\item The velocity field $u$ has vanishing divergence $\nabla \cdot  u = 0$, 
		its trace a vanishing normal component on the boundary of the domain $(u \cdot \no) |_{\partial \Om} = 0$, 
		and the equation for the momentum balance 
		\begin{align}\label{vbalmom}
		\into &\rho(\chi_u(\cdot, T)) u(\cdot, T) \cdot \eta (\cdot, T) \dx 
		- \into \rho(\chi_0)) u_0 \cdot \eta (\cdot, 0) \dx \nonumber \\
		=& \int_{0}^T \into \rho(\chi_u) u \cdot \partial_t \eta \dx \dt 
		+ \int_{0}^T \into \rho(\chi_u) u \otimes u: \nabla \eta \dx \dt  \\
		&- \int_{0}^T \into \mu (\nabla u+ \nabla u^\mathsf{T}) : \nabla \eta \dx \dt \nonumber \\
		&- \sigma \int_{0}^T \int_{\overline{\Om} \times \Ss} (\Id - s \otimes s) : \nabla \eta \dV \dt \nonumber
		\end{align}
		is satisfied for almost every $T \in  [0,T_w)$ and for every test vector field~$\eta$
		subject to $\eta \in C^\infty([0,T_w );C^1(\overline{\Omega};\R^d)\cap\bigcap_{p\geq 2}W^{2,p}(\Omega;\R^d))$,
		$\nabla \cdot \eta = 0$ as well as $(\eta \cdot \no) |_{\partial \Om} = 0$.
		We again made use of the 
		abbreviation $\rho(\chi) := \rho^+\chi + \rho^-(1 - \chi)$.
		\item The indicator $\chi_u$ satisfies the weak formulation of the transport equation
		\begin{equation}\label{vtraneq}
		\into \chi_u(\cdot, T) \varphi(\cdot, T) \dx - \into \chi_0 \varphi(\cdot, 0) \dx 
		= \int_{0}^T \into \chi_u (\pt \varphi {+} (u \cdot \nabla) \varphi ) \dx \dt 
		\end{equation}
		for almost every \(T \in  [0, T_w)\) and all 
		$\varphi \in C^\infty(\overline{\Omega} \times [0, T_w))$.
		\item The energy dissipation inequality 
		\begin{align} \label{endisineq}
		&\into \fract \rho(\chi_u(\cdot, T)) |u(\cdot, T)|^2 \dx + \sigma |V_T|_{\Ss}(\overline{\Om}) 
		+ \int_{0}^T \into \fracmu |\nabla u {+} \nabla u^\mathsf{T}|^2 \dx \dt \nonumber \\
		&\leq \into \fract \rho(\chi_0(\cdot )) |u_0(\cdot)|^2\dx + \sigma |\nabla \chi_0|(\Om) 
		\end{align}
		is satisfied for almost every \(T \in  [0, T_w)\).
		\item The phase boundary $\partial^* \{\chi_u(\cdot ,t) = 0\}\cap\Omega$ 
					and the varifold $V_t$ satisfy the compatibility condition 
		\begin{equation} \label{compcond}
		\int_{\overline{\Om} \times \Ss} \psi (x) \cdot s \dV 
		= \into \psi(x) \cdot \,\mathrm{d} \nabla \chi_u (x,t)
		\end{equation}
		for almost every \(t \in  [0, T_w)\) and every smooth function 
		\(\psi \in C^\infty(\overline{\Om};\R^d) \)
		such that $(\psi \cdot \no) |_{\partial \Om} = 0$.
	\end{itemize}
Finally, if $(\chi_u,V)$ satisfy~\eqref{eq:varifoldIsBV} we call
the pair~$(\chi_u,u)$ a \emph{$\BV$~solution to the free boundary problem for the Navier-Stokes 
equation for two fluids with $90^\circ$ contact angle and initial data $(\chi_0,u_0)$}.
\end{definition}

We conclude with a remark concerning the notion of
varifold solutions. Denote by $V_t \in \mathcal{M} (\overline{\Om} \times \Ss)$ the non-negative measure 
representing at time $t\in [0,T_w)$ the varifold associated to a varifold solution 
$(\chi_u, u, V )$. The compatibility condition~\eqref{compcond} entails 
that $|\nabla \chi_u(\cdot,t)|\llcorner\Omega$ is absolutely continuous with respect 
to $|V_t|_{\Ss}\llcorner\Omega$; in fact, 
$|\nabla \chi_u(\cdot,t)|\llcorner\Omega \leq |V_t|_{\Ss}\llcorner\Omega$ 
in the sense of measures on~$\Omega$. Hence, we may define the Radon--Nikodym derivative
	\begin{equation} 
	\theta_t := \frac{\mathrm{d} |\nabla \chi_u (\cdot, t)|\llcorner \Omega}
	{ \mathrm{d} |V_t|_{\Ss}\llcorner\Omega},
	\label{thetat}
	\end{equation}
which is a $(|V_t|_{\Ss}\llcorner\Omega)$-measurable function with $|\theta_t| \leq 1$ valid
$(|V_t|_{\Ss}\llcorner\Omega)$-almost everywhere in~$\Omega$. In other words, the quantity~$\frac{1}{\theta_t}$ 
represents the multiplicity of the varifold (in the interior). With this notation in place, it then holds
	\begin{equation} \label{compcond2}
	\into f(x) \, \mathrm{d} |\nabla \chi_u (\cdot, t)|(x)
	= \into \theta_t (x) f(x) \, \mathrm{d} |V_t|_{\Ss}(x)
	\end{equation}
for every $f \in L^1(\Om, |\nabla \chi_u(\cdot , t)|)$ and almost every $t \in [0, T_w)$.

\subsection{Notation}
Throughout the present work, we employ the notational
conventions of~\cite{Fischer2020}. A notable addition
is the following convention. If $D \subset \R^d$
is an open subset and $\Gamma \subset D$ a closed
subset of Hausdorff-dimension $k \in \{0,\ldots,d{-}1\}$,
we write $C^k(\overline{D}\setminus\Gamma)$
for all maps $f\colon D\to\R$ which
are $k$-times continuously differentiable
throughout $D\setminus\Gamma$ such that
the function together with all its derivatives
stays bounded throughout~$D\setminus\Gamma$.
Analogously, one defines the space $C^k_tC^m_x(\overline{D}\setminus\Gamma)$
for $D=\bigcup_{t\in[0,T]} D(t) \times \{t\}$
and $\Gamma =\bigcup_{t\in[0,T]} \Gamma(t) \times \{t\}$,
where $(D(t))_{t\in[0,T]}$ is a family of open subsets of~$\R^d$
and $(\Gamma(t))_{t\in[0,T]}$ is a family of closed
subsets $\Gamma(t)\subset D(t)$ of constant Hausdorff-dimension
$k \in \{0,\ldots,d{-}1\}$.
%
	

\section{Proof of main results}
	
\subsection{Relative entropy inequality: Proof of Proposition~\ref{proprelent90}}
\label{subsec:timeEvolRelEntropy}
The general structure of the proof is in parts
similar to the proof of~\cite[Proposition~10]{Fischer2020}. 
In what follows, we thus mainly focus on how to exploit
the boundary conditions for the velocity fields~$(u,v)$ and
a boundary adapted extension~$\xi$ of the strong interface unit normal in these computations.

\textit{Step 1:} Since~$\rho(\chi_v)$ is an affine function of~$\chi_v$,
it consequently satisfies
\begin{equation} \label{step1}
\into \rho(\chi_v(\cdot, T')) \varphi(\cdot, T') \dx 
- \into \rho(\chi^0_v) \varphi(\cdot, 0) \dx 
= \int_{0}^T \into \rho(\chi_v) (\pt \varphi + (v \cdot \nabla) \varphi ) \dx \dt 
\end{equation}
for almost every \(T' \in  [0, T]\) and all $\varphi \in \C^\infty(\overline{\Omega} \times [0, T])$. 
By the regularity of~$v$ and an approximation argument, we may test this equation with 
$v \cdot \eta$ for any $\eta \in C^\infty(\overline{\Omega}\times [0,T];\R^d)$, 
yielding
\begin{align} 
\label{step2} 
\nonumber
&\into \rho(\chi_v(\cdot, T')) v(\cdot, T') \cdot \eta(\cdot, T')  \dx 
- \into \rho(\chi^0_v) v(\cdot, 0) \cdot \eta(\cdot, 0)\dx 
\\
&=\int_{0}^{T'}\into \rho(\chi_v) (v \cdot \pt \eta + \eta \cdot \pt v)  \dx \dt 
\\&~~~\nonumber
+ \int_{0}^{T'} \into \rho(\chi_v)(\eta \cdot (v \cdot \nabla) v + v \cdot  (v \cdot \nabla) \eta  ) \dx \dt   
\end{align}
for almost every \(T' \in  [0, T]\). Next, we subtract from~\eqref{step2} 
the equation for the momentum balance~\eqref{sbalmom} of the strong solution.
It follows that the velocity field $v$ of the strong solution satisfies
\begin{align} \label{step3}
0 =& \int_{0}^{T'} \into \rho(\chi_v)  \eta \cdot \pt v  \dx \dt + 
\int_{0}^{T'} \into \rho(\chi_v)\eta \cdot (v \cdot \nabla) v   \dx \dt
\\ \nonumber 
&+ \int_{0}^{T'} \into \mu (\nabla v + \nabla v^T) : \nabla \eta \dx \dt 
- \sigma \int_{0}^{T'} \int_{I_v(t)} \h_{I_v} \cdot \, \eta \dS \dt 
\end{align}
for almost every $T' \in  [0,T]$ and every test vector field 
$\eta \in C^\infty(\overline{\Omega}\times [0,T];\R^d)$
such that $\nabla \cdot \eta = 0$ 
and $(\eta \cdot \no) |_{\partial \Om} = 0$. 
For any such test vector field~$\eta$, note that by means of~\eqref{eq:divConstraintXi}, 
the incompressibility of~$\eta$ as well as $(\eta \cdot \no)|_{\partial \Om} =0$,
we may rewrite
\begin{align}
- \sigma \int_{0}^{T'} \int_{I_v(t)} \h_{I_v} \cdot \, \eta \dS \dt 
&= \sigma \int_{0}^{T'} \int_{I_v(t)} (\nabla \cdot \xi) \eta \cdot n_{I_v} \dS \dt  \nonumber \\
&= - \sigma \int_{0}^{T'} \into \chi_v (\eta\cdot \nabla ) (\nabla \cdot \xi) \dx \dt. \label{surtenAux1}
\end{align}
Hence, we deduce from inserting~\eqref{surtenAux1} back into~\eqref{step3} that
\begin{align} \label{step3Aux}
0 =& \int_{0}^{T'} \into \rho(\chi_v)  \eta \cdot \pt v  \dx \dt + 
\int_{0}^{T'} \into \rho(\chi_v)\eta \cdot (v \cdot \nabla) v   \dx \dt
\\ \nonumber 
&+ \int_{0}^{T'} \into \mu (\nabla v + \nabla v^T) : \nabla \eta \dx \dt 
- \sigma \int_{0}^{T'} \into \chi_v (\eta\cdot \nabla ) (\nabla \cdot \xi) \dx \dt 
\end{align}
for almost every $T' \in  [0,T]$ and every test vector field 
$\eta \in C^\infty(\overline{\Omega}\times [0,T];\R^d)$ 
such that $\nabla \cdot \eta = 0$ 
and $(\eta \cdot \no) |_{\partial \Om} = 0$. The merit
of rewriting~\eqref{step3} into the form~\eqref{step3Aux}
consists of the following observation. Consider a test vector
field $\eta \in C^\infty([0,T];H^1(\Omega;\R^d))$ 
such that $\nabla \cdot \eta = 0$ and $(\eta \cdot \no) |_{\partial \Om} = 0$.
Denoting by~$\psi$ a standard mollifier, for every~$k\in\mathbb{N}$ by $\psi_k:=k^d\psi(k\cdot)$
its usual rescaling, and by~$P_{\Omega}$ the Helmholtz projection associated 
with the smooth domain~$\Omega$, it follows from standard
theory (e.g., by a combination of~\cite{Simader1992} and
standard $W^{m,2}(\Omega)$-elliptic regularity theory -- see also Appendix \ref{appendix}) that $\eta_k := P_{\Omega}(\psi_k \ast \eta)$
is an admissible test vector field for~\eqref{step3Aux}. Moreover,
taking the limit~$k\to\infty$ in~\eqref{step3Aux} with~$\eta_k$
as test vector fields is admissible and results in
\begin{align} \label{step3Aux2}
0 =& \int_{0}^{T'} \into \rho(\chi_v)  \eta \cdot \pt v  \dx \dt + 
\int_{0}^{T'} \into \rho(\chi_v)\eta \cdot (v \cdot \nabla) v   \dx \dt
\\ \nonumber 
&+ \int_{0}^{T'} \into \mu (\nabla v + \nabla v^T) : \nabla \eta \dx \dt 
- \sigma \int_{0}^{T'} \into \chi_v (\eta\cdot \nabla ) (\nabla \cdot \xi) \dx \dt 
\end{align}
for almost every $T' \in  [0,T]$ and every test vector field 
$\eta \in C^\infty([0,T];H^1(\Omega;\R^d))$ 
such that $\nabla \cdot \eta = 0$ and $(\eta \cdot \no) |_{\partial \Om} = 0$.
As an important consequence, because of the boundary condition
for the velocity fields~$(u,v)$ and their solenoidality, 
we may choose (after performing a mollification argument in the time variable) 
$\eta=u-v$ as a test function in~\eqref{step3Aux2} which entails
for almost every $T' \in  [0,T]$
\begin{align} \label{1}
0 =& \int_{0}^{T'} \into \rho(\chi_v)  (u-v) \cdot \pt v  \dx \dt + 
\int_{0}^{T'} \into \rho(\chi_v) (u-v) \cdot (v \cdot \nabla) v   \dx \dt 
\\ \nonumber 
&+ \int_{0}^{T'} \into \mu (\nabla v {+} \nabla v^\mathsf{T}) : \nabla (u{-}v) \dx \dt 
- \sigma \int_{0}^{T'} \into \chi_v ((u{-}v)\cdot \nabla ) (\nabla \cdot \xi) \dx \dt.
\end{align}
We proceed by testing the analogue of~\eqref{step1} for the phase-dependent 
density~$\rho(\chi_u)$ with the test function~$\frac{1}{2}|v|^2$, obtaining
for almost every $T' \in  [0,T]$
\begin{align} \label{2}
&\into  \frac{1}{2}\rho(\chi_u(\cdot, T')) |v(\cdot, T')|^2 \dx 
- \into \frac{1}{2} \rho(\chi^0_u) |v_0(\cdot)|^2 \dx  \nonumber \\
&=\int_{0}^{T'} \into \rho(\chi_u) v \cdot \pt v  \dx \dt 
+ \int_{0}^{T'} \into \rho(\chi_u)  v \cdot (u \cdot \nabla) v \dx \dt. 
\end{align}
We next want to test~\eqref{vbalmom} with 
the fluid velocity~$v$. Modulo a mollification argument in the time variable, we have
to argue that~$\nabla v$ does not jump across the interface so that~$v$ is an admissible
test function. Indeed, since the tangential derivative~$(\tau_{I_v}\cdot\nabla)v$
is continuous across the interface it follows from $\nabla\cdot v=0$ that also $n_{I_v}\cdot(n_{I_v}\cdot\nabla)v$
does not jump across~$I_v$. The only component which may jump is thus~$\tau_{I_v}\cdot(n_{I_v}\cdot\nabla)v$.
However, this is ruled out by the equilibrium condition for the stresses along~$I_v$ 
together with having~$\mu_+=\mu_-$. In summary, using~$v$ in~\eqref{vbalmom} implies
\begin{align}\label{4}
\nonumber
-&\into \rho(\chi_u(\cdot, T')) u(\cdot, T') \cdot v (\cdot, T') \dx 
+ \into \rho(\chi_u^0)) u_0 \cdot v_0 (\cdot) \dx   \\
-& \int_{0}^{T'} \into \mu (\nabla u + \nabla u^\mathsf{T}) : \nabla v \dx \dt \nonumber \\
= &-\int_{0}^{T'} \into \rho(\chi_u) u \cdot \partial_t v \dx \dt 
- \int_{0}^{T'} \into \rho(\chi_u) u \cdot ( u \cdot \nabla)v \dx \dt  \\
&+ \sigma \int_{0}^{T'} \int_{\overline{\Om} \times \Ss} (\Id - s \otimes s) : \nabla v \dV \dt \nonumber
\end{align}
for almost every $T' \in  [0,T]$.
We finally use $\sigma (\nabla \cdot  \xi) $ as a test function
in the transport equation \eqref{vtraneq} for the indicator function $\chi_u$ 
of the varifold solution. Hence, we obtain
\begin{align*} 
&\sigma \into \chi_u(\cdot, T') (\nabla \cdot  \xi)(\cdot, T') \dx 
- \into \chi^0_u (\nabla \cdot  \xi)(\cdot, 0) \dx \\
&= \sigma  \int_{0}^{T'} \into \chi_u (\nabla \cdot  \pt \xi + (u \cdot \nabla) (\nabla \cdot  \xi) ) \dx \dt.
\end{align*}
for almost every $T' \in  [0,T]$. Based on the boundary condition~\eqref{eq:boundaryValueXi},
which in turn in particular implies $(\pt \xi \cdot \no ) |_{ \partial \Om}
=\pt (\xi \cdot \no)|_{ \partial \Om}=0$, we may integrate by parts
to upgrade the previous display to
\begin{align} \label{5}
&- \sigma \into n_u(\cdot, T') \cdot  \xi (\cdot, T') \dchiut  + \into n^0_u \cdot  \xi(\cdot, 0) \dchiuo \nonumber \\ 
&=- \sigma \int_{0}^{T'}  \into n_u \cdot  \pt \xi  \dchiu \dt   +
\sigma  \int_{0}^{T'} \into \chi_u  (u \cdot \nabla) (\nabla \cdot  \xi)  \dx \dt 
\end{align}
for almost every $T' \in  [0,T]$. 

\textit{Step 2:} Summing~\eqref{1}, \eqref{2}, \eqref{endisineq} as well as~\eqref{4}, we obtain
\begin{align}\label{sum}
&LHS_{kin}(T') + LHS_{visc} + LHS_{surEn}(T') \nonumber \\
&\leq RHS_{kin}(0) + RHS_{surEn}(0) + RHS_{dt} + RHS_{adv} + RHS_{surTen},
\end{align}
where the individual terms are given by (cf.\ the proof of~\cite[Proposition~10]{Fischer2020}) 
\begin{align}
\label{RLHSkin}
LHS_{kin}(T') &:= \into \frac{1}{2}\rho(\chi_u(\cdot, T')) |u{-}v|^2(\cdot, T') \dx,\,
\\
\label{RLHSkin2}
RHS_{kin}(0) &:= \into \fract \rho(\chi^0_u) |u_0-v_0|^2\dx,
\\
\label{LHSsuren}
LHS_{surEn}(T') &:=	\sigma |\nabla \chi_u(\cdot, T')|(\Om) 
+ \sigma \into (1- \theta_{T'}) \, \mathrm{d} |V_{T'}|_{\Ss}(x),
\\
\label{RHSsuren}
RHS_{surEn}(0) &:= \sigma |\nabla \chi^0_u(\cdot)|(\Om),
\\
\label{LHSvisc}
LHS_{visc} &:= \int_{0}^{T'} \into \fracmu |\nabla (u-v) + \nabla (u-v)^\mathsf{T} |^2 \dx \dt,
\\
\label{RHSdt}
RHS_{dt} &:=- \int_{0}^{T'}  \into (\rho(\chi_v) - \rho(\chi_u)) (u-v) \cdot \pt v  \dx \dt,
\\
\label{RHSadv}
RHS_{adv} &:= - \int_{0}^{T'}  \into (\rho(\chi_u) - \rho(\chi_v)) (u-v) \cdot (v \cdot \nabla) v   \dx \dt 
\\&~~~~ \nonumber
- \int_{0}^{T'}  \into \rho(\chi_u) (u-v) \cdot ((u-v) \cdot \nabla) v   \dx \dt,
\\
\label{surten.0}
RHS_{surTen} &:= - \sigma \int_{0}^{T'} \into \chi_v ((u{-}v)\cdot \nabla ) 
(\nabla \cdot \xi) \dx \dt
\\&~~~~\nonumber
+ \sigma \int_{0}^{T'}  \int_{\overline{\Om} \times \Ss} (\Id - s \otimes s) : \nabla v \dV \dt. 
\end{align} 
Adding zeros, $\nabla \cdot v =0$, the boundary condition
$n_{\partial\Omega} \cdot (\nabla v {+} (\nabla v)^\mathsf{T})\xi = 
n_{\partial\Omega} \cdot (\nabla v {+} (\nabla v)^\mathsf{T})
(\mathrm{Id} - n_{\partial\Omega}\otimes n_{\partial\Omega})\xi = 0$
due to~\eqref{eq:boundaryCondFluid2} and~\eqref{eq:boundaryValueXi}, 
and the compatibility condition~\eqref{compcond}
allow to rewrite the second term of~\eqref{surten.0} as follows
\begin{align} \label{surten2.0}
\nonumber
\sigma& \int_{0}^{T'} \int_{\overline{\Om} \times \Ss} (\Id - s \otimes s) : \nabla v \dV \dt 
\\
= &  -  \sigma \int_{0}^{T'} \int_{\overline{\Om} \times \Ss} (s - \xi ) \cdot ((s - \xi) \cdot \nabla) v \dV \dt \nonumber \\
&  -  \sigma \int_{0}^{T'} \int_{\overline{\Om} \times \Ss}  
s \cdot (\nabla v + (\nabla v)^\mathsf{T})\xi \dV \dt 
\nonumber \\
&+ \sigma \int_{0}^{T'} \int_{\overline{\Om} \times \Ss}  \xi \cdot  ( \xi \cdot \nabla) v \dV \dt
\nonumber \\
= &   -  \sigma \int_{0}^{T'} \int_{\overline{\Om} \times \Ss} (s - \xi ) \cdot ((s - \xi) \cdot \nabla) v \dV \dt \\
& -  \sigma \int_{0}^{T'} \into \xi \cdot  (n_u \cdot \nabla) v \dchiu \dt
-  \sigma \int_{0}^{T'} \into  n_u \cdot (\xi \cdot \nabla) v \dchiu \dt 
\nonumber \\
&+ \sigma \int_{0}^{T'} \int_{\overline{\Omega}}  \xi \cdot  ( \xi \cdot \nabla) v \dVm \dt. 	\nonumber 
\end{align}
Furthermore, because of~\eqref{compcond2} we obtain 
\begin{align}  \label{surten2.1}
&\sigma \int_{0}^{T'} \int_{\overline{\Omega}}  \xi \cdot  ( \xi \cdot \nabla) v \dVm \dt  \\
&=\sigma \int_{0}^{T'} \into (1 {-} \theta_t) \xi \cdot  ( \xi \cdot \nabla) v \dVm \dt 
+ \sigma \int_{0}^{T'} \into \theta_t \xi \cdot  ( \xi \cdot \nabla) v \dVm \dt   \nonumber
\\&~~~ \nonumber
+ \sigma \int_{0}^{T'} \int_{\partial\Omega}  \xi \cdot  ( \xi \cdot \nabla) v \dVm \dt 
\\
&=\sigma \int_{0}^{T'} \into (1 - \theta_t) \xi \cdot  ( \xi \cdot \nabla) v \dVm \dt 
+ \sigma \int_{0}^{T'} \into \xi \cdot  ( \xi \cdot \nabla) v \dchiu \dt   \nonumber 
\\&~~~ \nonumber
+ \sigma \int_{0}^{T'} \int_{\partial\Omega}  \xi \cdot  ( \xi \cdot \nabla) v \dVm \dt.
\end{align}
The combination of~\eqref{surten.0}, \eqref{surten2.0} and~\eqref{surten2.1}
together with~$\nabla\cdot v=0$ then implies 
\begin{align} 	\label{RHSsurten}       
RHS_{surTen} =& 
-  \sigma \int_{0}^{T'} \int_{\overline{\Om} \times \Ss} (s - \xi ) \cdot ((s - \xi) \cdot \nabla) v \dV \dt \\
&+\sigma \int_{0}^{T'} \into (1 - \theta_t) \xi \cdot  ( \xi \cdot \nabla) v \dVm \dt \nonumber \\
&+ \sigma \int_{0}^{T'} \int_{\partial\Omega}  \xi \cdot  ( \xi \cdot \nabla) v \dVm \dt \\
&- \sigma \int_{0}^{T'} \into \chi_v ((u-v)\cdot \nabla ) (\nabla \cdot \xi) \dx \dt  \nonumber \\
& - \sigma \int_{0}^{T'} \into   \xi \cdot  ( (n_u - \xi )\cdot \nabla) v \dchiu \dt \nonumber \\
& - \sigma \int_{0}^{T'} \into  (n_u - \xi ) \cdot  ( \xi \cdot \nabla) v \dchiu \dt \nonumber \\
&+ \sigma \int_{0}^{T'} \into   ( \Id - \xi \otimes \xi ): \nabla v \dchiu \dt \nonumber.    
\end{align}
In summary, plugging back \eqref{RLHSkin}--\eqref{RHSadv} and~\eqref{RHSsurten} into~\eqref{sum}, 
and then summing~\eqref{5} to the resulting inequality yields
in view of the definition~\eqref{relent} of the relative entropy
\begin{align} \label{finalsum}
&E[\chi_u, u, V|\chi_v, v](T')  
+ \int_{0}^{T'} \into \fracmu |\nabla (u-v) + \nabla (u-v)^\mathsf{T} |^2 \dx \dt  \nonumber \\
&\leq  E[\chi_u, u, V|\chi_v, v](0) + R_{dt} + R_{adv} + R^{(1)}_{surTen} + R^{(2)}_{surTen}   
\end{align}
for almost every $T' \in [0,T]$, where in addition to the notation of Proposition~\ref{proprelent90}
we also defined the two auxiliary quantities
\begin{align}
\label{eq:auxRsurTen}
R^{(1)}_{surTen} &:= -\sigma \int_{0}^{T'} \int_{\overline{\Om} \times \Ss} (s - \xi ) 
\cdot ((s - \xi) \cdot \nabla) v \dV \dt
\\&~~~~\nonumber
+\sigma \int_{0}^{T'} \into (1 - \theta_t) \xi \cdot  ( \xi \cdot \nabla) v \dVm \dt
\\&~~~~\nonumber
+ \sigma \int_{0}^{T'} \int_{\partial\Omega}  \xi \cdot  ( \xi \cdot \nabla) v \dVm \dt,
\\\label{eq:auxRsurTen2}
R^{(2)}_{surTen} &:= \sigma  \int_{0}^T \into \chi_u  (u \cdot \nabla) (\nabla \cdot  \xi)  \dx \dt 
\\&~~~~\nonumber
- \sigma \int_{0}^{T'} \into \chi_v ((u{-}v)\cdot \nabla ) (\nabla \cdot \xi) \dx \dt
\\&~~~~\nonumber
- \sigma \int_{0}^{T'} \into   \xi \cdot  ( (n_u - \xi )\cdot \nabla) v \dchiu \dt 
\\&~~~~\nonumber
- \sigma \int_{0}^{T'} \into  (n_u - \xi ) \cdot  ( \xi \cdot \nabla) v \dchiu \dt
\\&~~~~\nonumber
+ \sigma \int_{0}^{T'} \into   ( \Id - \xi \otimes \xi ): \nabla v \dchiu \dt 
\\&~~~~\nonumber
- \sigma \int_{0}^{T'}  \into n_u \cdot  \pt \xi  \dchiu \dt.
\end{align}
The remainder of the proof is concerned with the post-processing of the term~$R^{(2)}_{surTen}$.

\textit{Step 3:} By adding zeros, we can rewrite the last right hand 
side term of~\eqref{eq:auxRsurTen2} as
\begin{align}
\label{eq:step3aux1}
\nonumber
&- \sigma \int_{0}^{T'}  \into n_u \cdot  \pt \xi  \dchiu \dt  
\\&
=- \sigma \int_{0}^{T'} \into (n_u {-} \xi ) \cdot  
(\pt \xi {+} (v \cdot \nabla)\xi {+} (\mathrm{Id}{-}\xi\otimes\xi)(\nabla v)^\mathsf{T} \xi)  \dchiu \dt 
\\\nonumber
&~~~ - \sigma \int_0^{T'} \into ((n_u - \xi) \cdot \xi ) (\xi\otimes\xi:\nabla v)\dchiu \dt 
\\\nonumber
&~~~ - \sigma \int_{0}^{T'}   \into \Big(\pt \frac{1}{2}|\xi|^2 + (v \cdot \nabla)\frac{1}{2}|\xi|^2\Big)   \dchiu \dt 
\\\nonumber
&~~~ + \sigma \int_{0}^{T'}   \into \xi \otimes (n_u - \xi ) : \nabla v \dchiu \dt 
\\\nonumber
&~~~ +  \sigma \int_{0}^{T'}   \into n_u \cdot ((v \cdot \nabla)\xi)\dchiu \dt .
\end{align}
We proceed by manipulating the last term in the latter identity. To this end, 
we compute applying the product rule in the first step and then adding	zero
\begin{align}
\label{eq:step3aux2}
\nonumber
&\sigma \int_{0}^{T'}  \into n_u \cdot ((v \cdot \nabla)\xi)\dchiu \dt 
\\&
=  \sigma \int_{0}^{T'}  \into n_u \cdot (\nabla \cdot (\xi \otimes v)) \dchiu \dt 
\\ \nonumber
&~~~+ \, \sigma \int_{0}^{T'}  \into (1 - n_u \cdot \xi) (\nabla \cdot v) \dchiu \dt 
-  \sigma \int_{0}^{T'}  \into \Id  : \nabla v \dchiu \dt  .
\end{align}
Noting that for symmetry reasons $\nabla \cdot (\nabla \cdot ( \xi \otimes v))
= \nabla \cdot (\nabla \cdot ( v \otimes \xi))$, 
an integration by parts based on the boundary conditions~\eqref{eq:boundaryValueXi} 
and~$(v\cdot n_{\partial\Omega})|_{\partial\Omega}=0$ entails 
\begin{align*}
\sigma& \int_{0}^{T'}  \into n_u \cdot (\nabla \cdot (\xi \otimes v)) \dchiu \dt \\
=& - \sigma \int_{0}^{T'} \into \chi_u \nabla \cdot (\nabla \cdot ( v \otimes \xi)) \dx \dt   
- \sigma \int_{0}^{T'}  \int_{ \partial \Om} \chi_u (\no \otimes v : \nabla \xi )\dS \dt \\
=& \; \sigma \int_{0}^{T'}  \into n_u \cdot (\nabla \cdot (v \otimes \xi)) \dchiu \dt
\\& 
+ \sigma \int_{0}^{T'}  \int_{ \partial \Om} \chi_u (\no \cdot ((\xi \cdot \nabla)v- (v\cdot \nabla)\xi)) \dS \dt .
\end{align*}
We next observe that the last right hand side term of the previous display is zero. 
Indeed, note first that thanks to the boundary conditions~\eqref{eq:boundaryValueXi} 
and~$(v\cdot n_{\partial\Omega})|_{\partial\Omega}=0$ the involved gradients are in fact
tangential gradients along~$\partial\Omega$. Since the tangential gradient
of a function only depends on its definition along the manifold, we are free to 
substitute~$(\xi \cdot \tano) \tano$ for~$\xi$ resp.\ $(v \cdot \tano) \tano$ for~$v$, 
obtaining in the process  
\begin{align*}
&\int_{0}^{T'} \int_{\partial \Om} \chi_u (\no \cdot ((\xi \cdot \nabla ) v - (v \cdot \nabla ) \xi))\dS \dt  \\
&= \int_{0}^{T'} \int_{\partial \Om} \chi_u  [ (\xi \cdot \nabla ) (v \cdot \tano)
-  (v \cdot \nabla ) (\xi \cdot \tano)] (\tano \cdot \no) \dS \dt \\
&~~~ + \int_{0}^{T'} \int_{\partial \Om} \chi_u [((v \cdot \tano) \xi 
- (\xi \cdot \tano) v)\cdot \nabla ) \tano ] \cdot\no  \dS \dt =0 .
\end{align*}
The combination of the previous two displays together with an integration by parts
and an application of the product rule thus yields
\begin{align*}
& \sigma  \int_{0}^{T'}  \into n_u \cdot (\nabla \cdot (\xi \otimes v)) \dchiu \dt 
\\& =  \sigma \int_{0}^{T'}  \into (n_u \cdot v)(\nabla \cdot \xi ) \dchiu \dt 
+  \sigma \int_{0}^{T'}  \into n_u \otimes \xi : \nabla v \dchiu \dt. 
\end{align*}
By another integration by parts, relying in the process 
also on $\nabla \cdot v =0$ and $(v \cdot \no)|_{\partial \Om} =0$,
we may proceed computing
\begin{align}
\label{eq:step3aux3}
\nonumber
\sigma & \int_{0}^{T'}  \into n_u \cdot (\nabla \cdot (\xi \otimes v)) \dchiu \dt 
\\\nonumber
= &  
- \sigma \int_{0}^{T'}  \into \chi_u \nabla \cdot( v(\nabla \cdot \xi )) \dx \dt 
+  \sigma \int_{0}^{T'}  \into n_u \otimes \xi : \nabla v \dchiu \dt 
\\
= &  
- \sigma \int_{0}^{T'}  \into \chi_u (v \cdot \nabla )(\nabla \cdot \xi ) \dx \dt 
+  \sigma \int_{0}^{T'}  \into n_u \otimes \xi : \nabla v \dchiu \dt. 
\end{align}
In summary, taking together~\eqref{eq:step3aux1}--\eqref{eq:step3aux3}
and adding for a last time zero yields 
\begin{align}
\label{eq:step3aux4}
\nonumber
-& \sigma \int_{0}^{T'}  \into n_u \cdot  \pt \xi  \dchiu \dt  \\
=&- \sigma \int_{0}^{T'}  \into \chi_u (v \cdot \nabla )(\nabla \cdot \xi ) \dx \dt \\\nonumber
&- \sigma \int_{0}^{T'} \into (n_u {-} \xi ) \cdot  
(\pt \xi {+} (v \cdot \nabla)\xi {+} (\mathrm{Id}{-}\xi\otimes\xi)(\nabla v)^\mathsf{T} \xi)  \dchiu \dt 
\\\nonumber& - \sigma \int_0^{T'} \into ((n_u - \xi) \cdot \xi ) (\xi\otimes\xi:\nabla v)\dchiu \dt \\\nonumber
& - \sigma \int_{0}^{T'}   \into \Big(\pt \frac{1}{2}|\xi|^2 + (v \cdot \nabla)\frac{1}{2}|\xi|^2\Big)   \dchiu \dt 
\\\nonumber&+ \sigma \int_{0}^{T'} \into  (1- n_u \cdot \xi)  (\nabla \cdot v )\dchiu \dt \\\nonumber
& +  \sigma \int_{0}^{T'}  \into (n_u - \xi) \otimes \xi  : \nabla v \dchiu \dt
+ \sigma \int_{0}^{T'} \into \xi \otimes (n_u - \xi ) : \nabla v \dchiu \dt \\\nonumber
& 	- \sigma \int_{0}^{T'}  \into (\Id - \xi\otimes \xi ): \nabla v \dchiu \dt.
\end{align}
Inserting~\eqref{eq:step3aux4} into~\eqref{eq:auxRsurTen2}
then implies that~$R^{(1)}_{surTen} + R^{(2)}_{surTen}$
combines to the desired term~$R_{surTen}$. In particular,
the estimate~\eqref{finalsum} upgrades to~\eqref{relentineq}
as asserted. \qed


\subsection{Time evolution of the bulk error: Proof of Lemma~\ref{lem:stabilityBulkError}}
\label{subsec:timeEvolBulkError}
Note that the sign conditions for the transported weight $\vartheta$,
see Definition~\ref{def:transportedWeight}, ensure that
\begin{align*}
E_{\mathrm{vol}}[\chi_u|\chi_v](t) = \int_\Omega 
\big(\chi_u(\cdot,t) - \chi_v(\cdot,t)\big) \vartheta(\cdot,t) \dx
\end{align*} 
for all $t\in [0,T]$. Hence, as a consequence of the transport equations for $\chi_v$ and $\chi_u$
(see Definition~\ref{Def_strongsol} and Definition~\ref{Def_varsol}, respectively)
one obtains
\begin{align}
\label{eq:auxComputation1}
E_{\mathrm{vol}}[\chi_u|\chi_v](T') 
&= E_{\mathrm{vol}}[\chi_u|\chi_v](0)
\\&~~~\nonumber
+ \int_0^{T'} \int_{\Omega} (\chi_u {-} \chi_v) \partial_t\vartheta \dx \dt
+ \int_0^{T'} \int_{\Omega} (\chi_uu {-} \chi_vv) \cdot\nabla\vartheta \dx \dt
\end{align}
for almost every $T'\in [0,T]$. Note that for any
sufficiently regular solenoidal vector field $F$
with $(F\cdot n_{\partial\Omega})|_{\partial\Omega}=0$, 
since $\vartheta=0$ along $I_v$ (see Definition~\ref{def:transportedWeight}), 
an integration by parts yields 
\begin{align}
\label{eq:auxComputation2}
\int_\Omega \chi_v (F\cdot\nabla)\vartheta \dx = 0.
\end{align}
Adding zero in~\eqref{eq:auxComputation1} and making use of~\eqref{eq:auxComputation2}
with respect to the choices $F=u$ and $F=v$ in form of 
$\int_\Omega \chi_v \big((u{-}v)\cdot\nabla\big)\vartheta \dx = 0$
then updates~\eqref{eq:auxComputation1} to~\eqref{eq:stabilityBulkErrorFunctional}.
This concludes the proof of Lemma~\ref{lem:stabilityBulkError}. \qed

\subsection{Conditional weak-strong uniqueness: Proof of Proposition~\ref{prop:conditionalWeakStrong}}
\label{subsec:proofCondWeakStrongUniqueness}
Starting point for a proof of the conditional weak-strong uniqueness principle
is the following important coercivity estimate (cf.\ \cite[Lemma~20]{Fischer2020}).

\begin{lemma}
\label{lem:coercivitySlicing}
Let the assumptions and notation of Proposition~\ref{prop:conditionalWeakStrong}
be in place. Then there exists a constant $C=C(\chi_v,v,T)>0$
such that for all $\delta\in (0,1]$ it holds
\begin{align}
\nonumber
\int_0^{T'} \int_{\Omega} |\chi_v{-}\chi_u||u{-}v| \dx \dt
&\leq \frac{C}{\delta} \int_0^{T'} E[\chi_u,u,V|\chi_v,v](t) + E_{\mathrm{vol}}[\chi_u|\chi_v](t) \dt
\\&~~~\label{eq:coercivitySlicing}
+ \delta \int_0^{T'} \int_{\Omega} |\nabla u - \nabla v|^2 \dx \dt
\end{align}
for all $T'\in [0,T]$.
\end{lemma}

\begin{proof}
It turns out to be convenient to introduce a decomposition of
the interface $I_v$ into its topological features:
the connected components of $I_v\cap\Omega$ and the connected components
of $I_v\cap \partial\Omega$. Let $N\in\mathbb{N}$ denote the total
number of such topological features of $I_v$, and split
$\{1,\ldots,N\}=:\mathcal{I}\cupdot\mathcal{C}$ as follows. 
The subset $\mathcal{I}$ enumerates the space-time connected components
of $I_v\cap\Omega$ (being time-evolving connected \textit{interfaces}), whereas
the subset $\mathcal{C}$ enumerates the space-time connected components of $I_v\cap\partial\Omega$
(being time-evolving \textit{contact points} if $d=2$, or time-evolving connected \textit{contact lines} if $d=3$).
If $i\in\mathcal{I}$, we let $\mathcal{T}_i$ denote the space-time trajectory in $\Omega$ of the 
corresponding connected interface. Furthermore, for every $c\in\mathcal{C}$ we write $\mathcal{T}_c$
representing the space-time trajectory in $\partial\Omega$ of the corresponding contact point
(if $d=2$) or line (if $d=3$). Finally, let us write $i\sim c$ for $i\in\mathcal{I}$
and $c\in\mathcal{C}$ if and only if $\mathcal{T}_i$ ends at $\mathcal{T}_c$.
With this language and notation in place, the proof is now split into five steps.

\textit{Step 1: (Choice of a suitable localization scale)} 
Denote by $n_{\partial\Omega}$ the unit normal vector field of $\partial\Omega$
pointing into $\Omega$, and by $n_{I_v}(\cdot,t)$ the unit normal vector field
of $I_v(t)$ pointing into $\Omega_v(t)$. Because of the uniform $C^2_x$ regularity 
of the boundary $\partial\Omega$ and the uniform $C_tC^2_x$ regularity of the interface $I_v(t)$,
$t\in [0,T]$, we may choose a scale $r \in (0,\frac{1}{2}]$ such that for all $t\in [0,T]$
and all $i\in\mathcal{I}$ the maps
\begin{align}
\label{eq:diffeoBoundaryOmega}
\Psi_{\partial\Omega}\colon \partial\Omega \times (-3r,3r) \to \R^d,
&\quad (x,y) \mapsto x + yn_{\partial\Omega}(x), 
\\
\label{eq:diffeoInterface}
\Psi_{\mathcal{T}_i(t)}\colon \mathcal{T}_i(t) \times (-3r,3r) \to \R^d,
&\quad (x,y) \mapsto x + yn_{I_v}(x,t)
\end{align}
are $C^1$ diffeomorphisms onto their image. By uniform regularity of $\partial\Omega$
and $I_v$ (the latter in space-time), we have bounds 
\begin{align}
\label{eq:boundsDiffeoBoundaryOmega}
\sup_{\partial\Omega{\times}[-r, r]} |\nabla\Psi_{\partial\Omega}| \leq C,
&\quad \sup_{\Psi_{\partial\Omega}(\partial\Omega{\times}[-r, r])}
|\nabla\Psi_{\partial\Omega}^{-1}| \leq C,
\\ \label{eq:boundsDiffeoInterface}
\sup_{t\in [0,T]} \sup_{\mathcal{T}_i(t){\times}[-r,r]} 
|\nabla\Psi_{\mathcal{T}_i(t)}| \leq C,
&\quad \sup_{t\in [0,T]} \sup_{\Psi_{\mathcal{T}_i(t)}(\mathcal{T}_i(t){\times}[-r,r])} 
|\nabla\Psi_{\mathcal{T}_i(t)}^{-1}| \leq C
\end{align}
for all $i\in\mathcal{I}$. By possibly choosing $r\in (0,\frac{1}{2}]$ even
smaller, we may also guarantee that for all $t\in [0,T]$ and all $i\in\mathcal{I}$ it holds
\begin{align}
\label{eq:locProperty1}
\Psi_{\mathcal{T}_i(t)}(\mathcal{T}_i(t){\times}[-r,r]) \cap 
\Psi_{\mathcal{T}_{i'}(t)}(\mathcal{T}_{i'}(t){\times}[-r,r]) &= \emptyset
\text{ for all } i'\in\mathcal{I},\,i'\neq i,
\\ \label{eq:locProperty2}
\Psi_{\mathcal{T}_i(t)}(\mathcal{T}_i(t){\times}[-r,r]) \cap 
\Psi_{\partial\Omega}(\partial\Omega{\times}[-r, r]) &\neq \emptyset
\,\Leftrightarrow\, \exists c \in \mathcal{C} \colon i \sim c,
\\ \label{eq:locProperty3}
\Psi_{\mathcal{T}_i(t)}(\mathcal{T}_i(t){\times}[-r,r]) \cap 
\Psi_{\partial\Omega}(\partial\Omega{\times}[-r, r]) &\subset 
B_{2r}(\mathcal{T}_c(t))
\text{ if } \exists c \in \mathcal{C} \colon i \sim c
\\ \label{eq:locProperty4}
B_{2r}(\mathcal{T}_c(t)) \cap B_{2r}(\mathcal{T}_{c'}(t)) &= \emptyset
\text{ for all } c,c'\in\mathcal{C},\,c'\neq c.
\end{align}
Note finally that because of the $90^\circ$ contact angle condition
and by possibly choosing $r\in (0,\frac{1}{2}]$ even smaller, we can 
furthermore ensure that
\begin{equation}
\begin{aligned}
\label{eq:inclusionLargeDistance}
&\Omega\setminus\Big(\Psi_{\partial\Omega}(\partial\Omega{\times}[-r, r]) \cup
\bigcup_{i\in\mathcal{I}}\Psi_{\mathcal{T}_i(t)}(\mathcal{T}_i(t){\times}[-r,r])\Big)
\\&
\subset \Omega \cap \{x\in\R^d\colon \dist(x,\partial\Omega) \wedge \dist(x,I_v(t)) > r\}
\end{aligned}
\end{equation}
for all $t\in [0,T]$.  Indeed, for $x\in\Omega\setminus\big(\Psi_{\partial\Omega}(\partial\Omega{\times}[-r, r]) \cup
\bigcup_{i\in\mathcal{I}}\Psi_{\mathcal{T}_i(t)}(\mathcal{T}_i(t){\times}[-r,r])\big)$
it follows that $\dist(x,\partial\Omega) > r$. In case the interface $I_v(t)$
intersects $\partial\Omega$ it may not be immediately clear that also $\dist(x,I_v(t)) > r$
holds true. Assume there exists a point $x\in\Omega\setminus\big(\Psi_{\partial\Omega}(\partial\Omega{\times}[-r, r]) \cup
\bigcup_{i\in\mathcal{I}}\Psi_{\mathcal{T}_i(t)}(\mathcal{T}_i(t){\times}[-r,r])\big)$
such that $\dist(x,I_v(t)) \leq r$. Then necessarily 
$x\in(\Omega\cap B_{r}(c(t)))\setminus
\bigcup_{i\in\mathcal{I}}\Psi_{\mathcal{T}_i(t)}(\mathcal{T}_i(t){\times}[-r,r])$
for some boundary point $c(t)\in\partial\Omega\cap I_v(t)$. Hence, because of the
uniform $C^2_x$ regularity of $\partial\Omega$ and $I_v(t)$ intersecting $\partial\Omega$
at an angle of $90^\circ$, one may choose $r\in (0,\frac{1}{2}]$ small enough
such that $x\in(\Omega\cap B_{r}(c(t)))$ implies $\dist(x,\partial\Omega)\leq r$.
As we have already seen, this contradicts $x\in\Omega\setminus
\Psi_{\partial\Omega}(\partial\Omega{\times}[-r, r])$.

\textit{Step 2: (A reduction argument)} We may estimate by a union bound 
and~\eqref{eq:inclusionLargeDistance}
\begin{align}
\nonumber
&\int_0^{T'} \int_{\Omega} |\chi_v{-}\chi_u||u{-}v| \dx \dt
\\& \label{eq:reductionCoercivity}
\leq \int_0^{T'} \int_{\Omega \cap \Psi_{\partial\Omega}(\partial\Omega{\times}[-r, r])
\setminus \bigcup_{c\in\mathcal{C}} B_{2r}(\mathcal{T}_c(t))} 
|\chi_v{-}\chi_u||u{-}v| \dx \dt
\\&~~~ \nonumber
+ \sum_{i\in\mathcal{I}}\int_0^{T'} 
\int_{\Omega \cap \Psi_{\mathcal{T}_i(t)}(\mathcal{T}_i(t){\times}[-r,r])
\setminus \bigcup_{c\in\mathcal{C}} B_{2r}(\mathcal{T}_c(t))} 
|\chi_v{-}\chi_u||u{-}v| \dx \dt
\\&~~~ \nonumber
+ C \sum_{c\in\mathcal{C}}\int_0^{T'} 
\int_{\Omega\cap B_{2r}(\mathcal{T}_c(t))}
|\chi_v{-}\chi_u||u{-}v| \dx \dt
\\&~~~ \nonumber
+ \int_0^{T'} \int_{\Omega \cap \{\dist(\cdot,\partial\Omega)\wedge\dist(\cdot,I_v(t)) > r\}} 
|\chi_v{-}\chi_u||u{-}v| \dx \dt.
\end{align}
An application of H\"older's inequality and Young's inequality,
the definition~\eqref{relent} of the relative entropy functional,
the coercivity estimate~\eqref{eq:lowerBoundTransportedWeight} for the transported weight,
and the definition~\eqref{eq:bulkErrorFunctional} of the bulk error functional
further imply 
\begin{align*}
&\int_0^{T'} \int_{\Omega \cap \{\dist(\cdot,\partial\Omega)\wedge\dist(\cdot,I_v(t)) > r\}} 
|\chi_v{-}\chi_u||u{-}v| \dx \dt
\\&
\leq C\int_0^{T'} \int_{\Omega \cap \{\dist(\cdot,\partial\Omega)
\wedge\dist(\cdot,I_v(t)) > r\}} |\chi_v{-}\chi_u| \dx \dt
+ C \int_0^{T'} E[\chi_u,u,V|\chi_v,v](t) \dt
\\&
\leq C\int_0^{T'} E[\chi_u,u,V|\chi_v,v](t) + E_{\mathrm{vol}}[\chi_u,\chi_v](t) \dt.
\end{align*}
Hence, it remains to estimate the first three terms on the
right hand side of~\eqref{eq:reductionCoercivity}.

\textit{Step 3: (Estimate near the interface but away from contact points)}
First of all, because of the localization 
properties~\eqref{eq:locProperty1}--\eqref{eq:locProperty3} it holds
for all $i\in\mathcal{I}$
\begin{align}
\label{eq:distTopFeatures1}
\dist(\cdot,\mathcal{T}_i) = \dist(\cdot,\partial\Omega) 
\wedge \dist(\cdot,I_v(t))
\end{align}
in $\Omega \cap \Psi_{\mathcal{T}_i(t)}(\mathcal{T}_i(t){\times}[-r,r])
\setminus \bigcup_{c\in\mathcal{C}} B_{2r}(\mathcal{T}_c(t))$.
Hence, the local interface error height as measured in the direction 
of $n_{I_v}$ on $\mathcal{T}_i$
\begin{align*}
h_{\mathcal{T}_i}(x,t) := \int_{-r}^r 
|\chi_u - \chi_v|(\Psi_{\mathcal{T}_i(t)}(x,y),t) \dy,
\quad x\in\mathcal{T}_i(t),\,t\in [0,T],
\end{align*}
is, because of~\eqref{eq:distTopFeatures1} and the coercivity
estimate~\eqref{eq:lowerBoundTransportedWeight} of the transported 
weight $\vartheta$, subject to the estimate
\begin{align}
\nonumber
h_{\mathcal{T}_i}^2(x,t) &\leq 
C\int_{-r}^r |\chi_u - \chi_v|(\Psi_{\mathcal{T}_i(t)}(x,y),t) y \dy
\\& \label{eq:boundInterfaceErrorHeight1}
\leq C \int_{-r}^r |\chi_u - \chi_v|(\Psi_{\mathcal{T}_i(t)}(x,y),t) 
|\vartheta|(\Psi_{\mathcal{T}_i(t)}(x,y),t) \dy
\end{align}
for all $x\in\mathcal{T}_i(t)\setminus\bigcup_{c\in\mathcal{C}}B_{2r}(\mathcal{T}_c(t))$,
all $t\in [0,T]$ and all $i\in\mathcal{I}$. Carrying out the slicing argument of the proof
of~\cite[Lemma~20]{Fischer2020} in $\Omega \cap \Psi_{\mathcal{T}_i(t)}(\mathcal{T}_i(t){\times}[-r,r])
\setminus \bigcup_{c\in\mathcal{C}} B_{2r}(\mathcal{T}_c(t))$ by means of~$ \Psi_{\mathcal{T}_i(t)}$, 
which is indeed admissible thanks to~\eqref{eq:diffeoInterface}, 
\eqref{eq:boundsDiffeoInterface} and~\eqref{eq:boundInterfaceErrorHeight1},
shows that one obtains an estimate of required form
\begin{align*}
&\sum_{i\in\mathcal{I}}\int_0^{T'} 
\int_{\Omega \cap \Psi_{\mathcal{T}_i(t)}(\mathcal{T}_i(t){\times}[-r,r])
\setminus \bigcup_{c\in\mathcal{C}} B_{2r}(\mathcal{T}_c(t))} 
|\chi_v{-}\chi_u||u{-}v| \dx \dt
\\&
\leq \frac{C}{\delta} \int_0^{T'} E[\chi_u,u,V|\chi_v,v](t) + E_{\mathrm{vol}}[\chi_u|\chi_v](t) \dt
+ \delta \int_0^{T'} \int_{\Omega} |\nabla u - \nabla v|^2 \dx \dt.
\end{align*}

\textit{Step 4: (Estimate near the boundary of the domain but away from contact points)}
The argument is similar to the one of the previous step, with the only major
difference being that the slicing argument of the proof
of~\cite[Lemma~20]{Fischer2020} is now carried out in $\Omega \cap \Psi_{\partial\Omega}(\partial\Omega{\times}[-r, r])
\setminus \bigcup_{c\in\mathcal{C}} B_{2r}(\mathcal{T}_c(t))$
by means of $\Psi_{\partial\Omega}$. This
in turn is facilitated by the following facts. First, the localization 
properties~\eqref{eq:locProperty1}--\eqref{eq:locProperty3} ensure 
\begin{align}
\label{eq:distTopFeatures2}
\dist(\cdot,\partial\Omega) = \dist(\cdot,\partial\Omega) 
\wedge \dist(\cdot,I_v(t))
\end{align}
in  $\Omega \cap \Psi_{\partial\Omega}(\partial\Omega{\times}[-r, r])
\setminus \bigcup_{c\in\mathcal{C}} B_{2r}(\mathcal{T}_c(t))$.
Second, as a consequence of~\eqref{eq:distTopFeatures2} and
the coercivity estimate~\eqref{eq:lowerBoundTransportedWeight} of the transported 
weight $\vartheta$, the local interface error height as measured in 
the direction of $n_{\partial\Omega}$
\begin{align*}
h_{\partial\Omega}(x,t) := \int_{-r}^r |\chi_u - \chi_v|(\Psi_{\partial\Omega}(x,y),t) \dy,
\quad x\in\partial\Omega,\,t\in [0,T],
\end{align*}
satisfies the estimate
\begin{align}
\nonumber
h_{\partial\Omega}^2(x,t) &\leq 
C\int_{-r}^r |\chi_u - \chi_v|(\Psi_{\partial\Omega}(x,y),t) y \dy
\\& \label{eq:boundInterfaceErrorHeight2}
\leq C \int_{-r}^r |\chi_u - \chi_v|(\Psi_{\partial\Omega}(x,y),t) 
|\vartheta|(\Psi_{\partial\Omega}(x,y),t) \dy.
\end{align}
Hence, we obtain
\begin{align*}
&\int_0^{T'} \int_{\Omega \cap \Psi_{\partial\Omega}(\partial\Omega{\times}[-r, r])
\setminus \bigcup_{c\in\mathcal{C}} B_{2r}(\mathcal{T}_c(t))} 
|\chi_v{-}\chi_u||u{-}v| \dx \dt
\\&
\leq \frac{C}{\delta} \int_0^{T'} E[\chi_u,u,V|\chi_v,v](t) + E_{\mathrm{vol}}[\chi_u|\chi_v](t) \dt
+ \delta \int_0^{T'} \int_{\Omega} |\nabla u - \nabla v|^2 \dx \dt.
\end{align*}

\textit{Step 5: (Estimate near contact points)} Fix $c\in\mathcal{C}$,
and let $i\in\mathcal{I}$ denote the unique connected interface $\mathcal{T}_i$
such that $i\sim c$. Because of the regularity of $\partial\Omega$,
the regularity of $\mathcal{T}_i$, and the $90^\circ$ contact angle condition
we may decompose the neighborhood $\Omega\cap B_{2r}(\mathcal{T}_c(t))$---by possibly reducing the localization 
scale $r\in (0,\frac{1}{2}]$ even further---into three pairwise disjoint open sets
$W_{\partial\Omega}(t)$, $W_{\mathcal{T}_i}(t)$ and $W_{\partial\Omega\sim\mathcal{T}_i}(t)$
such that $\Omega\cap B_{2r}(\mathcal{T}_c(t)) \setminus 
\big(W_{\partial\Omega}(t) \cup W_{\mathcal{T}_i}(t) \cup W_{\partial\Omega\sim\mathcal{T}_i}(t)\big)$
is an $\mathcal{H}^d$ null set and
\begin{align}
\label{eq:distTopFeatures3}
\dist(\cdot,\partial\Omega) &= \dist(\cdot,\partial\Omega) \wedge \dist(\cdot,I_v(t))
&&\text{ in } W_{\partial\Omega}(t), 
\\ \label{eq:distTopFeatures4}
\dist(\cdot,\mathcal{T}_i(t)) &= \dist(\cdot,\partial\Omega) \wedge \dist(\cdot,I_v(t))
&&\text{ in } W_{\mathcal{T}_i}(t), 
\\ \label{eq:distTopFeatures5}
\dist(\cdot,\partial\Omega) &\sim \dist(\cdot,\mathcal{T}_i(t))
\sim \dist(\cdot,I_v(t))
&&\text{ in } W_{\partial\Omega\sim\mathcal{T}_i}(t),
\end{align}
as well as
\begin{align}
\label{eq:locProperty5}
 W_{\partial\Omega}(t) &\subset \Psi_{\partial\Omega}(\partial\Omega{\times}(-3r,3r)),
\\ \label{eq:locProperty6}
W_{\mathcal{T}_i}(t) &\subset \Psi_{\mathcal{T}_i(t)}(\mathcal{T}_i(t){\times}(-3r,3r)),
\\ \label{eq:locProperty7}
W_{\partial\Omega\sim\mathcal{T}_i}(t) &\subset 
\Psi_{\partial\Omega}(\partial\Omega{\times}(-3r,3r))
\cap \Psi_{\mathcal{T}_i(t)}(\mathcal{T}_i(t){\times}(-3r,3r)).
\end{align}
(Up to a rigid motion, these sets can in fact be defined independent of $t\in [0,T]$.)
Hence, applying the argument of \textit{Step~3} based on~\eqref{eq:distTopFeatures4}
and~\eqref{eq:locProperty6} with respect to
$\Omega\cap B_{2r}(\mathcal{T}_c(t))\cap W_{\mathcal{T}_i}(t)$,
the argument of \textit{Step~4} based on~\eqref{eq:distTopFeatures3}
and~\eqref{eq:locProperty5} with respect to
$\Omega\cap B_{2r}(\mathcal{T}_c(t))\cap W_{\partial\Omega}(t)$,
and either the argument of \textit{Step~3} or \textit{Step~4}
based on~\eqref{eq:distTopFeatures5} and~\eqref{eq:locProperty7}
with respect to $\Omega\cap B_{2r}(\mathcal{T}_c(t))\cap W_{\partial\Omega\sim\mathcal{T}_i}(t)$
entails
\begin{align*}
&\sum_{c\in\mathcal{C}}\int_0^{T'} 
\int_{\Omega\cap B_{2r}(\mathcal{T}_c(t))}
|\chi_v{-}\chi_u||u{-}v| \dx \dt
\\&
\leq \frac{C}{\delta} \int_0^{T'} E[\chi_u,u,V|\chi_v,v](t) + E_{\mathrm{vol}}[\chi_u|\chi_v](t) \dt
+ \delta \int_0^{T'} \int_{\Omega} |\nabla u - \nabla v|^2 \dx \dt.
\end{align*}
This in turn concludes the proof of Lemma~\ref{lem:coercivitySlicing}.
\end{proof}

\begin{proof}[Proof of Proposition~\ref{prop:conditionalWeakStrong}]
The proof proceeds in three steps.

\textit{Step 1: (Post-processing the relative entropy inequality~\eqref{relentineq})}
It follows immediately from the $L^\infty_{x,t}$-bound for $\partial_t v$ and 
$\rho(\chi_v)-\rho(\chi_u)=(\rho^+{-}\rho^-)(\chi_v{-}\chi_u)$ that
\begin{align}
\label{eq:postProcessedRdt}
|R_{dt}| \leq C \int_0^{T'} \int_{\Omega} |\chi_v{-}\chi_u||u{-}v| \dx \dt
\end{align}
for almost every $T'\in [0,T]$. Furthermore, the $L^\infty_tW^{1,\infty}_x$-bound
for $v$, the definition~\eqref{relent} of the relative entropy functional, and again 
the identity $\rho(\chi_v)-\rho(\chi_u)=(\rho^+{-}\rho^-)(\chi_v{-}\chi_u)$ imply that
\begin{align}
\label{eq:postProcessedRadv}
|R_{adv}| \leq C \int_0^{T'} \int_{\Omega} |\chi_v{-}\chi_u||u{-}v| \dx \dt
+ C \int_0^{T'} E[\chi_u,u,V|\chi_v,v](t) \dt
\end{align}
for almost every $T'\in [0,T]$. For a bound on the interface contribution
$R_{surTen}$, we rely on the $L^\infty_tW^{1,\infty}_x$-bound
for $v$, the $L^\infty_tW^{2,\infty}_x$-bound for $\xi$, the
$L^\infty_tW^{1,\infty}_x$-bound for $B$, the definition~\eqref{relent} of the relative entropy functional,
as well as the estimates~\eqref{eq:timeEvolutionXi} and~\eqref{eq:timeEvolutionLengthXi}
of a boundary adapted extension~$\xi$ of~$n_{I_v}$ to the effect that
\begin{align}
\label{eq:postProcessedRsurTen}
|R_{surTen}| &\leq 
C \int_0^{T'} \int_{\Omega} |\chi_v{-}\chi_u||u{-}v| \dx \dt
\\&~~~ \nonumber
+ C \int_0^{T'} \int_{\overline{\Omega}\times\mathbb{S}^{d-1}} |s-\xi|^2 \,\mathrm{d}V_t(x,s) \dt
\\&~~~ \nonumber
+ C \int_{0}^{T'} \int_{\Omega} 1- \theta_t \,\mathrm{d} |V_t|_{\Ss} \dt
\\&~~~ \nonumber
+ C  \int_{0}^{T'} \int_{\partial\Omega} 1 \,\mathrm{d} |V_t|_{\Ss} \dt 
\\&~~~ \nonumber
+ C \int_0^{T'} \int_{\Omega} |n_u - \xi|^2 \,\mathrm{d}|\nabla\chi_u| \dt
\\&~~~ \nonumber
+ C \int_0^{T'} \int_{\Omega} \dist^2(\cdot,I_v) \wedge 1 
\,\mathrm{d}|\nabla\chi_u| \dt
\\&~~~ \nonumber
+ C \int_0^{T'} \int_{\Omega} |\xi \cdot (\xi - n_u)|
\,\mathrm{d}|\nabla\chi_u| \dt
\\&~~~ \nonumber
+ C \int_0^{T'} E[\chi_u,u,V|\chi_v,v](t) \dt
\end{align}
for almost every $T'\in [0,T]$. 
It follows from property~\eqref{eq:coercivityByModulationOfLength} of a boundary adapted extension~$\xi$
and the trivial estimates $|\xi \cdot (\xi - n_u)| \leq (1{-}|\xi|^2) + (1 {-} n_u \cdot \xi )
\leq 2(1{-}|\xi|) + (1 {-} n_u \cdot \xi )$ and $1-|\xi| \leq 1 - n_u\cdot\xi$ that
\begin{align}
\label{eq:postProcessedAux1}
&\int_0^{T'} \int_{\Omega} \dist^2(\cdot,I_v) \wedge 1 
\,\mathrm{d}|\nabla\chi_u| \dt  + \int_0^{T'} \int_{\Omega} |\xi \cdot (\xi - n_u)|
\,\mathrm{d}|\nabla\chi_u| \dt \\
&\leq C \int_0^{T'} E[\chi_u,u,V|\chi_v,v](t) \dt. \nonumber 
\end{align}
Moreover, the trivial estimate $|n_u-\xi|^2 \leq 2(1-n_u\cdot \xi)$ implies
\begin{align}
\label{eq:postProcessedAux2}
\int_0^{T'} \int_{\Omega} |n_u - \xi|^2 \,\mathrm{d}|\nabla\chi_u| \dt
\leq C \int_0^{T'} E[\chi_u,u,V|\chi_v,v](t) \dt.
\end{align}
Recall finally from~\eqref{eq:tiltExcessVarifoldControl}
and~\eqref{eq:multiplicityControl} that
\begin{align}
\nonumber
\int_0^{T'} \int_{\overline{\Omega}\times\mathbb{S}^{d-1}} |s-\xi|^2 \,\mathrm{d}V_t(x,s) \dt
&\leq C \int_0^{T'} E[\chi_u,u,V|\chi_v,v](t) \dt,
\\ \label{eq:postProcessedAux3}
\int_{0}^{T'} \int_{\Omega} 1- \theta_t \,\mathrm{d} |V_t|_{\Ss} \dt
+ \int_{0}^{T'} \int_{\partial\Omega} 1 \,\mathrm{d} |V_t|_{\Ss} \dt 
&\leq C \int_0^{T'} E[\chi_u,u,V|\chi_v,v](t) \dt.
\end{align}

By inserting back the estimates~\eqref{eq:postProcessedRdt}--\eqref{eq:postProcessedAux3}
into the relative entropy inequality~\eqref{relentineq}, then making use of the
coercivity estimate~\eqref{eq:coercivitySlicing} and Korn's inequality, and finally
carrying out an absorption
argument, it follows that there exist two constants $c=c(\chi_v,v,T)>0$ and $C=C(\chi_v,v,T)>0$ such that
for almost every $T'\in [0,T]$
\begin{align}
\nonumber
&E[\chi_u,u,V|\chi_v,v](T') + c\int_{0}^{T'} \int_{\Omega} |\nabla (u{-}v) + \nabla (u{-}v)^\mathsf{T} |^2 \dx \dt
\\& \label{eq:postProcessedRelEntropyInequality}
\leq E[\chi_u,u,V|\chi_v,v](0) 
+ C\int_0^{T'} E[\chi_u,u,V|\chi_v,v](t) + E_{\mathrm{vol}}[\chi_u|\chi_v](t) \dt.
\end{align}

\textit{Step 2: (Post-processing the identity~\eqref{eq:stabilityBulkErrorFunctional})}
By the $L^\infty_tW^{1,\infty}_x$-bound for the transported weight $\vartheta$,
the estimate~\eqref{eq:advDerivTransportedWeight} on the advective derivative of the transported weight $\vartheta$,
and the definition~\eqref{eq:bulkErrorFunctional} of the bulk error functional
we infer that
\begin{align*}
E_{\mathrm{vol}}[\chi_u|\chi_v](T') &\leq E_{\mathrm{vol}}[\chi_u|\chi_v](0)
+ C\int_0^{T'} E_{\mathrm{vol}}[\chi_u|\chi_v](t) \dt  
\\&~~~
+ C\int_0^{T'} \int_{\Omega} |\chi_v{-}\chi_u||u{-}v| \dx \dt
\end{align*}
for almost every $T'\in [0,T]$. Adding~\eqref{eq:postProcessedRelEntropyInequality}
to the previous display, and making use of the coercivity estimate~\eqref{eq:coercivitySlicing}
in combination with Korn's inequality and an absorption argument thus implies that
for almost every $T'\in [0,T]$
\begin{align}
\nonumber
&E[\chi_u,u,V|\chi_v,v](T') + E_{\mathrm{vol}}[\chi_u|\chi_v](T')
+ c\int_{0}^{T'} \int_{\Omega} |\nabla (u{-}v) + \nabla (u{-}v)^\mathsf{T} |^2 \dx \dt
\\& \label{eq:postProcessedBulkError}
\leq E[\chi_u,u,V|\chi_v,v](0) + E_{\mathrm{vol}}[\chi_u|\chi_v](0)
\\&~~~ \nonumber
+ C\int_0^{T'} E[\chi_u,u,V|\chi_v,v](t) + E_{\mathrm{vol}}[\chi_u|\chi_v](t) \dt.
\end{align}

\textit{Step 3: (Conclusion)} The stability estimates~\eqref{eq:stabilityRelEntropy} 
and~\eqref{eq:stabilityBulkError} are an immediate consequence of the
estimate~\eqref{eq:postProcessedBulkError} by an application of Gronwall's lemma. In case
of coinciding initial conditions, it follows that $E_{\mathrm{vol}}[\chi_u|\chi_v](t) = 0$
for almost every $t\in [0,T]$. This in turn implies that $\chi_u(\cdot,t)=\chi_v(\cdot,t)$
almost everywhere in $\Omega$ for almost every $t\in[0,T]$. The asserted representation
of the varifold follows from the fact that $E[\chi_u,u,V|\chi_v,v](t) = 0$ for almost
every $t\in[0,T]$. This concludes the proof of the conditional
weak-strong uniqueness principle.
\end{proof}

\subsection{Proof of Theorem~\ref{theo:mainResult}}
This is now an immediate consequence of Proposition~\ref{prop:conditionalWeakStrong} 
and the existence results of Proposition~\ref{prop:existenceCalibration}
and Lemma~\ref{lem:existenceTransportedWeight}, respectively. \qed

\section{Bulk extension of the interface unit normal}
\label{sec:localCalibrationBulk}
The aim of this short section is the construction of an extension of
the interface unit normal in the vicinity of
a space-time trajectory in~$\Omega$ of a connected component of the 
interface~$I_v$ corresponding to a strong solution in the sense of Definition~\ref{Def_strongsol}
on a time interval~$[0,T]$.

Mainly for reference purposes in later sections, it turns out to be beneficial
to introduce already at this stage some notation in relation
to a decomposition of the interface~$I_v$ into its topological features:
the connected components of $I_v\cap\Omega$ and the connected components
of $I_v\cap \partial\Omega$. Denoting by $N\in\mathbb{N}$ the total
number of such topological features present in the interface~$I_v$ we split
$\{1,\ldots,N\}=:\mathcal{I}\cupdot\mathcal{C}$ by means of two disjoint subsets.
In particular, the subset~$\mathcal{I}$ enumerates the space-time connected components
of $I_v\cap\Omega$, i.e., time-evolving connected \textit{interfaces}, whereas
the subset $\mathcal{C}$ enumerates the space-time connected components of 
$I_v\cap\partial\Omega$, i.e., time-evolving \textit{contact points}. If $i\in\mathcal{I}$, 
we denote by $\mathcal{T}_i:=\bigcup_{t\in [0,T]}\mathcal{T}_{i}(t){\times}\{t\}\subset I_v\cap (\Omega{\times}[0,T])$ 
the space-time trajectory of the corresponding connected interfaces
$\mathcal{T}_{i}(t)\subset I_v(t)\cap\Omega$, $t\in [0,T]$.

For each~$i\in\mathcal{I}$, we want to define a vector field~$\xi^{i}$
subject to conditions as in Definition~\ref{def:calibrationTwoPhaseFluidFlow}; at least
in a suitable neighborhood of~$\mathcal{T}_i$. We first formalize what we mean by 
the latter in form of the following definition.

\begin{definition}
\label{def:locRadiusInterface}
Let $d=2$, and let $\Omega\subset\R^2$ be a bounded domain with orientable and smooth boundary.
Let $(\chi_v,v)$ be a strong solution to the incompressible
Navier--Stokes equation for two fluids in the sense of Definition~\ref{Def_strongsol}
on a time interval $[0,T]$. Fix a two-phase interface $i\in\mathcal{I}$.
We call $r_i\in (0,1]$ an \emph{admissible localization radius for the interface} 
$\mathcal{T}_i \subset I_v\cap (\Omega{\times}[0,T])$ if the map
	\begin{align} \label{eq:diffeoInterfaceLoc} 
	\Psi_{\mathcal{T}_i}\colon \mathcal{T}_{i} \times (-2r_i,2r_i) \to \R^2\times[0,T], \; \;
	(x,t,s) \mapsto \big(x + s n_{I_v}(x,t),t\big)
	\end{align}
is bijective onto its image~$\mathrm{im}(\Psi_{\mathcal{T}_i})
:=\Psi_{\mathcal{T}_i}\big(\mathcal{T}_{i} {\times} (-2r_i,2r_i)\big)$,
and its inverse is a diffeomorphism of class $C^0_tC^2_x(\overline{\mathrm{im}(\Psi_{\mathcal{T}_i})})
\cap C^1_tC^0_x(\overline{\mathrm{im}(\Psi_{\mathcal{T}_i})})$.

In case such a scale~$r_i\in (0,1]$ exists, we may express the
inverse by means of $\Psi_{\mathcal{T}_i}^{-1}=:
(P_{\mathcal{T}_i},\Id, s_{\mathcal{T}_i})\colon \mathrm{im}(\Psi_{\mathcal{T}_i})
\to\mathcal{T}_{i} {\times} (-2r_i,2r_i)$. Hence,
the map~$P_{\mathcal{T}_i}$ represents in each time slice the \emph{nearest-point projection} 
onto the interface $\mathcal{T}_i(t) \subset I_v(t) \cap \Omega$, $t\in [0,T]$,
whereas~$s_{\mathcal{T}_i}$ bears the interpretation of a \emph{signed distance function}
with orientation fixed by~$\nabla s_{\mathcal{T}_i} = n_{I_v}$. In particular,
$s_{\mathcal{T}_i}\in C^0_tC^3_x(\overline{\mathrm{im}(\Psi_{\mathcal{T}_i})})
\cap C^1_tC^1_x(\overline{\mathrm{im}(\Psi_{\mathcal{T}_i})})$
as well as $P_{\mathcal{T}_i}\in C^0_tC^2_x(\overline{\mathrm{im}(\Psi_{\mathcal{T}_i})})
\cap C^1_tC^0_x(\overline{\mathrm{im}(\Psi_{\mathcal{T}_i})})$.

By a slight abuse of notation, we extend to~$\mathrm{im}(\Psi_{\mathcal{T}_i})$ 
the definition of the normal vector field resp.\ the scalar mean curvature of~$\mathcal{T}_i$
by means of 
\begin{align}
\label{eq:extNormal}
n_{I_v}\colon \mathrm{im}(\Psi_{\mathcal{T}_i})\to\mathbb{S}^{1},\quad
&(x,t)\mapsto n_{I_v}(P_{\mathcal{T}_i}(x,t),t) = \nabla s_{\mathcal{T}_i}(x,t),
\\
\label{eq:extMeanCurvature}
H_{I_v}\colon \mathrm{im}(\Psi_{\mathcal{T}_i})\to\mathbb{R},\quad
&(x,t)\mapsto -(\Delta s_{\mathcal{T}_i})(P_{\mathcal{T}_i}(x,t),t).
\end{align} 
Hence, we may register that $n_{I_v}\in C^0_tC^2_x(\overline{\mathrm{im}(\Psi_{\mathcal{T}_i})})
\cap C^1_tC^0_x(\overline{\mathrm{im}(\Psi_{\mathcal{T}_i})})$
as well as $H_{I_v} \in C^0_tC^1_x(\overline{\mathrm{im}(\Psi_{\mathcal{T}_i})})$.
\end{definition}

It is clear from Definition~\ref{Def_strongsol} of a strong solution to the 
incompressible Navier--Stokes equation for two fluids, in particular Definition~\ref{Def_smdom} 
of smoothly evolving domains and interfaces, that all interfaces admit an admissible localization radius in
the sense of Definition~\ref{def:locRadiusInterface} as a consequence of the tubular
neighborhood theorem. 

\begin{construction}
\label{const:CalibrationInterface}
	Let $d=2$, and let $\Omega\subset\R^2$ be a bounded domain with orientable and smooth boundary.
	Let $(\chi_v,v)$ be a strong solution to the incompressible
	Navier--Stokes equation for two fluids in the sense of Definition~\ref{Def_strongsol}
	on a time interval $[0,T]$. Fix a two-phase interface $i\in\mathcal{I}$ and let $r_i\in (0,1]$ be an admissible localization radius for the interface $\mathcal{T}_i \subset I_v$ in the sense of Definition~\ref{def:locRadiusInterface}.
	Then \emph{a bulk extension of the unit normal~$n_{I_v}$ along a smooth interface} 
	$\mathcal{T}_i $ is the vector field~$\xi^i$ defined by
	\begin{align} 
		&\xi^i(x,t) := n_{I_v}(x,t), \label{eq:DefXiInterface} 
		&& (x,t) \in \mathrm{im}(\Psi_{\mathcal{T}_i}) \cap (\Omega{\times}[0,T]).
	\end{align}
\end{construction}

We record the required properties of the vector field~$\xi^{i}$.

\begin{proposition}
	\label{prop:CalibrationInterface}
	Let the assumptions and notation of Construction~\ref{const:CalibrationInterface}
	be in place. Then, in terms of regularity it holds that $\xi^{i}\in C_t^0C_x^2 \cap C_t^1C_x^0
	(\overline{\mathrm{im}(\Psi_{\mathcal{T}_i})\cap(\Omega{\times}[0,T])})$.
	Moreover, we have
	\begin{align}
	\label{eq:divMeanCurvature}
	\nabla\cdot\xi^{i} + H_{I_v} = O(\dist(\cdot,\mathcal{T}_i)),
	\\
	\label{eq:evolEquXiInterface}
	\partial_t\xi^{i}
	+ (v\cdot\nabla)\xi^{i}
	+ (\Id - \xi^i \otimes \xi^i)(\nabla v)^\mathsf{T}\xi^{i} 
	&= O(\dist(\cdot,\mathcal{T}_i)),
	\\
	\label{eq:evolEquLengthXiInterface}
	\partial_t|\xi^{i}|^2
	+ (v\cdot\nabla)|\xi^{i}|^2
	&= 0
	\end{align}
	throughout the space-time domain~$\mathrm{im}(\Psi_{\mathcal{T}_i})\cap(\Omega{\times}[0,T])$.
\end{proposition}

\begin{proof} 
The asserted regularity of~$\xi^{i}$ is a direct consequence of its 
definition~\eqref{eq:DefXiInterface} and the regularity of~$n_{I_v}$
from Definition~\ref{def:locRadiusInterface}. 
In view of the definitions~\eqref{eq:DefXiInterface}, \eqref{eq:extNormal} 
and \eqref{eq:extMeanCurvature}, the estimate~\eqref{eq:divMeanCurvature}
is directly implied by a Lipschitz estimate based on the regularity of~$H_{I_v}$
from Definition~\ref{def:locRadiusInterface}. The equation~\eqref{eq:evolEquLengthXiInterface}
is trivially fulfilled because~$\xi^{i}$ is a unit vector, cf.\ the definition~\eqref{eq:DefXiInterface}.

For a proof of~\eqref{eq:evolEquXiInterface}, we first note that
$\partial_t s_{\mathcal{T}_i}(x,t) = -\big(v(P_{\mathcal{T}_i}(x,t),t)\cdot\nabla\big) s_{\mathcal{T}_i}(x,t)$ 
for all~$(x,t)\in\mathrm{im}(\Psi_{\mathcal{T}_i})\cap(\Omega{\times}[0,T])$.
Indeed, $\partial_t s_{\mathcal{T}_i}$ equals the normal speed (oriented with respect to~$-n_{I_v}$) 
of the nearest point on the connected interface~$\mathcal{T}_i$, 
which in turn by~$n_{I_v}=\nabla s_{\mathcal{T}_i}$
is precisely given by the asserted right hand side term. Differentiating
the equation for the time evolution of~$s_{\mathcal{T}_i}$ then yields~\eqref{eq:evolEquXiInterface}
by means of~$\nabla P_{\mathcal{T}_i} = \mathrm{Id} - n_{I_v}\otimes n_{I_v} - s_{\mathcal{T}_i} \nabla n_{I_v}$,
the chain rule, and the regularity of~$v$. Note carefully that this argument
is actually valid regardless of the assumption~$\mu_-=\mu_+$ since~$(\tau_{I_v}\cdot\nabla)v$
does not jump across the interface~$\mathcal{T}_i$.
\end{proof} 

\section{Extension of the interface unit normal at a $90^\circ$ contact point}
\label{sec:localCalibrationContactPoint}
This section constitutes the core of the present work. We establish
the existence of a boundary adapted extension of the interface unit normal 
in the vicinity of a space-time trajectory 
of a~$90^\circ$ contact point on the boundary~$\partial\Omega$. 

The vector field from the previous section serves
as the main building block for an extension of~$n_{I_v}$ away
from the domain boundary~$\partial\Omega$. However,
it is immediately clear that the bulk construction in general does not
respect the necessary boundary condition~$n_{\partial\Omega}\cdot\xi=0$
along~$\partial\Omega$. (Even more drastically, on non-convex parts of~$\partial\Omega$
the domain of definition for the bulk construction from the previous section may not
even include~$\partial\Omega$!) Hence, in the vicinity of contact points
a careful perturbation of the rather trivial construction from the previous section is required
to enforce the boundary condition. That this can indeed be achieved is summarized 
in the following Proposition~\ref{prop:CalibrationContactPoint}, representing the main result 
of this section.

For its formulation, it is convenient for the purposes of Section~\ref{sec:globalConstruction}
to recall the notation in relation to the decomposition of the interface~$I_v$ 
in terms of its topological features. More precisely, denoting by $N\in\mathbb{N}$ the total
number of such topological features present in the interface~$I_v$, we split
$\{1,\ldots,N\}=:\mathcal{I}\cupdot\mathcal{C}$, where~$\mathcal{I}$ enumerates 
the time-evolving connected \textit{interfaces} of~$I_v\cap\Omega$, whereas~$\mathcal{C}$ 
enumerates the time-evolving \textit{contact points} of~$I_v\cap\partial\Omega$. 
If~$i\in\mathcal{I}$, $\mathcal{T}_i:=\bigcup_{t\in [0,T]}\mathcal{T}_{i}(t){\times}\{t\}
\subset I_v\cap (\Omega{\times}[0,T])$ denotes the space-time trajectory of the 
corresponding connected interface, whereas if~$c\in\mathcal{C}$, we denote by 
$\mathcal{T}_c:=\bigcup_{t\in [0,T]}\mathcal{T}_{c}(t){\times}\{t\}
\subset I_v \cap (\partial\Omega{\times}[0,T])$ the space-time trajectory of the corresponding contact point.
Finally, we write $i\sim c$ for $i\in\mathcal{I}$ and $c\in\mathcal{C}$ if and only 
if $\mathcal{T}_i$ ends at $\mathcal{T}_c$; otherwise $i\not\sim c$.

\begin{proposition}
\label{prop:CalibrationContactPoint}
Let~${d=2}$, and let $\Omega\subset\R^2$ be a bounded domain with orientable and smooth boundary~$\partial\Omega$.
Let~$(\chi_v,v)$ be a strong solution to the incompressible
Navier--Stokes equation for two fluids in the sense of Definition~\ref{Def_strongsol}
on a time interval~$[0,T]$. 
Fix a contact point~$c\in\mathcal{C}$ and let~$i\in\mathcal{I}$ be such that~$i\sim c$. 
Let $r_c\in (0,1]$ be an associated admissible localization radius in 
the sense of Definition~\ref{def:locRadius} below.

There exists a potentially smaller radius~$\widehat r_c\in (0,r_c]$, and a vector field
\begin{align*}
\xi^{c}\colon \mathcal{N}_{\widehat r_c,c}(\Omega) \to \mathbb{S}^1
\end{align*} 
defined on the space-time domain $\mathcal{N}_{\widehat r_c,c}(\Omega)
:=\bigcup_{t\in [0,T]} \big(B_{\widehat r_c}(\mathcal{T}_c(t)) 
\cap \Omega\,\big) {\times} \{t\}$,
such that the following conditions are satisfied:
	
\begin{itemize}[leftmargin=0.7cm]
\item[i)] It holds $\xi^c \in \big(C^0_tC^2_x\cap C^1_tC^0_x \big)
					\big(\overline{\mathcal{N}_{\widehat r_c,c}(\Omega)}\setminus\mathcal{T}_c\big)$.
\item[ii)] We have $\xi^c(\cdot,t)=n_{I_v}(\cdot,t)$ and 
					$ \nabla \cdot \xi^c(\cdot,t)=- H_{I_v}(\cdot,t)$
					along $\mathcal{T}_i(t) \cap B_{\widehat r_c}(\mathcal{T}_c(t))$ 
					for all $t\in [0,T]$.
\item[iii)] The required boundary condition is satisfied even
						away from the contact point, namely $\xi^c \cdot \no=0$ 
						along $\mathcal{N}_{\widehat r_c,c}(\Omega) \cap (\partial\Omega{\times}[0,T])$.
\item[iv)] The following estimates on the time evolution of~$\xi^c$hold true in~$\mathcal{N}_{\widehat r_c,c}(\Omega)$
					\begin{align}
					\label{eq:timeEvolutionXiContactPoint}
					\partial_t\xi^c + (v\cdot\nabla)\xi^c 
					+ (\Id - \xi^c\otimes \xi^c)(\nabla v)^\mathsf{T}\xi^c 
					&= O\big(\dist(\cdot ,\mathcal{T}_i)\big),	\\
					\label{eq:timeEvolutionLengthXiContactPoint}
					\partial_t|\xi^c|^2 + (v\cdot\nabla)|\xi^c|^2 
					&= 0.
					\end{align}
\item[v)] Let~$r_i\in (0,1]$ be an admissible localization radius
					for the interface~$\mathcal{T}_i$, and let~$\xi^i$ be
					the bulk extension of the interface unit normal on scale~$r_i$
					as provided by Proposition~\ref{prop:CalibrationInterface}.
					The vector field~$\xi^c$ is a perturbation of the bulk extension~$\xi^i$
					in the sense that the following compatibility bounds hold true
\begin{align}
\label{eq:compBoundLocalCalibrations1}
|\xi^{i}(\cdot,t) - \xi^c(\cdot,t)| 
+ |\nabla\cdot\xi^{i}(\cdot,t) - \nabla\cdot\xi^c(\cdot,t)| &\leq C\dist(\cdot,\mathcal{T}_i(t)),
\\
\label{eq:compBoundLocalCalibrations2}
|\xi^{i}(\cdot,t)\cdot(\xi^{i} {-} \xi^c)(\cdot,t)| &\leq C\dist^2(\cdot,\mathcal{T}_i(t))
\end{align}
within~$B_{\widehat r_c \wedge r_i}(\mathcal{T}_c(t)) \cap 
\big(W^c_{\mathcal{T}_i}(t)\cup W^c_{\Omega^\pm_v}(t)\big)$
for all~$t\in [0,T]$, cf.\ Definition~\ref{def:locRadius}.
\end{itemize}
A vector field $\xi^{c}$ subject to these requirements
will be referred to as a \emph{contact point extension 
of the interface unit normal on scale~$\widehat r_c$}.
\end{proposition}

A proof of Proposition~\ref{prop:CalibrationContactPoint}
is provided in Subsection~\ref{sec:proofCalibrationContactPoint}.
The preceding three subsections collect all the ingredients
required for the construction. 

\subsection{Description of the geometry close to a moving contact point,
choice of orthonormal frames, and a higher-order compatibility condition} 
\label{sec:notationContactPoint}
We provide a suitable decomposition for a space-time neighborhood 
of a moving contact point~$\mathcal{T}_c$, $c\in\mathcal{C}$. 
The main ingredient is given by the following notion 
of an admissible localization radius.
Though rather technical and lengthy in appearance, all requirements in the definition
are essentially a direct consequence of the regularity of a strong solution. The main purpose of
the notion of an admissible localization radius is to collect in a unified way
notation and properties which will be referred to numerous times in the sequel.

\begin{definition}
\label{def:locRadius}
Let~${d=2}$, and let $\Omega\subset\R^2$ be a bounded domain with orientable and smooth boundary~$\partial\Omega$.
Let~$(\chi_v,v)$ be a strong solution to the incompressible
Navier--Stokes equation for two fluids in the sense of Definition~\ref{Def_strongsol}
on a time interval~$[0,T]$. Fix a contact point~$c\in\mathcal{C}$ and let~$i\in\mathcal{I}$ 
be such that~$i\sim c$. Let~$r_i\in (0,1]$ be an admissible localization radius for
the connected interface~$\mathcal{T}_i$ in the sense of Definition~\ref{def:locRadiusInterface}.
We call $r_c\in (0,r_i]$ an 
\emph{admissible localization radius for the moving~$90^\circ$ contact point~$\mathcal{T}_c$} 
if the following list of properties is satisfied:
\begin{itemize}[leftmargin=0.7cm]
\item[i)] Let the map $\Psi_{\partial\Omega}\colon \partial\Omega {\times} (-2r_c,2r_c) \to \mathbb{R}^2$
					be given by $(x,s)\mapsto x{+}sn_{\partial\Omega}(x)$. We require~$\Psi_{\partial\Omega}$
					to be bijective onto its image~$\mathrm{im}(\Psi_{\partial\Omega})
					:=\Psi_{\partial\Omega}\big(\partial\Omega {\times} (-2r_c,2r_c)\big)$,
					and its inverse $\Psi_{\partial\Omega}^{-1}$ is a diffeomorphism of class 
					$C^2_{x}(\overline{\mathrm{im}(\Psi_{\partial\Omega})})$.
					We may express the inverse by means of $\Psi_{\partial\Omega}^{-1}=:
					(P_{\partial\Omega},s_{\partial\Omega})\colon \mathrm{im}(\Psi_{\partial\Omega})
					\to\partial\Omega {\times} (-2r_c,2r_c)$. Hence,
					$P_{\partial\Omega}$ represents the \emph{nearest-point projection} 
					onto~$\partial\Omega$, whereas~$s_{\partial\Omega}$ is 
					the \emph{signed distance function} with orientation fixed 
					by~$\nabla s_{\partial\Omega} = n_{\partial\Omega}$. In particular,
					$s_{\partial\Omega}\in C^3_x(\overline{\mathrm{im}(\Psi_{\partial\Omega})})$
					and $P_{\partial\Omega}\in C^2_x(\overline{\mathrm{im}(\Psi_{\partial\Omega})})$.
					
					By a slight abuse of notation, we extend to~$\mathrm{im}(\Psi_{\partial\Omega})$ 
					the definition of the normal vector field resp.\ the scalar mean curvature of~$\partial\Omega$
					by means of 
					\begin{align}
					\label{eq:extNormalBoundary}
					n_{\partial\Omega}\colon \mathrm{im}(\Psi_{\partial\Omega})\to\mathbb{S}^{1},\quad
					&(x,t)\mapsto n_{\partial\Omega}(P_{\partial\Omega}(x)) = \nabla s_{\partial\Omega}(x),
					\\
					\label{eq:extMeanCurvatureBoundary}
					H_{\partial\Omega}\colon \mathrm{im}(\Psi_{\partial\Omega})\to\mathbb{R},\quad
					&(x,t)\mapsto -(\Delta s_{\partial\Omega})(P_{\partial\Omega}(x)).
					\end{align} 
					Hence, we note that $n_{\partial\Omega}\in 
					C^2_x(\overline{\mathrm{im}(\Psi_{\partial\Omega})})$
					and $H_{\partial\Omega} \in C^1_x(\overline{\mathrm{im}(\Psi_{\partial\Omega})})$.
\item[ii)] There exist sets $W^c_{\mathcal{T}_i}=\bigcup_{t\in[0,T]}W^c_{\mathcal{T}_i}(t){\times}\{t\}$,
					 $W_{\Omega^\pm_v}^c=\bigcup_{t\in[0,T]}W_{\Omega^\pm_v}^c(t){\times}\{t\}$ and
					 $W_{\partial\Omega}^{\pm,c}=\bigcup_{t\in[0,T]}W_{\partial\Omega}^{\pm,c}(t){\times}\{t\}$
					 with the following properties: 					
					
					 First, for every~$t\in [0,T]$, the sets~$W^c_{\mathcal{T}_i}(t)$, $W_{\Omega^\pm_v}^c(t)$
					 and~$W_{\partial\Omega}^{\pm,c}(t)$ are non-empty subsets of~$\overline{B_{r_c}(\mathcal{T}_c(t))}$
					 with pairwise disjoint interior. For all~$t\in [0,T]$, each of these sets is represented
				   by a cone with apex at the contact point~$\mathcal{T}_c(t)$ intersected with 
					 $\overline{B_{r_c}(\mathcal{T}_c(t))}$. More precisely, there exist six time-dependent pairwise distinct
					 unit-length vectors~$X^{\pm}_{\mathcal{T}_i}$, $X_{\Omega^\pm_v}$ and~$X_{\partial\Omega}^{\pm}$
					 of class~$C^1_t([0,T])$ such that for all~$t\in [0,T]$ it holds
					 \begin{align}
					 \label{eq:defWedgeInterface}
					 ~~~~W^c_{\mathcal{T}_i}(t)
					 &= \big(\mathcal{T}_c(t){+}\{\alpha X^{+}_{\mathcal{T}_i}(t) + \beta X^-_{\mathcal{T}_i}(t)\colon
						  \alpha,\beta\in [0,\infty)\}\big) \cap \overline{B_{r_c}(\mathcal{T}_c(t))},
					 \\
					 \label{eq:defWedgeInterpol}
					 ~~~~W_{\Omega^\pm_v}^c(t)
					 &= \big(\mathcal{T}_c(t){+}\{\alpha X_{\Omega^\pm_v}(t) + \beta X^\pm_{\mathcal{T}_i}(t)\colon
							\alpha,\beta\in [0,\infty)\}\big) \cap \overline{B_{r_c}(\mathcal{T}_c(t))},
					 \\
					 \label{eq:defWedgeBoundary}
					 ~~~~W_{\partial\Omega}^{\pm,c}(t)
					 &= \big(\mathcal{T}_c(t){+}\{\alpha X^\pm_{\partial\Omega}(t) + \beta X_{\Omega^\pm_v}(t)\colon
							\alpha,\beta\in [0,\infty)\}\big) \cap \overline{B_{r_c}(\mathcal{T}_c(t))}.
					 \end{align}
					 The opening angles of these cones are constant, and numerically fixed by
					 \begin{align}
					 \label{eq:openingAnglesWedges}
					 X^\pm_{\partial\Omega} \cdot X_{\Omega^\pm_v}
					 = X^+_{\mathcal{T}_i} \cdot X^-_{\mathcal{T}_i}&= \cos(\pi/3), \quad
					 X_{\Omega^\pm_v} \cdot X^\pm_{\mathcal{T}_i} = \cos(\pi/6). 
					 \end{align}
					
					 Second, for every~$t\in [0,T]$, the sets~$W^c_{\mathcal{T}_i}(t)$, $W_{\Omega^\pm_v}^c(t)$
					 and~$W_{\partial\Omega}^{\pm,c}(t)$ provide a decomposition of $B_{r_c}(\mathcal{T}_c(t))$ in form of
					 \begin{equation}
					 \begin{aligned}
					 \label{eq:decompContactPoint}
					 &\overline{B_{r_c}(\mathcal{T}_c(t))} \cap \overline{\Omega}
					 \\&
					 = \big(W^c_{\mathcal{T}_i}(t)\cup W^c_{\Omega^+_v}(t)\cup W^c_{\Omega^{-}_v}(t)
					 \cup W_{\partial\Omega}^{+,c}(t)\cup W_{\partial\Omega}^{-,c}(t)\big) 
					 \cap \overline{\Omega}.
		       \end{aligned}
					 \end{equation}
					
					 Third, for each~$t\in [0,T]$, the following inclusions hold true
					 (recall from Definition~\ref{def:locRadiusInterface} 
					 the notation for the diffeomorphism~$\Psi_{\mathcal{T}_i}$):
					 \begin{align}
					 \label{eq:interpolWedge1}
					 &\overline{B_{r_c}(\mathcal{T}_c(t))} \cap \mathcal{T}_i(t)
					 \subset \big(W^c_{\mathcal{T}_i}(t)\setminus\mathcal{T}_c(t)\big) 
					 \subset \{x\in\Omega\colon (x,t)\in\mathrm{im}(\Psi_{\mathcal{T}_i})\},
					 \\ 
					 \label{eq:interpolWedge2}
					 &\overline{B_{r_c}(\mathcal{T}_c(t))} \cap \partial\Omega
					 \subset W_{\partial\Omega}^{+,c}(t) \cup W_{\partial\Omega}^{-,c}(t),
					 \\
					 \label{eq:inclusionWedges3}
					 &W_{\partial\Omega}^{\pm,c}(t) \subset \{x\in\mathbb{R}^2\colon 
					 x\in\mathrm{im}(\Psi_{\partial\Omega})\},
					 \\&
					 \label{eq:inclusionWedges4}
			     W_{\Omega^\pm_v}^c(t)\setminus\mathcal{T}_c(t) \subset \Omega^\pm_v(t)
			     \cap \{x\in\Omega\colon (x,t)\in\mathrm{im}(\Psi_{\mathcal{T}_i}),\,
					 x \in \mathrm{im}(\Psi_{\partial\Omega})\}.
					 \end{align}
\item[iii)] Finally, there exists a constant~$C>0$ such that 
					  \begin{align} 
						\label{eq:distCondition}
					  \dist (\cdot, \mathcal{T}_c) \vee \dist (\cdot, \partial \Om )
						\leq C\dist (\cdot, \mathcal{T}_i)
						&&\text{on } W^{c}_{\Omega^\pm_v} \cup W^{\pm,c}_{\partial\Omega},
						\end{align}									
\end{itemize}
We refer from here onwards to $W^c_{\mathcal{T}_i}$ as the \emph{interface wedge}, 
$W^{\pm,c}_{\partial\Omega}$ as \emph{boundary wedges}, and $W^c_{\Omega^\pm_v}$ 
as \emph{interpolation wedges}. 
\end{definition}

\begin{figure}
	\begin{subfigure}[b]{0.4\textwidth}
	\begin{center}
		\begin{tikzpicture}
		\fill[inner color = red!30!, outer color = red!30!] (0, 0) --	(1.41,0) arc (0:30:1.41cm);
		\fill[inner color = red!30!, outer color = red!30!] (0, 0) --	(1.41,0) arc (0:-30:1.41cm);
		\draw[scale=6.5,domain=0:0.19,smooth,variable=\s,orange] plot ({\s},{(0.58)*\s	});;
		\draw[scale=6.5,domain=-0.109:0.109,smooth,variable=\s,blue] plot ({\s},{(1.73)*\s	});;
		\draw[scale=6.5,domain=0:0.19,smooth,variable=\s,orange] plot ({\s},{-(0.58)*\s	});;
		\draw[scale=6.5,domain=-0.109:0.109,smooth,variable=\s,blue] plot ({\s},{-(1.73)*\s	});;
		\node (none) at (-0.15,+1.7) {\small $\partial \Om$};
		\node (none) at (2,0) {\small $\mathcal{T}_i$};
		\node[left] (-1,0) {\small $\mathcal{T}_c$};;
		\draw[scale=6.5,domain=0:0.3,thick,smooth,variable=\s,red] plot ({\s},{(0.5)*\s*\s	});;
		\draw[dashed,thick,black!40!green] (0,0) circle (1.41cm);
		\draw[scale=4,domain=-0.42:0.42,thick,smooth,variable=\s,black] plot ({-(0.65)*\s*\s},{ \s	});;
		\node[left] at (0.04,0) [circle,fill,red,inner sep=0.7pt]  {};;
		\end{tikzpicture}
	\end{center}
	\caption{Interface wedge $W^c_{\mathcal{T}_i}$.}
	\label{Fig_InterfaceWedge}
    \end{subfigure}
 
 \begin{subfigure}[b]{0.4\textwidth}
	\begin{center}
		\begin{tikzpicture}
		\fill[inner color = blue!30!, outer color = blue!30!] (0, 0) --	(0,1.41) arc (90:120:1.41cm);
		\fill[inner color = blue!30!, outer color = blue!30!] (0, 0) --	(0,1.41) arc (90:60:1.41cm);
		\fill[inner color = blue!30!, outer color = blue!30!] (0, 0) --	(0,-1.41) arc (270:300:1.41cm);
		\fill[inner color = blue!30!, outer color = blue!30!] (0, 0) --	(0,-1.41) arc (270:240:1.41cm);
		\draw[scale=6.5,domain=0:0.19,smooth,variable=\s,orange] plot ({\s},{(0.58)*\s	});;
		\draw[scale=6.5,domain=-0.109:0.109,smooth,variable=\s,blue] plot ({\s},{(1.73)*\s	});;
		\draw[scale=6.5,domain=0:0.19,smooth,variable=\s,orange] plot ({\s},{-(0.58)*\s	});;
		\draw[scale=6.5,domain=-0.109:0.109,smooth,variable=\s,blue] plot ({\s},{-(1.73)*\s	});;
		\node (none) at (-0.15,+1.7) {\small $\partial \Om$};
		\node (none) at (2,0) {\small $\mathcal{T}_i$};
		\node[left] (-1,0) {\small $\mathcal{T}_c$};;
		\draw[scale=6.5,domain=0:0.3,thick,smooth,variable=\s,red] plot ({\s},{(0.5)*\s*\s	});;
		\draw[dashed,thick,black!40!green] (0,0) circle (1.41cm);
		\draw[scale=4,domain=-0.42:0.42,thick,smooth,variable=\s,black] plot ({-(0.65)*\s*\s},{ \s	});;
		\node[left] at (0.04,0) [circle,fill,red,inner sep=0.7pt]  {};;
		\end{tikzpicture}
	\end{center}
	\caption{Boundary wedges $W^\pm_{\partial \Om}$.}
	\label{Fig_BoundaryWedges}
\end{subfigure}

\begin{subfigure}[b]{0.4\textwidth}
	\begin{center}
		\begin{tikzpicture}
		\fill[inner color = yellow!30!, outer color = yellow!30!] (0, 0) --	(1.22,0.705) arc (30:60:1.41cm);
		\fill[inner color = yellow!30!, outer color = yellow!30!] (0, 0) --	(0.705,-1.22) arc (300:330:1.41cm);
		\draw[scale=6.5,domain=0:0.19,smooth,variable=\s,orange] plot ({\s},{(0.58)*\s	});;
		\draw[scale=6.5,domain=-0.109:0.109,smooth,variable=\s,blue] plot ({\s},{(1.73)*\s	});;
		\draw[scale=6.5,domain=0:0.19,smooth,variable=\s,orange] plot ({\s},{-(0.58)*\s	});;
		\draw[scale=6.5,domain=-0.109:0.109,smooth,variable=\s,blue] plot ({\s},{-(1.73)*\s	});;
		\node (none) at (-0.15,+1.7) {\small $\partial \Om$};
		\node (none) at (2,0) {\small $\mathcal{T}_i$};
		\node[left] (-1,0) {\small $\mathcal{T}_c$};;
		\draw[scale=6.5,domain=0:0.3,thick,smooth,variable=\s,red] plot ({\s},{(0.5)*\s*\s	});;
		\draw[dashed,thick,black!40!green] (0,0) circle (1.41cm);
		\draw[scale=4,domain=-0.42:0.42,thick,smooth,variable=\s,black] plot ({-(0.65)*\s*\s},{ \s	});;
		\node[left] at (0.04,0) [circle,fill,red,inner sep=0.7pt]  {};;
		\end{tikzpicture}
	\end{center}
	\caption{Interpolation wedges $W^c_{\Omega^\pm_v}$.}
	\label{Fig_InterpolationWedges}
\end{subfigure}
\caption{Decomposition for a space-time neighborhood of $\mathcal{T}_c$.}
\label{Fig_Wedges}
\end{figure}
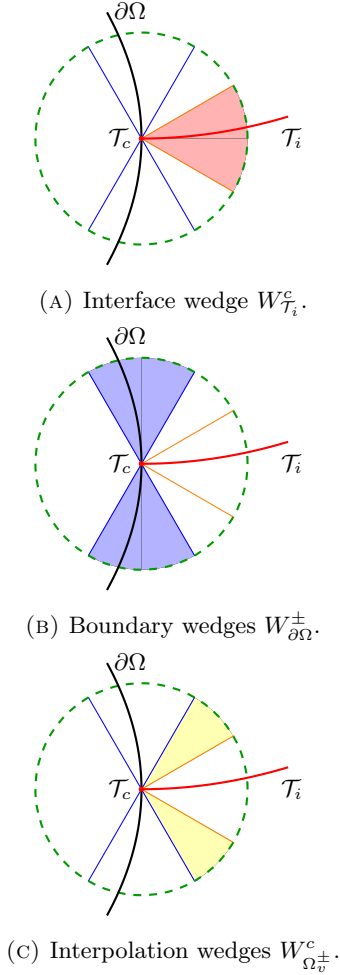

Figures \ref{Fig_Wedges}--\ref{Fig_InclusionDiffeo} contain several illustrations of the previous definition. 
Before moving on, we briefly discuss the existence of an admissible localization radius.

\begin{lemma}
\label{lemma:existenceLocRadius}
Let the assumptions and notation of Definition~\ref{def:locRadius} be in place.
There exists a constant $C=C(\partial\Omega,\chi_v,v,T)\geq 1$ such that each $r_c\in (0,\frac{1}{C}]$
is an admissible localization radius for the contact point~$\mathcal{T}_c$ in the sense
of Definition~\ref{def:locRadius}.
\end{lemma}

\begin{proof} 
The first item in the definition of an admissible localization radius
is an immediate consequence of the tubular neighborhood theorem, which in turn is
facilitated by the regularity of the domain boundary~$\partial\Omega$.
	
For a construction of the wedges, we only have to provide a definition
for the vectors~$X^{\pm}_{\mathcal{T}_i}$, $X_{\Omega^\pm_v}$ and~$X_{\partial\Omega}^{\pm}$
A possible choice is the following. Fix $t\in [0,T]$ and let~$\{c(t)\}=\mathcal{T}_c(t)$. 
The desired unit vectors are obtained through rotation of the inward-pointing unit normal 
$n_{\partial\Omega}(c(t))$. Note that 
$\big(n_{\partial\Omega}(c(t)),\,n_{I_v}(c(t),t)\big)$
form an orthonormal basis of~$\mathbb{R}^2$ thanks to the 
contact angle condition~\eqref{sangcond}. We then let~$X^\pm_{\mathcal{T}_i}(t)$ be the unique
unit vector with $X^\pm_{\mathcal{T}_i}(t)\cdot n_{\partial\Omega}(c(t))=\frac{\sqrt{3}}{2}$
as well as $\mathrm{sign}\big(X^\pm_{\mathcal{T}_i}(t)\cdot n_{I_v}(c(t),t)\big)=\pm 1$.
Similarly, $X_{\Omega^\pm_v}(t)$ represents the unique
unit vector with $X_{\Omega_v^\pm}(t)\cdot n_{\partial\Omega}(c(t))=\frac{1}{2}$
and $\mathrm{sign}\big(X_{\Omega^\pm_v}(t)\cdot n_{I_v}(c(t),t)\big)=\pm 1$. 
Finally, $X^\pm_{\partial\Omega}(t)$ denotes the unique
unit vector with $X^\pm_{\Omega}(t)\cdot n_{\partial\Omega}(c(t))=-\frac{1}{2}$
and $\mathrm{sign}\big(X^\pm_{\Omega}(t)\cdot n_{I_v}(c(t),t)\big)=\pm 1$. 
For an illustration, we refer again to Figure~\ref{Fig_Wedges}.

The wedges~$W^c_{\mathcal{T}_i}(t)$, $W_{\Omega^\pm_v}^c(t)$ and~$W_{\partial\Omega}^{\pm,c}(t)$
may now be defined through the right hand sides of~\eqref{eq:defWedgeInterface},
\eqref{eq:defWedgeInterpol} and~\eqref{eq:defWedgeBoundary}, respectively.
The properties~\eqref{eq:decompContactPoint}--\eqref{eq:distCondition} are then
obviously valid for sufficiently small radii as a consequence of the regularity of the domain
boundary~$\partial\Omega$, the regularity of the interface~$I_v$
due to Definition~\ref{Def_strongsol} of a strong solution, as well
as the~$90^\circ$ contact angle condition~\eqref{sangcond}.
\end{proof}

A main step in the construction of a contact point extension
of the interface unit normal 
consists of perturbing the bulk construction of Section~\ref{sec:localCalibrationBulk}
by introducing suitable tangential terms, cf.\ Subsection~\ref{subsec:step_1CP} 
below. (This in turn becomes necessary due to the boundary constraint 
$n_{\partial\Omega}\cdot\xi^c=0$ along~$\partial\Omega$.)
To this end, the following constructions and formulas will be of frequent use.

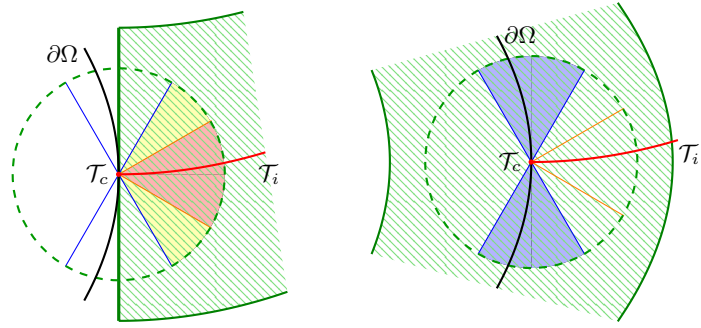
\begin{figure}
\begin{subfigure}[b]{0.4\textwidth}
	\begin{center}
		\begin{tikzpicture}
		\fill[inner color = red!30!, outer color = red!30!] (0, 0) --	(1.41,0) arc (0:30:1.41cm);
		\fill[inner color = red!30!, outer color = red!30!] (0, 0) --	(1.41,0) arc (0:-30:1.41cm);
		\fill[inner color = yellow!30!, outer color = yellow!30!] (0, 0) --	(1.22,0.705) arc (30:60:1.41cm);
		\fill[inner color = yellow!30!, outer color = yellow!30!] (0, 0) --	(0.705,-1.22) arc (300:330:1.41cm);
			\fill [pattern=north west lines, pattern color=black!25!green!50,scale=6.5,domain=0:0.255, variable=\x]
		(0, 0)
		-- plot ({\x}, {(0.5)*\x*\x+0.3})
		-- (0.3,0.045)
		-- cycle;
		\fill [pattern=north west lines, pattern color=black!25!green!50,scale=6.5,domain=0:0.345, variable=\x]
		(0,0)
		-- plot ({\x}, {(0.5)*\x*\x-0.3})
		-- (0.3,0.045)
		-- cycle;
		\draw[scale=6.5,domain=0:0.255, thick,smooth,variable=\s,black!50!green] plot ({\s},{(0.5)*\s*\s	+0.3});;
		\draw[scale=6.5,domain=0:0.345, thick,smooth,variable=\s,black!50!green] plot ({\s},{(0.5)*\s*\s	-0.3});;
		\draw[scale=6.5,very thick,smooth,variable=\s,black!50!green] (0,-0.3)--(0,+0.3);;
		\draw[scale=6.5,domain=0:0.19,smooth,variable=\s,orange] plot ({\s},{(0.58)*\s	});;
		\draw[scale=6.5,domain=-0.109:0.109,smooth,variable=\s,blue] plot ({\s},{(1.73)*\s	});;
		\draw[scale=6.5,domain=0:0.19,smooth,variable=\s,orange] plot ({\s},{-(0.58)*\s	});;
		\draw[scale=6.5,domain=-0.109:0.109,smooth,variable=\s,blue] plot ({\s},{-(1.73)*\s	});;
		\node (none) at (-0.75,1.6) {\small $\partial \Om$};
		\node (none) at (2,0) {\small $\mathcal{T}_i$};
		\node[left] (-1,0) {\small $\mathcal{T}_c$};;
		\draw[scale=6.5,domain=0:0.3,thick,smooth,variable=\s,red] plot ({\s},{(0.5)*\s*\s	});;
		\draw[dashed,thick,black!40!green] (0,0) circle (1.41cm);
		\draw[scale=4,domain=-0.42:0.42,thick,smooth,variable=\s,black] plot ({-(0.65)*\s*\s},{ \s	});;
		\node[left] at (0.04,0) [circle,fill,red,inner sep=0.7pt]  {};;
		\end{tikzpicture}
	\end{center}
	\caption{Inclusion in the image of $\Psi_{\mathcal{T}_i}$.}
	\label{Fig_InclusionDiffeoInterface}
\end{subfigure}
\begin{subfigure}[b]{0.4\textwidth}
	\begin{center}
		\begin{tikzpicture}
		\fill[inner color = blue!30!, outer color = blue!30!] (0, 0) --	(0,1.41) arc (90:120:1.41cm);
		\fill[inner color = blue!30!, outer color = blue!30!] (0, 0) --	(0,1.41) arc (90:60:1.41cm);
		\fill[inner color = blue!30!, outer color = blue!30!] (0, 0) --	(0,-1.41) arc (270:300:1.41cm);
		\fill[inner color = blue!30!, outer color = blue!30!] (0, 0) --	(0,-1.41) arc (270:240:1.41cm);
			\fill [pattern=north west lines, pattern color=black!25!green!50,scale=4,domain=-0.53:0.53, variable=\x]
		(-0.11466,-0.42)
		-- plot ({-(0.65)*\x*\x+0.47},{ \x	})
		-- (-0.11466,0.42)
		-- cycle;
		\fill [pattern=north west lines, pattern color=black!25!green!50,scale=4,domain=-0.306:0.306, variable=\x]
		(-0.11466,-0.42)
		-- plot ({-(0.65)*\x*\x-0.47},{ \x	})
		-- (-0.11466,0.42)
		-- cycle;
			\draw[scale=4,domain=-0.306:0.306, thick,smooth,variable=\x,black!50!green] plot ({-(0.65)*\x*\x-0.47},{ \x	});;
		\draw[scale=4,domain=-0.53:0.53, thick,smooth,variable=\x,black!50!green] plot ({-(0.65)*\x*\x+0.47},{ \x	});;
		\draw[scale=6.5,domain=0:0.19,smooth,variable=\s,orange] plot ({\s},{(0.58)*\s	});;
		\draw[scale=6.5,domain=-0.109:0.109,smooth,variable=\s,blue] plot ({\s},{(1.73)*\s	});;
		\draw[scale=6.5,domain=0:0.19,smooth,variable=\s,orange] plot ({\s},{-(0.58)*\s	});;
		\draw[scale=6.5,domain=-0.109:0.109,smooth,variable=\s,blue] plot ({\s},{-(1.73)*\s	});;
		\node (none) at (-0.15,+1.7) {\small $\partial \Om$};
		\node (none) at (2.1,0.1) {\small $\mathcal{T}_i$};
		\node[left] (-1,0) {\small $\mathcal{T}_c$};;
		\draw[scale=6.5,domain=0:0.3,thick,smooth,variable=\s,red] plot ({\s},{(0.5)*\s*\s	});;
		\draw[dashed,thick,black!40!green] (0,0) circle (1.41cm);
		\draw[scale=4,domain=-0.42:0.42,thick,smooth,variable=\s,black] plot ({-(0.65)*\s*\s},{ \s	});;
		\node[left] at (0.04,0) [circle,fill,red,inner sep=0.7pt]  {};;
		\end{tikzpicture}
	\end{center}
	\caption{Inclusion in the image of $\Psi_{\partial\Omega}$.}
	\label{Fig_InclusionDiffeoBuondary}
\end{subfigure}
\caption{Inclusion properties of diffeomorphisms.}
\label{Fig_InclusionDiffeo}
\end{figure}

\begin{lemma}
\label{lem:gradientsFrames}
Let the assumptions and notation of Definition~\ref{def:locRadiusInterface}
and Definition~\ref{def:locRadius} be in place. Let~$r_c$ be an admissible 
localization radius of a contact point~$\mathcal{T}_c$ and let~$i\in\mathcal{I}$
such that~$i \sim c$. Define $\mathcal{N}_{r_c,c}(\Omega):=
\bigcup_{t\in [0,T]} \big(B_{r_c}(\mathcal{T}_c(t)) 
\cap \overline{\Omega}\,\big) {\times} \{t\}$.
We fix unit-length tangential vector fields~$\widetilde\tau_{I_v}$ 
resp.\ $\widetilde\tau_{\partial\Omega}$ along $\mathcal{N}_{r_c,c}(\Omega) \cap \mathcal{T}_i$
resp.\ $\partial\Omega$ with orientation chosen such that~$\widetilde\tau_{I_v} =-n_{\partial\Omega}$
resp.\ $\widetilde\tau_{\partial\Omega}= n_{I_v}$ hold true at the contact point~$\mathcal{T}_c$.
We then define extensions
\begin{align*}
\tau_{I_v}\colon \mathcal{N}_{r_c,c}(\Omega) \cap \mathrm{im}(\Psi_{\mathcal{T}_i})
\to \mathbb{S}^1,\quad &(x,t) \mapsto \widetilde\tau_{I_v}(P_{\mathcal{T}_i}(x,t),t),
\\
\tau_{\partial\Omega}\colon \mathrm{im}(\Psi_{\partial\Omega})
\to \mathbb{S}^1,\quad &x \mapsto \widetilde\tau_{\partial\Omega}(P_{\partial\Omega}(x)), 
\end{align*}

Then, it holds $\tau_{I_v}\in C^0_tC^2_x(\overline{\mathcal{N}_{r_c,c}(\Omega) \cap \mathrm{im}(\Psi_{\mathcal{T}_i})})
\cap C^1_tC^0_x(\overline{\mathcal{N}_{r_c,c}(\Omega) \cap \mathrm{im}(\Psi_{\mathcal{T}_i})})$
as well as $\tau_{\partial\Omega}\in C^2_x(\overline{\mathrm{im}(\Psi_{\partial\Omega})})$. Moreover,
\begin{align}
\label{eq:gradientNo}
\nabla n_{I_v} &= - H_{I_v} \tau_{I_v}\otimes\tau_{I_v}
+ O(\dist(\cdot,\mathcal{T}_i))
&&\text{in } \mathcal{N}_{r_c,c}(\Omega) \cap \mathrm{im}(\Psi_{\mathcal{T}_i}),
\\
\label{eq:gradientTau}
\nabla\tau_{I_v} &= H_{I_v} n_{I_v}\otimes\tau_{I_v}
+ O(\dist(\cdot,\mathcal{T}_i))
&&\text{in } \mathcal{N}_{r_c,c}(\Omega) \cap \mathrm{im}(\Psi_{\mathcal{T}_i}).
\end{align}
Analogous formulas hold on~$\mathrm{im}(\Psi_{\partial\Omega})$
for the orthonormal frame $(n_{\partial\Omega},\tau_{\partial\Omega})$. 
\end{lemma}

\begin{proof}
By the choice of the orientations, there exists a constant matrix~$R$ 
representing rotation by~$90^\circ$ so that~$n_{I_v}=R\tau_{I_v}$
and~$n_{\partial\Omega}=R\tau_{\partial\Omega}$. The regularity of
the tangential fields~$\tau_{I_v}$ and~$\tau_{\partial\Omega}$ 
thus follows from Definition~\ref{def:locRadiusInterface}
and Definition~\ref{def:locRadius}, respectively. Moreover,
the formula~\eqref{eq:gradientTau} simply follows from~\eqref{eq:gradientNo}
and the product rule. For a proof of~\eqref{eq:gradientNo}, note first that
$(n_{I_v}\cdot\nabla)n_{I_v}=\nabla\frac{1}{2}|n_{I_v}|^2=0$
and, as a consequence of $\nabla n_{I_v} = \nabla^2 s_{\mathcal{T}_i}$
being symmetric, that $(\nabla n_{I_v})^\mathsf{T} n_{I_v}=(n_{I_v}\cdot\nabla)n_{I_v}=0$.
The only surviving component of~$\nabla n_{I_v}$ is thus the one
in the direction of~$\tau_{I_v}\otimes\tau_{I_v}$, which on
the interface in turn evaluates to~$-H_{I_v}$, see~\eqref{eq:extMeanCurvature}.
The regularity of the map~$H_{I_v}$ from Definition~\ref{def:locRadiusInterface}
then entails~\eqref{eq:gradientNo}. Of course, the exact same argument works
in terms of the orthonormal frame $(n_{\partial\Omega},\tau_{\partial\Omega})$.
\end{proof}

The values of a contact point extension in the sense 
of Proposition~\ref{prop:CalibrationContactPoint} are highly
constrained along the domain boundary~$\partial\Omega$ 
(i.e., $n_{\partial\Omega}\cdot\xi^c=0$) or along the interface~$\mathcal{T}_i$ 
(i.e., $\xi^c=n_{I_v}$), respectively. This will be reflected in the construction 
by stitching together certain local building blocks (i.e., $\xi^c_{\partial\Omega}$ 
and $\xi^c_{\mathcal{T}_i}$, see Subsection~\ref{subsec:step_1CP}  below)
which in turn take care of these restrictions on an individual
basis (i.e., $n_{\partial\Omega}\cdot\xi^c_{\partial\Omega}=0$
along~$\partial\Omega$, or $\xi^c_{\mathcal{T}_i}=n_{I_v}$ along~$\mathcal{T}_i$,
in the vicinity of the contact point). These local building blocks will be 
unified into a single vector field by interpolation 
(see Subsection~\ref{subsec:step_2CP} below).
With this in mind, it is of no surprise that 
compatibility conditions (including a higher-order one) at the contact point
are needed to implement this procedure. Indeed, recall from Proposition~\ref{prop:CalibrationContactPoint}
that a contact point extension requires a certain amount of regularity in combination 
with a control on its time evolution. We therefore collect
for reference purposes the necessary compatibility conditions in the following result.

\begin{lemma}
\label{lem:higherOrderCompConditions}
Let the assumptions and notation of Definition~\ref{def:locRadiusInterface},
Definition~\ref{def:locRadius} and Lemma~\ref{lem:gradientsFrames} be in place.
Then it holds
\begin{align}
\label{eq:compZero}
n_{I_v}(\cdot,t) = \tau_{\partial\Omega}(\cdot), \quad 
\tau_{I_v}(\cdot,t) &= - n_{\partial\Omega}(\cdot)
&&\text{at } \mathcal{T}_c(t),\,t\in[0,T],
\\
\label{eq:compHigher}
\big(\tau_{I_v}(\cdot,t)\cdot\nabla\big)(n_{I_v}\cdot v)(\cdot,t)  
&= H_{\partial\Omega}(\cdot)(n_{I_v}\cdot v)(\cdot,t)  
&&\text{at } \mathcal{T}_c(t),\,t\in[0,T].
\end{align}
\end{lemma}

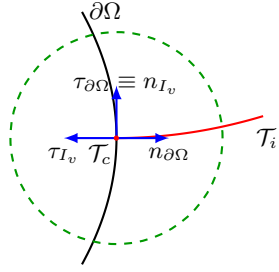
\begin{figure}
	\begin{center}
		\begin{tikzpicture}
		\draw[scale=6.5,domain=0:0.3,thick,smooth,variable=\s,red] plot ({\s},{(0.5)*\s*\s	});;
		\draw[scale=4,domain=-0.42:0.42,thick,smooth,variable=\s,black] plot ({-(0.65)*\s*\s},{ \s	});;
		\node (none) at (-0.15,+1.7) {\small $\partial \Om$};
		\node (none) at (2,0) {\small $\mathcal{T}_i$};
		\node (none) at (-0.22,-0.22) {\small $\mathcal{T}_c$};;
		\node (none) at (0.13,0.7) {\small $\tau_{\partial \Om } \equiv n_{I_v}$};;
		\node[below]  (none) at (-0.7,0) {\small $\tau_{I_v} $};
		\node[below]  (none) at (0.7,0) {\small $\no$};
		\draw[->,thick,blue,arrows={-latex}] (0,0) -- (0,0.7);
		\draw[->,thick,blue,arrows={-latex}] (0,0) -- (-0.7,0);
		\draw[->,thick,blue,arrows={-latex}] (0,0) -- (0.7,0);
		\draw[dashed,thick,black!40!green] (0,0) circle (1.41cm);	
		\node[left] at (0.04,0) [circle,fill,red,inner sep=0.7pt]  {};;
		\end{tikzpicture}
	\end{center}
	\caption{Orientation of normal and tangential vectors at $\mathcal{T}_c$.}
	\label{Fig_Normals}
\end{figure}

\begin{proof}
The relations~\eqref{eq:compZero} are immediate from
the choices made in the statement of Lemma~\ref{lem:gradientsFrames}.
Let $\{c(t)\}=\mathcal{T}_c(t)$ for all $t\in [0,T]$. The
compatibility condition~\eqref{eq:compHigher} follows
from differentiating in time the condition~$n_{I_v}(c(t),t)=\tau_{\partial\Omega}(c(t))$.
Indeed, one one side we may compute by means of the chain rule, the analogue
of~\eqref{eq:gradientTau} for~$\tau_{\partial\Omega}$, \eqref{eq:compZero}, 
and~$\frac{\mathrm{d}}{\mathrm{d}t}c(t)=\big(n_{I_v}(c(t),t)\cdot v(c(t),t)\big)n_{I_v}(c(t),t)$ 
that
\begin{align*}
\frac{\mathrm{d}}{\mathrm{d}t} \tau_{\partial\Omega}(c(t))
= H_{\partial\Omega}(c(t))\big(n_{I_v}(c(t),t)\cdot v(c(t),t)\big)n_{\partial\Omega}(c(t)).
\end{align*}
On the other side, it follows from an application of the chain rule,
the formula~\eqref{eq:gradientNo}, the previous expression of~$\frac{\mathrm{d}}{\mathrm{d}t}c(t)$,
$\partial_t s_{\mathcal{T}_i}(\cdot,t)=-n_{I_v}(\cdot,t)\cdot v(P_{\mathcal{T}_i}(\cdot,t),t)$, as well as
$n_{I_v}=\nabla s_{\mathcal{T}_i}$ that
\begin{align*}
\frac{\mathrm{d}}{\mathrm{d}t} n_{I_v}(c(t),t)
= - \big(\tau_{I_v}(c(t),t)\cdot\nabla\big)
\big(n_{I_v}\cdot v\big)(c(t),t) \tau_{I_v}(c(t),t).
\end{align*}
The second condition of~\eqref{eq:compZero} together with
the previous two displays thus implies the compatibility 
condition~\eqref{eq:compHigher} as asserted.
\end{proof}

\subsection{Construction and properties of local building blocks}
\label{subsec:step_1CP} 
We have everything in place to proceed on with the first major
step in the construction of a contact point extension in the
sense of Proposition~\ref{prop:CalibrationContactPoint}.
We define auxiliary extensions~$\xi^c_{\mathcal{T}_i}$
resp.\ $\xi^c_{\partial\Omega}$
of the unit normal vector field in the space-time 
domains~$\mathcal{N}_{r_c,c}(\Omega)\cap\mathrm{im}(\Psi_{\mathcal{T}_i})$ 
resp.\ $\mathcal{N}_{r_c,c}(\Omega)\cap(\mathrm{im}(\Psi_{\partial\Omega}){\times}[0,T])$. 
In other words, we construct the extensions separately in the regions close to the 
interface or close to the boundary (but always near to the contact point).

\subsubsection{Definition and regularity properties of local building blocks 
for the extension of the unit normal}
\label{par:ansatzXi}
A suitable ansatz for the two vector fields 
$\xi^c_{\mathcal{T}_i}$ and $\xi^c_{\partial\Omega}$ may be provided as follows.

\begin{construction}
\label{AnsatzAuxiliaryXi}
Let the assumptions and notation of Definition~\ref{def:locRadiusInterface},
Definition~\ref{def:locRadius} and Lemma~\ref{lem:gradientsFrames} be in place.
Expressing~$\{c(t)\}=\mathcal{T}_c(t)$ for all~$t\in [0,T]$,
we define coefficients
\begin{align}
\label{Condition_alpha}
\alpha_{\mathcal{T}_i}\colon \mathcal{N}_{r_c,c}(\Omega)\cap\mathrm{im}(\Psi_{\mathcal{T}_i}) \to \mathbb{R},
\quad &(x,t) \mapsto -H_{\partial\Omega}(c(t),t),
\\
\label{Condition_alphaprime}
\alpha_{\partial\Omega}\colon \mathcal{N}_{r_c,c}(\Omega)
\cap(\mathrm{im}(\Psi_{\partial\Omega}){\times}[0,T]) \to \mathbb{R},
\quad &(x,t) \mapsto -H_{I_v}(c(t),t).
\end{align}
Based on these coefficient functions, we then define extensions 
\begin{align*}
\xi^c_{\mathcal{T}_i}\colon \mathcal{N}_{r_c,c}(\Omega)\cap\mathrm{im}(\Psi_{\mathcal{T}_i}) \to \mathbb{R}^2, \quad
\xi^c_{\partial \Om}\colon \mathcal{N}_{r_c,c}(\Omega)\cap
\big(\mathrm{im}(\Psi_{\partial\Omega}){\times}[0,T]\big) \to \mathbb{R}^2
\end{align*}
of the normal vector field~$n_{I_v}$ by means of an expansion ansatz
\begin{align}
\label{eq:AnsatzAuxiliaryXiInterface}
\xi^c_{\mathcal{T}_i} &:= n_{I_v} + \alpha_{\mathcal{T}_i}s_{\mathcal{T}_i}\tau_{I_v}
- \frac{1}{2}\alpha_{\mathcal{T}_i}^2s_{\mathcal{T}_i}^2n_{I_v},
\\
\label{eq:AnsatzAuxiliaryXiBuondary}
\xi^c_{\partial \Om} &:= \tau_{\partial\Omega} + \alpha_{\partial\Omega}s_{\partial\Omega}n_{\partial\Omega}
- \frac{1}{2}\alpha_{\partial\Omega}^2s_{\partial\Omega}^2\tau_{\partial\Omega}.
\end{align}
\end{construction}	

Regularity properties of $\xi^c_{\mathcal{T}_i}$ and $\xi^c_{\partial\Omega}$,
in particular compatibility up to first order at the contact point, 
are the content of the following result.

\begin{lemma} 
\label{lem:firstOrderCompXi}
Let the assumptions and notation of Construction \ref{AnsatzAuxiliaryXi} be in place. 
Then the auxiliary vector fields satisfy $\xi^c_{\mathcal{T}_i}\in (C^0_tC^2_x\cap C^1_tC^0_x)
(\overline{\mathcal{N}_{r_c,c}(\Omega)\cap\mathrm{im}(\Psi_{\mathcal{T}_i})})$
and $\xi^c_{\partial\Omega}\in (C^0_tC^2_x\cap C^1_tC^0_x)
(\overline{\mathcal{N}_{r_c,c}(\Omega)\cap(\mathrm{im}(\Psi_{\partial\Omega}){\times}[0,T])})$,
with corresponding estimates for $k \in \{0, 1, 2\}$
\begin{align}
\label{eq:reqXiInterface}
|\nabla^k\xi^c_{\mathcal{T}_i}| + |\partial_t\xi^c_{\mathcal{T}_i}| &\leq C,
\quad\text{on } \overline{\mathcal{N}_{r_c,c}(\Omega)\cap\mathrm{im}(\Psi_{\mathcal{T}_i})},
\\
\label{eq:reqXiBoundary}
|\nabla^k\xi^c_{\partial\Omega}| + |\partial_t\xi^c_{\partial\Omega}| &\leq C,
\quad\text{on } \overline{\mathcal{N}_{r_c,c}(\Omega)\cap(\mathrm{im}(\Psi_{\partial\Omega}){\times}[0,T])}.
\end{align}	
Moreover, the constructions are compatible to first order at the contact point
in the sense that
\begin{align}
\label{eq:compatibilityXiContactPoint}
\xi^c_{\mathcal{T}_i} (\cdot,t)= \xi^c_{\partial \Om} (\cdot,t), \quad
\nabla \xi^c_{\mathcal{T}_i} (\cdot,t)= \nabla \xi^c_{\partial \Om} (\cdot,t)
\quad\text{at } \mathcal{T}_c(t),\,t\in[0,T].
\end{align}
\end{lemma}

\begin{proof}
\textit{Step 1 (Regularity estimates):} 
Note first that~$\alpha_{\mathcal{T}_i},\alpha_{\partial\Omega}\in C^1_t([0,T])$
due to the regularity of the maps~$H_{I_v}$ resp.\ $H_{\partial\Omega}$
from~\eqref{eq:extMeanCurvature} resp.\ \eqref{eq:extMeanCurvatureBoundary}.
The asserted bounds~\eqref{eq:reqXiInterface} and~\eqref{eq:reqXiBoundary} for the derivatives of the vector 
fields $\xi^c_{\mathcal{T}_i}$ and $\xi^c_{\partial \Om}$ can thus 
be inferred from the definitions~\eqref{eq:AnsatzAuxiliaryXiInterface} and~\eqref{eq:AnsatzAuxiliaryXiBuondary}
in combination with the regularity of~$s_{\mathcal{T}_i},n_{I_v}$ from Definition~\ref{def:locRadiusInterface},
the regularity of~$s_{\partial\Omega},n_{\partial\Omega}$ from Definition~\ref{def:locRadius},
as well as the regularity of~$\tau_{I_v},\tau_{\partial\Omega}$ from Lemma~\ref{lem:gradientsFrames}. 
 
\textit{Step 2 (First order compatibility at the contact point):} 
The zeroth order condition of~\eqref{eq:compatibilityXiContactPoint} is a 
direct consequence of the definitions~\eqref{eq:AnsatzAuxiliaryXiInterface} 
and~\eqref{eq:AnsatzAuxiliaryXiBuondary} in combination with the
compatibility condition~\eqref{eq:compZero}. In order to prove the
first order condition, it directly follows from~\eqref{eq:gradientNo}--\eqref{eq:gradientTau}
and their analogues for the frame~$(n_{\partial\Omega},\tau_{\partial\Omega})$,
as well as the definitions~\eqref{eq:AnsatzAuxiliaryXiInterface} 
and~\eqref{eq:AnsatzAuxiliaryXiBuondary} that
\begin{align}
\label{eq:gradientXiInterface}
\nabla \xi^c_{\mathcal{T}_i}  &=  
- H_ {I_v} \tau_{I_v} \otimes \tau_{I_v} + \alpha_{\mathcal{T}_i} \tau_{I_v} \otimes n_{I_v} 
+ O(\dist(\cdot,\mathcal{T}_i)),
\\
\label{eq:gradientXiBoundary}
\nabla \xi^c_{\partial \Om} 
&=  H_{\partial \Om}  n_{\partial \Om}  \otimes \tau_{\partial \Om}  
+ \alpha_{\partial \Om} n_{\partial \Om}  \otimes \no + O(\dist(\cdot,\partial\Omega)).
\end{align}
Finally, since we have~\eqref{eq:compZero} due to the conventions 
adopted, using~\eqref{Condition_alpha} and~\eqref{Condition_alphaprime}
we can deduce the first order compatibility condition of~\eqref{eq:compatibilityXiContactPoint}.
\end{proof}  

\subsubsection{Evolution equations for local building blocks} 
The following lemma provides the approximate evolution equations
for our local constructions~$\xi^c_{\mathcal{T}_i}$ and~$\xi^c_{\partial \Om}$, 
which will eventually lead us to~\eqref{eq:timeEvolutionXiContactPoint}--\eqref{eq:timeEvolutionLengthXiContactPoint}.

\begin{lemma} 
\label{lemma:PropertiesAuxPairs}
Let the assumptions and notation of Construction~\ref{AnsatzAuxiliaryXi} be in place.
	Then it holds
	\begin{align}
	\label{eq:evolEquAuxXiInterface}
	\partial_t \xi^c_{\mathcal{T}_i}
	+ (v\cdot\nabla) \xi^c_{\mathcal{T}_i}
	+ (\Id - \xi^c_{\mathcal{T}_i} \otimes \xi^c_{\mathcal{T}_i})
	(\nabla v)^\mathsf{T} \xi^c_{\mathcal{T}_i}
	&= O(\dist(\cdot,\mathcal{T}_i)),
	\\
	\label{eq:evolEquLengthAuxXiInterface} 
	\partial_t|\xi^c_{\mathcal{T}_i}|^2
	+ (v\cdot\nabla)|\xi^c_{\mathcal{T}_i}|^2
	&  = O(\dist^3(\cdot,\mathcal{T}_i)),
	\\
	\label{eq:lengthControlAuxXiInterface}
	|1-|\xi^c_{\mathcal{T}_i}|^2| 
	&  = O(\dist^4(\cdot,\mathcal{T}_i))
	\end{align}
	throughout the space-time domain $\mathcal{N}_{r_c,c}(\Omega)\cap\mathrm{im}(\Psi_{\mathcal{T}_i})$.
	Moreover, we have
	\begin{align}
	\label{eq:evolEquAuxXiBuondary}
	\partial_t \xi^c_{\partial \Om}
	{+} (v\cdot\nabla) \xi^c_{\partial \Om}
	{+} (\Id - \xi^c_{\partial \Om} \otimes \xi^c_{\partial \Om}) 
	(\nabla v)^\mathsf{T} \xi^c_{\partial \Om}
	&= O(\dist(\cdot,\partial \Om)\vee\dist(\cdot,\mathcal{T}_c)),
	\\
	\label{eq:evolEquLengthAuxXiBuondary}
	\partial_t|\xi^c_{\partial \Om}|^2
	+ (v\cdot\nabla)|\xi^c_{\partial \Om}|^2
	& = O(\dist^3(\cdot,\partial \Om)),
	\\
	\label{eq:lengthControlAuxXiBuondary}
	|1-|\xi^c_{\partial \Om}|^2| 
	& = O(\dist^4(\cdot,\partial \Om))
	\end{align}
	throughout the space-time domain $\mathcal{N}_{r_c,c}(\Omega)\cap
\big(\mathrm{im}(\Psi_{\partial\Omega}){\times}[0,T]\big)$.
\end{lemma}

\begin{proof}
		\textit{Step 1 (Proof of~\eqref{eq:evolEquAuxXiInterface}):}
		Note that because of the definitions~\eqref{eq:DefXiInterface} and~\eqref{eq:AnsatzAuxiliaryXiInterface},
		it holds $\xi^c_{\mathcal{T}_i} = \xi^i + \alpha_{\mathcal{T}_i}s_{\mathcal{T}_i}\tau_{I_v}
	 - \frac{1}{2}\alpha_{\mathcal{T}_i}^2s_{\mathcal{T}_i}^2n_{I_v}$. Since 
		we already proved~\eqref{eq:evolEquXiInterface}, we only need to show that
		\begin{align*}
			\alpha_{I_v} (\partial_t s_{\mathcal{T}_i})\tau_{I_v} 
			+ \alpha_{I_v}  (v\cdot\nabla s_{\mathcal{T}_i}) \tau_{I_v}
			= O(\dist(\cdot,\mathcal{T}_i)).
		\end{align*}
		However, the above relation is an immediate consequence of the identity $\partial_t s_{\mathcal{T}_i}(x,t) 
		= -\big(v(P_{\mathcal{T}_i}(x,t),t)\cdot\nabla\big) s_{\mathcal{T}_i}(x,t)$ and
		the regularity of~$v$, see Definition~\ref{Def_strongsol} of a strong solution,
		through a Lipschitz estimate. This proves~\eqref{eq:evolEquAuxXiInterface}.
		
		\textit{Step 2 (Proof of~\eqref{eq:evolEquAuxXiBuondary}):}
			From the definition~\eqref{eq:AnsatzAuxiliaryXiBuondary} 
			and~$\alpha_{\partial\Omega}\in C^1_t([0,T])$ it directly follows
			\begin{align*}
				\partial_t \xi^c_{\partial \Om} 
				= (\partial_t \alpha_{\partial \Om}) s_{\partial \Om }n_{\partial \Om }
				= O(\dist(\cdot, \partial \Om) ). 
			\end{align*}

			Having~$\xi^c_{\partial \Om} = \tau_{\partial \Om } + \alpha_{\partial \Om} s_{\partial \Om } \no 
			- \frac{1}{2}\alpha_{\partial\Omega}^2s_{\partial\Omega}^2\tau_{\partial\Omega}$, 
			cf.\ the definition~\eqref{eq:AnsatzAuxiliaryXiBuondary}, it follows from
			$\nabla s_{\partial \Om} = n_{\partial \Om} $, the analogues of~\eqref{eq:gradientNo}--\eqref{eq:gradientTau}
		  for the frame~$(n_{\partial\Omega},\tau_{\partial\Omega})$, as well as the boundary 
			condition~$v\cdot n_{\partial\Omega}=0$ along~$\partial\Omega$ that
			\begin{align*}
			(v \cdot \nabla ) \xi^c_{\partial \Om} &= 	
			(v \cdot \nabla ) (\tau_{\partial \Om } + \alpha_{\partial \Om} s_{\partial \Om } \no ) 
			+  O(\dist(\cdot, \partial \Om))
			\\&
			= (v\cdot\tau_{\partial\Omega}) \tau_{\partial \Om } \cdot (H_{\partial \Om}   \tau_{\partial \Om } \otimes \no 
			+ \alpha_{\partial \Om} \no \otimes \no )    + O(\dist(\cdot, \partial \Om) )
			\\&
			= (v\cdot\tau_{\partial\Omega}) H_{\partial \Om}  \no + O(\dist(\cdot, \partial \Om)).
			\end{align*}
			Moreover, based on~$\xi^c_{\partial\Omega}=\tau_{\partial\Omega} + O(\dist(\cdot,\partial\Omega))$
			due to~\eqref{eq:AnsatzAuxiliaryXiBuondary}, $v(c(t),t)=\big(v(c(t),t)\cdot n_{I_v}(c(t),t)\big)
			n_{I_v}(c(t),t)$ along the moving contact point~$\{c(t)\}=\mathcal{T}_c(t)$, the formula~\eqref{eq:gradientNo},
			and the compatibility conditions~\eqref{eq:compZero}--\eqref{eq:compHigher} we infer that
			\begin{align*}
			&(\Id - \xi^c_{\partial \Om} \otimes \xi^c_{\partial \Om})
			(\nabla v)^\mathsf{T}  \xi^c_{\partial \Om} 
			\\
			&= (\Id - \tau_{\partial\Omega} \otimes \tau_{\partial\Omega})
			(\nabla v)^\mathsf{T}  \tau_{\partial\Omega} + O(\dist(\cdot, \partial \Om))
			\\
			&= (\tau_{\partial\Omega}\cdot(n_{\partial\Omega}\cdot\nabla) v) \no
			+ O(\dist(\cdot, \partial \Om))
			\\
			&= - \big(n_{I_v}(c(t),t)\cdot\big(\tau_{I_v}(c(t),t)\cdot\nabla\big) v(c(t),t)\big) \no
			+ O(\dist(\cdot, \partial \Om) \vee \dist(\cdot,\mathcal{T}_c))
			\\
			&= -\big(\big(\tau_{I_v}(c(t),t)\cdot\nabla\big)(v\cdot n_{I_v})(c(t),t) \big)\no
			+ O(\dist(\cdot, \partial \Om) \vee \dist(\cdot,\mathcal{T}_c))
			\\
			&= -(v\cdot\tau_{\partial\Omega}) H_{\partial \Om}  \no 
			+ O(\dist(\cdot, \partial \Om) \vee \dist(\cdot,\mathcal{T}_c)).
			\end{align*}
			Hence, the estimate~\eqref{eq:evolEquAuxXiBuondary}
			follows as a consequence of the previous three displays.
			
		\textit{Step 3 (Proof of~\eqref{eq:evolEquLengthAuxXiInterface}--\eqref{eq:lengthControlAuxXiInterface}
		and~\eqref{eq:evolEquLengthAuxXiBuondary}--\eqref{eq:lengthControlAuxXiBuondary}):} 
		Simply note that~\eqref{eq:evolEquLengthAuxXiInterface}--\eqref{eq:lengthControlAuxXiInterface}
		as well as~\eqref{eq:evolEquLengthAuxXiBuondary}--\eqref{eq:lengthControlAuxXiBuondary}
		directly follow from the definitions~\eqref{eq:AnsatzAuxiliaryXiInterface} resp.\
		\eqref{eq:AnsatzAuxiliaryXiBuondary} of the vector field $\xi^c_{\mathcal{T}_i}$ resp.\ 
		the vector field~$\xi^c_{\partial\Omega}$	in form of
		\begin{align}
		|\xi^c_{\mathcal{T}_i}|^2 &= \Big(1 {-} \frac{1}{2}\alpha_{\mathcal{T}_i}^2s_{\mathcal{T}_i}^2\Big)^2
		+ \alpha_{\mathcal{T}_i}^2s_{\mathcal{T}_i}^2
		= 1 + \frac{1}{4}\alpha_{\mathcal{T}_i}^4s_{\mathcal{T}_i}^4,
		\\
		|\xi^c_{\partial\Omega}|^2 &= \Big(1 {-} \frac{1}{2}\alpha_{\partial\Omega}^2s_{\partial\Omega}^2\Big)^2
		+ \alpha_{\partial\Omega}^2s_{\partial\Omega}^2
		= 1 + \frac{1}{4}\alpha_{\partial\Omega}^4s_{\partial\Omega}^4.
		\end{align}
		This concludes the proof of Lemma~\ref{lemma:PropertiesAuxPairs}.
\end{proof}  

\subsection{From building blocks to contact point extensions by interpolation}
\label{subsec:step_2CP}
As we discussed in the previous subsections, the auxiliary vector fields~$\xi^c_{\mathcal{T}_i}$ and 
$\xi^c_{\partial \Omega}$ provide main building block for a contact point extension
of the interface unit normal near the connected interface~$\mathcal{T}_i$ or
near the domain boundary~$\partial\Omega$, respectively.
More precisely, we will make use of the auxiliary vector field~$\xi^c_{\mathcal{T}_i}$
on the wedges~$W^c_{\mathcal{T}_i}\cup W^c_{\Omega^+_v}\cup W^c_{\Omega^-_v}$,
and of the auxiliary vector field~$\xi^c_{\partial\Omega}$
on the wedges~$W^{+,c}_{\partial\Omega}\cup W^{-,c}_{\partial\Omega}
\cup W^c_{\Omega^+_v}\cup W^c_{\Omega^-_v}$. Note that this is indeed
admissible thanks to the inclusions~\eqref{eq:interpolWedge1},
\eqref{eq:inclusionWedges3} and~\eqref{eq:inclusionWedges4}.
As the domains of definition for the auxiliary vector fields overlap, 
we adopt an interpolation procedure on the interpolation wedges~$W^c_{\Om_v^\pm}$.
To this end, we first define suitable interpolation functions.

\begin{lemma}
\label{lemma:existenceInterpolationFunction}
Let the assumptions and notation of Definition~\ref{def:locRadius} be in place.
	Then there exists a pair of interpolation functions
	\begin{align*}
	\lambda^\pm_c\colon\bigcup_{t\in [0,T]} 
	\big(W^c_{\Om_v^\pm}(t)\setminus \mathcal{T}_c(t)\big)
	{\times}\{t\}\to [0,1]
	\end{align*}
	which satisfies the following list of properties:
	\begin{itemize}[leftmargin=0.7cm]
		\item[i)] On the boundary of the interpolation wedges~$W^c_{\Om_v^\pm}$
							intersected with~$B_{r_c}(\mathcal{T}_c)$, the values of~$\lambda^\pm_c$ and 
							its derivatives up to second order are given by
		\begin{align}
		\lambda^\pm_c(\cdot,t) &= 0 &&\textnormal{ on } 
		\big(\partial W^c_{\Om_v^\pm}(t)\cap\partial W^{\pm,c}_{\partial \Om }(t)\big) 
		\setminus\mathcal{T}_c(t), \label{eq:LambdaProp0}
		\\
		\lambda^\pm_c(\cdot,t) & = 1 &&\textnormal{ on }
		\big(\partial W^c_{\Om_v^\pm}(t) \cap\partial W^c_{\mathcal{T}_i}(t)\big)
		\setminus\mathcal{T}_c(t), \label{eq:LambdaProp1}
		\\ 
		\nabla\lambda^\pm_c(\cdot,t) & = 0, 
		&&\textnormal{ on } 	\big(\partial W^c_{\Om_v^\pm}(t) \cap B_{r_c}(\mathcal{T}_c(t))\big)
		\setminus\mathcal{T}_c(t),
		\label{first_derivative_lambda_vanishes_on_wedge}
		\\ 
		\nabla^2\lambda^\pm_c(\cdot,t) &  = 0,  \,\,	\partial_t\lambda^\pm_c(\cdot,t) = 0 
		&&\textnormal{ on }	\big(\partial W^c_{\Om_v^\pm}(t)\cap B_{r_c}(\mathcal{T}_c(t))\big)
		\setminus\mathcal{T}_c(t)
		\label{second_derivative_lambda_vanishes_on_wedge}
		\end{align}
		for all $t\in [0,T]$.
		\item[ii)] There exists a constant $C$ such that the estimates
		\begin{align}
		|\partial_t\lambda^\pm_c(\cdot,t)| + |\nabla\lambda^\pm_c(\cdot,t)|
		&\leq C|\dist(\cdot,\mathcal{T}_c(t))|^{-1}, \label{eq:LambdaFirstDeriv}
		\\
		|\nabla\partial_t\lambda^\pm_c(\cdot,t)| + |\nabla^2\lambda^\pm_c(\cdot,t)|
		& \leq C|\dist(\cdot,\mathcal{T}_c(t))|^{-2} \label{eq:LambdaSecondDeriv}
		\end{align}
	  hold true on~$W^c_{\Om_v^\pm}(t)\setminus \mathcal{T}_c(t)$ for all $t\in [0,T]$.
\item[iii)] We have an improved estimate on the advective derivative in form of
		\begin{align} \label{eq:AdvectiveDerivativeLambda}
		\big|\partial_t\lambda^\pm_c(\cdot,t) + \big(v\cdot\nabla\big)\lambda^\pm_c(\cdot,t)\big|	\leq C
		\end{align}
		on~$W^c_{\Om_v^\pm}(t)\setminus \mathcal{T}_c(t)$ for all $t\in [0,T]$.
	\end{itemize}
\end{lemma}

	\begin{proof}	
We fix a smooth function $\widetilde\lambda\colon\mathbb{R}\to [0,1]$
such that $\widetilde\lambda\equiv 0$ on $[\frac{2}{3},\infty)$
and $\widetilde\lambda\equiv 1$ on $(-\infty,\frac{1}{3}]$.
Recall the representation~\eqref{eq:defWedgeInterpol} of
the interpolation wedges~$W_{\Omega^\pm_v}$, and that their opening angle
is determined via $X^\pm_{\mathcal{T}_i}\cdot X_{\Omega^\pm_v} = \cos(\pi/6)$ along~$\mathcal{T}_c$, 
see~\eqref{eq:openingAnglesWedges}. We then define a function $\lambda\colon [-1,1]\to [0,1]$
by $\lambda(u):=\widetilde\lambda(\frac{1{-}u}{1{-}\cos(\pi/6)})$, and set
\begin{align*}
\lambda^\pm_c(x,t) := \lambda\bigg(X_{\mathcal{T}_i}^\pm(t)
\cdot\frac{x{-}c(t)}{|x{-}c(t)|}\bigg),
\quad t\in[0,T],\,x\in W_{\Omega^{\pm}_v}(t)\setminus\mathcal{T}_c(t).
\end{align*}
The assertions of the first two items of Lemma~\ref{lemma:existenceInterpolationFunction}
are now immediate consequences of the definitions due to~$\frac{\mathrm{d}}{\mathrm{d}t}X^{\pm}_{\mathcal{T}_i}\in C^0([0,T])$,
cf.\ Definition~\ref{def:locRadius}.

It remains to prove the estimate~\eqref{eq:AdvectiveDerivativeLambda} on the advective derivative.
To this end, abbreviating~$u^\pm:=X_{\mathcal{T}_i}^\pm(t)
\cdot\frac{x{-}c(t)}{|x{-}c(t)|}$ we compute
\begin{align*}
\partial_t \lambda_c^\pm(x,t)
&= \lambda'(u^\pm) X_{\mathcal{T}_i}^\pm(t)
\cdot \partial_t \frac{x{-}c(t)}{|x{-}c(t)|}
+ \lambda'(u^\pm) \frac{x{-}c(t)}{|x{-}c(t)|} 
\cdot \frac{\mathrm{d}}{\mathrm{d}t} X_{\mathcal{T}_i}^\pm(t)
\\
&= \lambda'(u^\pm) 
X_{\mathcal{T}_i}^\pm(t)\cdot\frac{1}{|x{-}c(t)|}
\Big( \Id {-} \frac{x{-}c(t)}{|x{-}c(t)|} 
\otimes \frac{x{-}c(t)}{|x{-}c(t)|}  \Big)
\frac{\mathrm{d}}{\mathrm{d}t}c(t) 
\\&~~~
+ \lambda'(u^\pm) \frac{x{-}c(t)}{|x{-}c(t)|} 
\cdot \frac{\mathrm{d}}{\mathrm{d}t} X_{\mathcal{T}_i}^\pm(t)
\\&
=-\Big(\frac{\mathrm{d}}{\mathrm{d}t} c(t) \cdot \nabla\Big)\lambda_c^\pm(x,t)
+ \lambda'(u^\pm) \frac{x{-}c(t)}{|x{-}c(t)|} 
\cdot \frac{\mathrm{d}}{\mathrm{d}t} X_{\mathcal{T}_i}^\pm(t).
\end{align*}
This in turn yields the asserted estimate \eqref{eq:AdvectiveDerivativeLambda} 
due to~$\frac{\mathrm{d}}{\mathrm{d}t}X^{\pm}_{\mathcal{T}_i}\in C^0([0,T])$,
cf.\ Definition~\ref{def:locRadius}, $\frac{\mathrm{d}}{\mathrm{d}t} c(t)
= v(c(t),t)$, and a Lipschitz estimate based on the regularity of the
fluid velocity~$v$ from Definition~\ref{Def_strongsol} (which counteracts
the blow-up~\eqref{eq:LambdaFirstDeriv} of~$\nabla \lambda_c^\pm$).
This concludes the proof.
\end{proof} 

We have by now everything in place to state the definition of a
vector field which in the end will give rise to 
a contact point extension of the interface unit normal in the precise sense of 
Proposition~\ref{prop:CalibrationContactPoint}.

\begin{construction}
\label{const:ansatzContactPointExtension}
Let the assumptions and notation of Definition~\ref{def:locRadius},
Construction~\ref{AnsatzAuxiliaryXi} and Lemma~\ref{lemma:existenceInterpolationFunction}
be in place. In particular, let~$r_c\in (0,1]$ be an admissible localization
radius for the contact point~$\mathcal{T}_c$.
We define a vector field
\begin{align*}
\widehat\xi^{c}\colon \mathcal{N}_{r_c,c}(\Omega) \to \mathbb{R}^2
\end{align*}
on the space-time domain~$\mathcal{N}_{r_c,c}(\Omega)
:=\bigcup_{t\in [0,T]} \big(B_{r_c}(\mathcal{T}_c(t)) 
\cap \Omega\,\big) {\times} \{t\}$ 
as follows (recall the decomposition~\eqref{eq:decompContactPoint}
of the neighborhood $B_r(\mathcal{T}_c(t))\cap \overline\Omega$):
\begin{align}
\label{eq:ansatzXiContactPoint}
\widehat\xi^{c}(\cdot,t) &:=
\begin{cases}
\xi^c_{\mathcal{T}_i}(\cdot,t) & \text{on } W^c_{\mathcal{T}_i}(t) \cap \overline\Omega, \\
\xi^c_{\partial\Omega}(\cdot,t) & \text{on } W^{\pm,c}_{\partial\Omega}(t) \cap \overline\Omega, \\
\lambda^\pm_c(\cdot,t)\xi^c_{\mathcal{T}_i}(\cdot,t) 
+ \big(1{-}\lambda^\pm_c(\cdot,t)\big)\xi^c_{\partial\Omega}(\cdot,t)
& \text{on } W_{\Omega^{\pm}_v}(t)\setminus\mathcal{T}_c(t) \cap \overline\Omega, \\ 
\end{cases}
\end{align}
for all~$t\in [0,T]$. Note that the vector field~$\widehat\xi^{c}$ is not yet normalized to unit length,
which is the reason for denoting it by~$\widehat\xi^{c}$ instead of $\xi^{c}$. 
Observe also that~\eqref{eq:ansatzXiContactPoint} is well-defined in view of the 
inclusions~\eqref{eq:interpolWedge1}, \eqref{eq:inclusionWedges3} and~\eqref{eq:inclusionWedges4}.
\end{construction}

\subsection{Proof of Proposition~\ref{prop:CalibrationContactPoint}} 
\label{sec:proofCalibrationContactPoint} 
The proof proceeds in several steps. We first establish the required properties 
in terms of the vector field~$\widehat\xi^c$. The penultimate step is devoted to fixing $\widehat r_c\in (0,r_c]$
such that~$\big|\widehat\xi^c\big|\geq \frac{1}{2}$ on~$\mathcal{N}_{\widehat r_c,c}(\Omega)$,
so that one may define $\xi:=\big|\widehat\xi^c\big|^{-1}\widehat\xi^c\in\mathbb{S}^1$
throughout~$\mathcal{N}_{\widehat r_c,c}(\Omega)$ and transfer the properties of~$\widehat\xi^c$ to~$\xi^c$.
Finally, in the last step we verify the asserted compatibility conditions between
a contact point extension and a bulk extension of the interface unit normal.

\textit{Step 1: Regularity of~$\widehat \xi^c$ and properties i)--iii).} 	
Because of the inclusion~\eqref{eq:interpolWedge1} as well as
the definitions~\eqref{eq:AnsatzAuxiliaryXiInterface} and~\eqref{eq:ansatzXiContactPoint},
it follows that $\widehat \xi^c (\cdot,t) = n_{I_v}(\cdot,t)$
along $\mathcal{T}_i(t) \cap B_{r_c}(\mathcal{T}_c(t))$ 
for all $t\in [0,T]$. By the same reasons, relying also on
$\xi^c_{\mathcal{T}_i} = \xi^i + \alpha_{\mathcal{T}_i}s_{\mathcal{T}_i}\tau_{I_v}
- \frac{1}{2}\alpha_{\mathcal{T}_i}^2s_{\mathcal{T}_i}^2n_{I_v}$,
cf.\  the definitions~\eqref{eq:DefXiInterface} and~\eqref{eq:AnsatzAuxiliaryXiInterface},
$\nabla s_{\mathcal{T}_i}=n_{I_v}$ and~\eqref{eq:divMeanCurvature},
we deduce that~$\nabla\cdot\widehat \xi^c (\cdot,t) = - H_{I_v}(\cdot,t)$
along $\mathcal{T}_i(t) \cap B_{r_c}(\mathcal{T}_c(t))$ 
for all $t\in [0,T]$. Moreover, in view of the inclusion~\eqref{eq:interpolWedge2}
as well as the definitions~\eqref{eq:AnsatzAuxiliaryXiBuondary} and~\eqref{eq:ansatzXiContactPoint},
we obtain $\widetilde \xi^c(\cdot,t)\cdot n_{\partial\Omega}=\tau_{\partial\Omega}\cdot n_{\partial\Omega}=0$
along $B_{r_c}(\mathcal{T}_c(t))\cap\partial\Omega$.
This yields the asserted properties i)--iii) of a contact point extension in terms of~$\widehat\xi^c$
on scale~$r_c$. 

The vector fields $\widehat\xi^c$, $\partial_t\widehat\xi^c$, $\nabla\widehat\xi^c$ and $\nabla^2\widehat\xi^c$
exist in a pointwise sense and are continuous throughout~$\mathcal{N}_{r_c,c}(\Omega)\setminus\mathcal{T}_c$
due to the definition~\eqref{eq:ansatzXiContactPoint} of~$\widehat\xi^c$,
the regularity of the local building blocks~$\xi^c_{\mathcal{T}_i}$ and~$\xi_{\partial\Omega}^c$
as provided by Lemma~\ref{lem:firstOrderCompXi}, as well as the regularity of the
interpolation parameter~$\lambda^\pm_c$ from Lemma~\ref{lemma:existenceInterpolationFunction}.
Note in this context that no jumps occur across the boundaries of the interpolation wedges
as a consequence of the conditions~\eqref{eq:LambdaProp0}--\eqref{second_derivative_lambda_vanishes_on_wedge}.
It remains to prove the bounds
\begin{align}
\label{eq:regularityEstimateAnsatzXiContactPoint}
|\partial_t\widehat\xi^c(\cdot,t)| + |\nabla^k\widehat\xi^c(\cdot,t)|
\leq C \quad\text{on } \big(\overline{B_{r_c}(\mathcal{T}_c(t))}
\setminus\mathcal{T}_c\big) \cap \overline{\Omega}
\end{align}
for $k \in \{0,1,2\}$, for all $t\in [0,T]$ and some constant~$C>0$.

In the wedges~$W^c_{\mathcal{T}_i}$ and $W^{\pm,c}_{\partial \Om}$ containing the interface or 
the boundary of the domain, respectively, the estimate follows directly from the 
estimates~\eqref{eq:reqXiInterface}--\eqref{eq:reqXiBoundary}
and the definition~\eqref{eq:ansatzXiContactPoint}. On interpolation wedges~$W^c_{\Om_v^\pm}$, 
we compute recalling~\eqref{eq:ansatzXiContactPoint}
	\begin{align*}
	\partial_t  \widehat \xi^c
	&= \lambda^\pm_c  \partial_t  \xi^c_{\mathcal{T}_i}
	+(1{-}\lambda^\pm_c) \partial_t  \xi^c_{\partial \Om}
	+ (\xi^c_{\mathcal{T}_i}{-}\xi^c_{\partial \Om})\partial_t\lambda^\pm_c
	\\
	\nabla\widehat \xi^c
	&= \lambda^\pm_c  \nabla  \xi^c_{\mathcal{T}_i}
	+(1{-}\lambda^\pm_c) \nabla   \xi^c_{\partial \Om}
	+ (\xi^c_{\mathcal{T}_i}{-}\xi^c_{\partial \Om})\otimes \nabla\lambda^\pm_c,
	\\
	\nabla^2\widehat \xi^c &= \lambda^\pm_c  \nabla^2 \xi^c_{\mathcal{T}_i}
	+ (1{-}\lambda^\pm_c) \nabla^2  \xi^c_{\partial \Om}
  + (\nabla\lambda^\pm_c\otimes\nabla^{\mathrm{sym}})(\xi^c_{\mathcal{T}_i}{-}\xi^c_{\partial \Om})
	+ (\xi^c_{\mathcal{T}_i}{-}\xi^c_{\partial \Om})\otimes\nabla^2\lambda^\pm_c.
	\end{align*}
	Then we recall the bounds~\eqref{eq:LambdaFirstDeriv} and~\eqref{eq:LambdaSecondDeriv}
	for the derivatives of the interpolation functions, the estimates~\eqref{eq:reqXiInterface} and~\eqref{eq:reqXiBoundary}
	as well as the compatibility conditions~\eqref{eq:compatibilityXiContactPoint} for the auxiliary vector 
	fields~$\xi^c_{\mathcal{T}_i}$ and~$\xi^c_{\partial \Om}$. Feeding these into
	the previous display establishes~\eqref{eq:regularityEstimateAnsatzXiContactPoint} on the interpolation wedges.
	
\textit{Step 2: Evolution equation in terms of~$\widehat\xi^c$.}
We claim that
\begin{align}
\label{eq:transportXiContactPoint}
\partial_t\widehat\xi^c + (v\cdot\nabla)\widehat\xi^c + (\nabla v)^\mathsf{T}\widehat\xi^c
&= O(\dist(\cdot,\mathcal{T}_i)) \quad\text{in } \mathcal{N}_{r_c,c}(\Omega).
\end{align}

The validity of~\eqref{eq:transportXiContactPoint} on the wedges~$W^c_{\mathcal{T}_i}$ 
and~$W^{\pm,c}_{\partial \Om}$ follows directly from the estimates~\eqref{eq:evolEquAuxXiInterface} 
resp.\ \eqref{eq:evolEquAuxXiBuondary}, the definition~\eqref{eq:ansatzXiContactPoint}
and the bound~\eqref{eq:distCondition}. Hence, we only need to prove the bound~\eqref{eq:transportXiContactPoint}
on the interpolation wedges~$W^c_{\Om_v^\pm}$.

To this end, recall first that on the interpolation wedges~$W^c_{\Om_v^\pm}$ the distance with respect to 
the contact point~$\mathcal{T}_c$ or the distance with respect to the domain boundary~$\partial\Omega$
is dominated by the distance to the connected interface $\mathcal{T}_i$, see~\eqref{eq:distCondition}.
Writing $\widehat \xi^c = \xi^c_{\mathcal{T}_i} + (1 {-} \lambda^\pm_c)
 ( \xi^c_{\partial \Om} {-} \xi^c_{\mathcal{T}_i} )$, 
and resp. $\widehat \xi^c = \xi^c_{\partial \Om} 
+ \lambda^\pm_c (\xi^c_{I_v} {-} \xi^c_{\partial \Om}  )$, 
we then immediately see that
\begin{align}
\label{eq:TensProdXiInterface}
	 \widehat \xi^c \otimes \widehat \xi^c 
	 &= \xi^c_{\mathcal{T}_i}  \otimes \xi^c_{\mathcal{T}_i} + O(\dist^2(\cdot, \mathcal{T}_i)),
\\ 
\label{eq:TensProdXiBuondary}
\widehat \xi^c \otimes \widehat \xi^c &= \xi^c_{\partial \Om } \otimes \xi^c_{\partial \Om } 
+ O(\dist^2(\cdot, \mathcal{T}_i)), 
\end{align}
due to compatibility~\eqref{eq:compatibilityXiContactPoint} up to first order at the contact point~\(\mathcal{T}_c\),
and the regularity estimates~\eqref{eq:reqXiInterface}--\eqref{eq:reqXiBoundary}.
Using the product rule and the definition~\eqref{eq:ansatzXiContactPoint} 
of~$\widehat \xi^c$ on~$W^c_{\Om_v^\pm}$, we thus obtain 
\begin{align} \label{proof:step4}
	&\pt \widehat \xi^c + (v \cdot \nabla ) \widehat \xi^c 
	+ (\Id - \widehat \xi^c \otimes \widehat \xi^c) (\nabla v)^\mathsf{T} \widehat \xi^c \nonumber  \\
	&= \lambda^\pm_c \big( \pt + (v \cdot \nabla )  + 
	(\Id -  \xi^c_{\mathcal{T}_i} \otimes \xi^c_{\mathcal{T}_i}) 
	(\nabla v)^\mathsf{T}\big)  \xi^c_{\mathcal{T}_i} \\
	&\quad + (1- \lambda^\pm_c) \big( \pt + (v \cdot \nabla )  
	+ (\Id -  \xi^c_{\partial \Om }  \otimes  \xi^c_{\partial \Om } )
	(\nabla v)^\mathsf{T}\big) \xi^c_{\partial \Om } \nonumber \\
	& \quad + (\pt \lambda^\pm_c+(v \cdot \nabla )\lambda^\pm_c ) 
	(\xi^c_{\mathcal{T}_i} - \xi^c_{\partial \Om }) \nonumber 
	+ O(\dist^2(\cdot, \mathcal{T}_i)).  
\end{align}
Hence, we obtain~\eqref{eq:transportXiContactPoint} on interpolation
wedges as a consequence of the estimates~\eqref{eq:evolEquAuxXiInterface} 
resp.\ \eqref{eq:evolEquAuxXiBuondary}, the bound~\eqref{eq:AdvectiveDerivativeLambda}
on the advective derivative of the interpolation parameter,
as well as the compatibility condition~\eqref{eq:compatibilityXiContactPoint}.

\textit{Step 3:} We next claim that
\begin{align}
\label{eq:transportXiContactLine3}
\partial_t\big|\widehat\xi^c\,\big|^2 + (v\cdot\nabla)\big|\widehat\xi^c\,\big|^2
&= O(\dist(\cdot,\mathcal{T}_i)) &&\text{in } \mathcal{N}_{r_c,c}(\Omega),
\\
\label{eq:derivLenghtHatXi}
\Big|\nabla|\widehat\xi^c\,\big|^2\Big| &= O(\dist(\cdot,\mathcal{T}_i)) 
&&\text{in } \mathcal{N}_{r_c,c}(\Omega).
\end{align}

Outside of interpolation wedges, both claims are already established
in view of the estimates~\eqref{eq:evolEquLengthAuxXiInterface}--\eqref{eq:lengthControlAuxXiInterface}
resp.\ \eqref{eq:evolEquLengthAuxXiBuondary}--\eqref{eq:lengthControlAuxXiBuondary},
the estimate~\eqref{eq:distCondition} as well as the definition~\eqref{eq:ansatzXiContactPoint}.
Using the latter, we may compute on interpolation wedges~$W^{c}_{\Om_v^\pm}$
\begin{align}
\label{eq:auxLengthAnsatzXi}
|\widehat\xi^c|^2 - 1 
&= \lambda^{\pm \, 2}_c(|\xi^c_{\mathcal{T}_i}|^2-1)
+ (1-\lambda^\pm_c)^2(|\xi^c_{\partial \Om}|^2-1) 
\\&~~~\nonumber
+ 2\lambda^\pm_c(1-\lambda^\pm_c)(\xi^c_{\mathcal{T}_i}\cdot \xi^c_{\partial \Om}-1),
\end{align}
and thus 
\begin{align}
\nonumber
\big(\partial_t{+}(v\cdot\nabla)\big)\big|\widehat\xi^c\,\big|^2
&= \big(\partial_t{+}(v\cdot\nabla)\big)\big((\lambda_c^\pm)^2|\xi^c_{\mathcal{T}_i}|^2
{+} (1{-}\lambda_c^\pm)^2|\xi^c_{\partial\Omega}|^2 + 2\lambda_c^\pm(1{-}\lambda_c^\pm)\big)
\\&~~~\label{eq:evolutionLengthContactLineAux}
+ (\xi^c_{\mathcal{T}_i}\cdot\xi^c_{\partial\Omega}{-}1)  
\big(\partial_t{+}(v\cdot\nabla)\big)\big(
2\lambda_c^\pm(1{-}\lambda_c^\pm)\big)
\\&~~~\nonumber
+ 2\lambda_c^\pm(1{-}\lambda_c^\pm)\big(\partial_t{+}(v\cdot\nabla)\big)
(\xi^c_{\mathcal{T}_i}\cdot\xi^c_{\partial\Omega}{-}1).
\end{align}

Because of~\eqref{eq:evolEquLengthAuxXiInterface}--\eqref{eq:lengthControlAuxXiInterface} and~\eqref{eq:evolEquLengthAuxXiBuondary}--\eqref{eq:lengthControlAuxXiBuondary},
the first right hand side term
of~\eqref{eq:evolutionLengthContactLineAux} is of required order. For an estimate
of the second and third right hand side term of~\eqref{eq:evolutionLengthContactLineAux}, 
observe that it suffices to prove
$\xi^c_{\mathcal{T}_i}\cdot\xi^c_{\partial\Omega}{-}1 = O(\dist^2(\cdot,\mathcal{T}_i))$ on interpolation
wedges as the advective derivative of the interpolation parameter is bounded, 
see~\eqref{eq:AdvectiveDerivativeLambda}. However, it follows immediately
from the definitions~\eqref{eq:AnsatzAuxiliaryXiInterface} and~\eqref{eq:AnsatzAuxiliaryXiBuondary},
the formulas~\eqref{eq:gradientXiInterface} and~\eqref{eq:gradientXiBoundary},
as well as the compatibility condition~\eqref{eq:compatibilityXiContactPoint}, that
at the contact point~$\mathcal{T}_c$ it holds $\xi^c_{\mathcal{T}_i}\cdot\xi^c_{\partial\Omega} = 1$,
$(\nabla\xi^c_{\mathcal{T}_i})^\mathsf{T}\xi^c_{\partial\Omega}= 0$ and
$(\nabla\xi^c_{\partial\Omega})^\mathsf{T}\xi^c_{\mathcal{T}_i}= 0$. Hence,
$\xi^c_{\mathcal{T}_i}\cdot\xi^c_{\partial\Omega}{-}1 = O(\dist^2(\cdot,\mathcal{T}_i))$ is
a consequence of a Lipschitz estimate making use of the estimates~\eqref{eq:reqXiInterface}--\eqref{eq:reqXiBoundary}
and the bound~\eqref{eq:distCondition}.

In summary, the above arguments upgrade~\eqref{eq:evolutionLengthContactLineAux} 
to~\eqref{eq:transportXiContactLine3}, and analogous considerations
based on~\eqref{eq:auxLengthAnsatzXi} also entail~\eqref{eq:derivLenghtHatXi}
on interpolation wedges.

\textit{Step 4: Choice of $\widehat r_c$ and definition of the normalized vector field $\xi^c$.}
By the definition~\eqref{eq:ansatzXiContactPoint} of the vector field~$\widehat \xi^c$ we have 
$|\widehat \xi^c(\cdot,t)|=1$ on $ B_{r_c}(\mathcal{T}_c(t))\cap (\partial \Om \cup \mathcal{T}_i(t))$ 
for all $t\in [0,T]$. Due to its Lipschitz continuity, see \textit{Step~1} of the proof,
we may choose a radius~$\widehat r_c \leq r_c$ such that~$|\widehat\xi^c|\geq\frac{1}{2}$
holds true in the space-time domain~$\mathcal{N}_{\widehat r_c,c}(\Omega)$.
We then define~$\xi^c:=\big|\widehat\xi^c\,\big|^{-1}\widehat\xi^c\in\mathbb{S}^1$
throughout~$\mathcal{N}_{\widehat r_c,c}(\Omega)$, so that it
remains to argue that the properties of~$\widehat\xi^c$ are inherited by~$\xi^c$. 

Since $\xi^c(\cdot,t)=\widehat\xi^c(\cdot,t)$ on 
$B_{r_c}(\mathcal{T}_c(t))\cap (\partial \Om \cup \mathcal{T}_i(t))$ 
for all $t\in [0,T]$, it immediately follows that
$\xi^c(\cdot,t)=n_{I_v}(\cdot,t)$ along~$\mathcal{T}_i(t)\cap B_{\widehat r_c}(\mathcal{T}_c(t))$
as well as $\xi^c(\cdot,t)\cdot n_{\partial\Omega}(\cdot)=0$
along~$\partial\Omega\cap B_{\widehat r_c}(\mathcal{T}_c(t))$
for all~$t\in [0,T]$. Moreover, $\nabla\cdot\xi^c
= |\widehat\xi^c|^{-1}\nabla\cdot\widehat\xi^c
- \frac{(\widehat\xi^c\cdot\nabla)|\widehat\xi^c|^2}{2|\widehat\xi^c|^3}$
so that~$\nabla\cdot\xi^c = -H_{I_v}(\cdot,t)$ holds true
on~$\mathcal{T}_i(t)\cap B_{\widehat r_c}(\mathcal{T}_c(t))$
for all~$t\in[0,T]$ because of~\eqref{eq:derivLenghtHatXi},
the validity of this equation in terms of~$\widehat\xi^c$,
and the fact that~$|\widehat\xi^c(\cdot,t)|=1$ on
$\mathcal{T}_i(t)\cap B_{\widehat r_c}(\mathcal{T}_c(t))$ for all~$t\in [0,T]$.
In summary, properties~\textit{ii)--iii)} are satisfied.

The required regularity is obtained by the choice of the radius~$\widehat r_c$, the 
definition~$\xi^c:=\big|\widehat\xi^c\,\big|^{-1}\widehat\xi^c$,
and the fact that the vector field~$\widehat\xi^c$ already satisfies it
as argued in~\textit{Step~1} of this proof. Since~$\xi^c\in\mathbb{S}^1$
throughout~$\mathcal{N}_{\widehat r_c,c}(\Omega)$,
\eqref{eq:timeEvolutionLengthXiContactPoint} holds true for trivial reasons.
For a proof of~\eqref{eq:timeEvolutionXiContactPoint}, one may argue as follows.
Recalling that~$|\widehat\xi^c|\geq\frac{1}{2}$
holds true in~$\mathcal{N}_{\widehat r_c,c}(\Omega)$,
adding zero and using the product rule yields
\begin{align*}
& \partial_t \xi^c + (v\cdot \nabla) \xi^c +(\Id - \xi^c \otimes \xi^c)(\nabla v)^\mathsf{T} \xi^c 
\\
&=  \partial_t \xi^c + (v\cdot \nabla) \xi^c +
(\Id - \widehat \xi^c \otimes  \widehat \xi^c)(\nabla v)^\mathsf{T} \xi^c 
-(1-|\widehat \xi^c|^2)( \xi^c \otimes  \xi^c ) (\nabla v)^\mathsf{T} \xi^c 
\\
&=\frac1{|\widehat \xi^c|} \big( \partial_t \widehat \xi^c + (v\cdot \nabla) \widehat \xi^c 
+(\Id - \widehat \xi^c \otimes  \widehat \xi^c)(\nabla v)^\mathsf{T} \widehat \xi^c\big)
-\frac{\widehat\xi^c}{2|\widehat \xi^c|^3}  ( \partial_t |\widehat \xi^c|^2+ (v\cdot \nabla) |\widehat \xi^c|^2 ) 
\\
&\quad -(1-|\widehat \xi^c|^2)( \xi^c \otimes  \xi^c ) (\nabla v)^\mathsf{T} \xi^c 
\end{align*}
throughout~$\mathcal{N}_{\widehat r_c,c}(\Omega)$. Observe that the first right hand side term
is estimated by~\eqref{eq:transportXiContactPoint}, the second by~\eqref{eq:transportXiContactLine3},
and the third by a Lipschitz estimate based on the fact~$|\widehat\xi^c(\cdot,t)|=1$ along
$\mathcal{T}_i(t)\cap B_{\widehat r_c}(\mathcal{T}_c(t))$ for all~$t\in [0,T]$. 
Hence, \eqref{eq:timeEvolutionXiContactPoint} holds true.

\textit{Step 5: Contact point extensions as perturbations of bulk extensions.}
As a preparation for the proof of the compatability estimates, we claim that
\begin{align}
\label{eq:higherOrderControlLength}
|\xi^{c}{-}\widehat\xi^{\, c}| \leq C\dist^2(\cdot,\mathcal{T}_i).
\end{align}
Note that because of the definition~\eqref{eq:ansatzXiContactPoint},
the compatibility conditions~\eqref{eq:compatibilityXiContactPoint} at the contact point,
the regularity estimates~\eqref{eq:reqXiInterface}--\eqref{eq:reqXiBoundary}
for the local building blocks, the controlled blow-up~\eqref{eq:LambdaFirstDeriv}
, the coercivity estimate~\eqref{eq:lengthControlAuxXiInterface}
,
and the estimate~\eqref{eq:distCondition}, it holds
\begin{align*}
\nabla \frac{1}{|\widehat\xi^{\,c}|} 
= - \frac{(\widehat\xi^{\,c}\cdot\nabla)\widehat\xi^{\,c}}{|\widehat\xi^{\,c}|^3}
&= - \frac{(\xi^{c}_{\mathcal{T}_i}\cdot\nabla)\widehat\xi^{\,c}}{|\widehat\xi^{\,c}|^3}
+ O(\dist(\cdot,\mathcal{T}_i))
\\
&= - \frac{(\xi^{c}_{\mathcal{T}_i}\cdot\nabla)\xi^{\,c}_{\mathcal{T}_i}}{|\widehat\xi^{\,c}|^3}
+ O(\dist(\cdot,\mathcal{T}_i))
= O(\dist(\cdot,\mathcal{T}_i)).
\end{align*}
Hence, the asserted estimate~\eqref{eq:higherOrderControlLength}
follows from $\xi^{c}{-}\widehat\xi^{\, c}=(|\widehat\xi^{\,c}|^{-1}{-}1)\widehat\xi^{\,c}$,
the fact that $\xi^{c}(\cdot,t)=\widehat\xi^{\,c}(\cdot,t)\equiv n_{I_v}(\cdot,t)$
along the local interface patch~$\mathcal{T}_i(t)\cap B_{\widehat r^{\,c}}(\mathcal{T}_c(t))$
for all $t\in [0,T]$, and the previous display.

We exploit~\eqref{eq:higherOrderControlLength} as follows.
Within the interface wedge~$W^c_{\mathcal{T}_i}$, it now follows from the
definitions~\eqref{eq:DefXiInterface}, \eqref{eq:AnsatzAuxiliaryXiInterface} 
and~\eqref{eq:ansatzXiContactPoint} that
\begin{align*}
\xi^c - \xi^i =\xi^c_{\mathcal{T}_i} - \xi^i + O(\dist^2(\cdot,\mathcal{T}_i))
=  \alpha_{\mathcal{T}_i}s_{\mathcal{T}_i}\tau_{I_v}
- \frac{1}{2}\alpha_{\mathcal{T}_i}^2s_{\mathcal{T}_i}^2n_{I_v}
+ O(\dist^2(\cdot,\mathcal{T}_i)).
\end{align*}
Within interpolation wedges, we have the same representation
thanks to the first-order compatibility~\eqref{eq:compatibilityXiContactPoint}
in form of
\begin{align*}
\xi^c - \xi^i &=\widehat \xi^c - \xi^i + O(\dist^2(\cdot,\mathcal{T}_i))
\\&
= (\xi^c_{\mathcal{T}_i} - \xi^i)
+ (1{-}\lambda^\pm_c)(\xi^c_{\partial\Omega} - \xi^c_{\mathcal{T}_i})
+ O(\dist^2(\cdot,\mathcal{T}_i))
\\&
= \alpha_{\mathcal{T}_i}s_{\mathcal{T}_i}\tau_{I_v}
- \frac{1}{2}\alpha_{\mathcal{T}_i}^2s_{\mathcal{T}_i}^2n_{I_v}
+ O(\dist^2(\cdot,\mathcal{T}_i)).
\end{align*}
In particular, the compatibility bounds~\eqref{eq:compBoundLocalCalibrations1} 
and~\eqref{eq:compBoundLocalCalibrations2} are satisfied within
interface and interpolation wedges, respectively. 
\qed 

\section{Existence of boundary adapted extensions of the unit normal}
\label{sec:globalConstruction}

\subsection{From local to global extensions}
\label{subsec:partitionOfUnity}
The idea for proving Proposition~\ref{prop:existenceCalibration}
consists of stitching together the local extensions from
the previous two sections by means of a suitable partition of unity
on the interface $I_v$. For a construction of the latter, recall first the decomposition of
the interface $I_v$ into its topological features, namely,
the connected components of $I_v\cap\Omega$ and the connected components
of $I_v\cap \partial\Omega$. Denoting by $N\in\mathbb{N}$ the total
number of such topological features present in the interface $I_v$ we split
$\{1,\ldots,N\}=:\mathcal{I}\cupdot\mathcal{C}$ by means of two disjoint subsets.
Here, the subset $\mathcal{I}$ enumerates the space-time connected components
of $I_v\cap\Omega$ (being time-evolving connected \textit{interfaces}), whereas
the subset $\mathcal{C}$ enumerates the space-time connected components of $I_v\cap\partial\Omega$
(being time-evolving \textit{contact points}).
If $i\in\mathcal{I}$, we let $\mathcal{T}_i\subset I_v$ denote the space-time trajectory in $\Omega$ of the 
corresponding connected interface. Furthermore, for every $c\in\mathcal{C}$ we write $\mathcal{T}_c$
representing the space-time trajectory in $\partial\Omega$ of the corresponding contact point. 
Finally, let us write $i\sim c$ for $i\in\mathcal{I}$
and $c\in\mathcal{C}$ if and only if $\mathcal{T}_i$ ends at $\mathcal{T}_c$; otherwise $i\not\sim c$.

\begin{lemma}[Construction of a partition of unity]
\label{lem:partitionOfUnity}
Let $d=2$, and let $\Omega\subset\R^2$ be a bounded domain with orientable and smooth boundary.
Let $(\chi_v,v)$ be a strong solution to the incompressible
Navier--Stokes equation for two fluids in the sense of Definition~\ref{Def_strongsol}
on a time interval $[0,T]$. For each $i\in\mathcal{I}$ let $r_i$
be the localization radius of Definition~\ref{def:locRadiusInterface}, and for each $c\in\mathcal{C}$
denote by $\widehat r_c$ the localization radius of Proposition~\ref{prop:CalibrationContactPoint}.
There then exists a family $(\eta_1,\ldots,\eta_N)$ of cutoff functions
\begin{align}
\nonumber
&\eta_n\colon \mathbb{R}^2\times [0,T] \to [0,1],
\quad n\in\{1,\ldots,N\},
\\\label{eq:regularityCutoff}
&\text{with the regularity }\eta_n\in (C^0_tC^2_x\cap C^1_tC^0_x)
\Big(\mathbb{R}^2{\times}[0,T]\setminus\bigcup_{c\in\mathcal{C}}\mathcal{T}_c\Big),
\end{align}
and a localization radius $\widehat r \in 
(0,\min_{i\in\mathcal{I}} r_i \wedge \min_{c\in\mathcal{C}} \widehat r_c)$, which 
together are subject to the following
list of conditions:
\begin{itemize}[leftmargin=0.7cm]
\item The family $(\eta_1,\ldots,\eta_N)$ is a partition of unity
			along the interface $I_v$. Defining a bulk cutoff
			by means of $\eta_{\mathrm{bulk}}:= 1 - \sum_{n=1}^N \eta_n$,
			it holds $\eta_{\mathrm{bulk}}\in [0,1]$. On top we have coercivity estimates in form of
			\begin{align}
			\label{eq:coercivityBulkCutoff}
			\frac{1}{C}(\dist^2(\cdot,I_v) \wedge 1) &\leq
			\eta_{\mathrm{bulk}} \leq C(\dist^2(\cdot,I_v) \wedge 1) 
			&&\text{in } \R^2\times [0,T],
			\\ \label{eq:coercivityGradientBulkCutoff}
			|\nabla\eta_{\mathrm{bulk}}| &\leq C(\dist(\cdot,I_v) \wedge 1) 
			&&\text{in } \R^2\times [0,T],
			\end{align}
\item For all two-phase interfaces $i\in\mathcal{I}$ it holds
			\begin{align}
			\label{eq:supportCutoffTwoPhase}
			\supp \eta_i(\cdot,t) \subset \Psi_{\mathcal{T}_i}
			(\mathcal{T}_i(t){\times}\{t\}{\times}[-\widehat r,\widehat r])
			\quad\text{for all } t\in [0,T],
			\end{align}
			with $\Psi_{\mathcal{T}_i}$ denoting the change of variables 
			from Definition~\ref{def:locRadiusInterface}.	For contact points $c\in\mathcal{C}$, it is required that
			\begin{align}
			\label{eq:supportCutoffContact}
			\supp \eta_c(\cdot,t) \subset B_{\widehat r}\big(\mathcal{T}_c(t)\big)
			\quad\text{for all } t\in [0,T].
			\end{align}
\item For all distinct two-phase interfaces $i,i'\in\mathcal{I}$ it holds
			\begin{align}
			\label{eq:supportCutoffDistinctTwoPhase}
			\supp \eta_i(\cdot,t) \cap \supp \eta_{i'}(\cdot,t) = \emptyset
			\quad\text{for all } t\in [0,T].
			\end{align}
			The same is required for all distinct contact points $c,c'\in\mathcal{I}$
			\begin{align}
			\label{eq:supportCutoffDistinctContact}
			\supp \eta_c(\cdot,t) \cap \supp \eta_{c'}(\cdot,t) = \emptyset
			\quad\text{for all } t\in [0,T].
			\end{align}
\item Let a two-phase interface $i\in\mathcal{I}$ and a contact point $c\in\mathcal{C}$
			be fixed. Then $\supp\eta_i \cap \supp\eta_c \neq \emptyset$ if and only if $i\sim c$,
			and in that case it holds
			\begin{align}
			\label{eq:LocalizationSupportTwoPhaseContact}
			\hspace*{4ex}\supp \eta_i(\cdot,t) \cap  \supp \eta_c(\cdot,t) 
			\subset B_{\widehat r}(\mathcal{T}_c(t)) \cap 
			\big(W^c_{\mathcal{T}_i}(t) \cup W^c_{\Omega^\pm_v}(t)\big)
			\end{align}
			for all~$t\in [0,T]$,
			with the wedges $W^c_{\mathcal{T}_i}$ and $W^c_{\Omega^\pm_v}$ 
			introduced in Definition~\ref{def:locRadius}.
\end{itemize}
\end{lemma}

\begin{proof}
The proof proceeds in several steps.

\textit{Step 1: (Definition of auxiliary cutoff functions)} 
Fix a smooth cutoff function $\theta\colon\R\to[0,1]$
with the properties that $\theta(r)=1$ for $|r|\leq \frac{1}{2}$
and $\theta(r)=0$ for $|r|\geq 1$. Define
\begin{align}
\label{eq:quadraticProfileBuildingBlock}
\zeta(r) := (1 - r^2)\theta( r^2  ),\quad r\in\R.
\end{align}
Based on this quadratic profile, we may introduce two classes of
cutoff functions associated to the two different natures of
topological features present in the interface~$I_v$. 
To this end, let $\widehat r \in (0,\min_{i\in\mathcal{I}} r_i 
\wedge \min_{c\in\mathcal{C}} \widehat r_c)$. Moreover, let $\delta\in (0,1]$ be a constant.
Both constants $\widehat r$ and $\delta$ will be determined in
the course of the proof.

For two-phase interfaces $\mathcal{T}_i\subset I_v$, $i\in\mathcal{I}$, we 
may then define
\begin{align}
\label{eq:cutOffProfileTwoPhase}
\zeta_i(x,t) := \zeta\Big(\frac{\sdist(x,\mathcal{T}_i(t))}{\delta \widehat r}\Big),
\, (x,t) \in \mathrm{im}(\Psi_{\mathcal{T}_i})
:=\Psi_{\mathcal{T}_i}\big(\mathcal{T}_{i} {\times} (-2r_i,2r_i)\big) 
\end{align}
where the change of variables $\Psi_{\mathcal{T}_i}$ 
and the associated signed distance $\sdist(\cdot,\mathcal{T}_i)$
are from Definition~\ref{def:locRadiusInterface} of the admissible localization radius~$r_i$.
Furthermore, for contact points $\mathcal{T}_c$, $c\in\mathcal{C}$, we define
\begin{align}
\label{eq:cutOffProfileContact}
\zeta_c(x,t) := \zeta\Big(\frac{\dist(x,\mathcal{T}_c(t))}{\delta \widehat r}\Big),
\quad (x,t) \in \R^2 \times [0,T].
\end{align}

\textit{Step 2: (Choice of the constant $\widehat r \in (0,\min_{i\in\mathcal{I}} r_i 
\wedge \min_{c\in\mathcal{C}} \widehat r_c)$)}
It is a consequence of the uniform regularity of the interface $I_v$ in space-time
that one may choose $\widehat r \in (0,\min_{i\in\mathcal{I}} r_i 
\wedge \min_{c\in\mathcal{C}} \widehat r_c)$ small
enough such that the following localization properties hold true
\begin{align}
\label{eq:locPropertyCutoff1}
\Psi_{\mathcal{T}_i}(\mathcal{T}_i(t){\times}\{t\}{\times}[-\widehat r,\widehat r]) \cap 
\Psi_{\mathcal{T}_{i'}}(\mathcal{T}_{i'}(t){\times}\{t\}{\times}[-\widehat r,\widehat r]) &= \emptyset
\quad\forall i'\in\mathcal{I},\,i'\neq i,
\\ \label{eq:locPropertyCutoff2}
\Psi_{\mathcal{T}_i}(\mathcal{T}_i(t){\times}\{t\}{\times}[-\widehat r,\widehat r]) \cap 
B_{\widehat r}(\mathcal{T}_c(t)) &\neq \emptyset
\,\,\Leftrightarrow\,\, \exists c \in \mathcal{C} \colon i \sim c,
\\ \label{eq:locPropertyCutoff4}
B_{\widehat r}(\mathcal{T}_c(t)) \cap B_{\widehat r}(\mathcal{T}_{c'}(t)) &= \emptyset
\quad\forall c,c'\in\mathcal{C},\,c'\neq c.
\end{align}
for all $t\in [0,T]$ and all $i\in\mathcal{I}$.

\textit{Step 3: (Construction of the partition of unity, part I)}
We start with the construction of the cutoffs $\eta_i$ for two-phase
interfaces $i\in\mathcal{I}$. Away from contact points, we set
\begin{align}
\label{eq:defCutoffTwoPhaseAwayContactPoints}
&\eta_i(x,t) := \zeta_i(x,t),
\quad
(x,t) \in \mathrm{im}(\Psi_{\mathcal{T}_i})\setminus
\bigcup_{c\in\mathcal{C}}\bigcup_{t'\in [0,T]} 
B_{\widehat r}\big(\mathcal{T}_c(t')\big) {\times} \{t'\},
\end{align}
which is well-defined due to the choice of $\widehat r$.

Assume now there exists $c\in\mathcal{C}$ such that $i\sim c$.
Recall from Definition~\ref{def:locRadius} of the admissible localization radius $r_c$
that for all $t\in[0,T]$ we decomposed $\Omega \cap B_{r_c}(\mathcal{T}_c(t))$
by means of five pairwise disjoint open wedges $W^{\pm,c}_{\partial\Omega}(t), W^c_{\mathcal{T}_i}(t),
W^c_{\Omega^\pm_v}(t)\subset\R^2$. In the wedge $W^c_{\mathcal{T}_i}$
containing the two-phase interface $\mathcal{T}_i\subset I_v$, we define
\begin{align}
\label{eq:defCutoffTwoPhaseInterfaceWedge}
&\eta_i(x,t) := (1-\zeta_c(x,t))\zeta_i(x,t),
\,
(x,t) \in \bigcup_{t'\in [0,T]} \big(B_{\widehat r}\big(\mathcal{T}_c(t')\big) 
\cap  W^c_{\mathcal{T}_i}(t')\big) {\times} \{t'\}. 
\end{align}
This is indeed well-defined by the choice of $\widehat r$ and having
\begin{align*}
B_{r_c}(\mathcal{T}_c(t)) \cap W^c_{\mathcal{T}_i}(t) \subset
\Psi_{\mathcal{T}_{i}}(\mathcal{T}_{i}(t){\times}\{t\}{\times}(-2 r_c,2 r_c))
\end{align*}
for all $t\in [0,T]$; the latter in turn being a consequence of Definition~\ref{def:locRadius} of the 
admissible localization radius $r_c$.

Within the ball $B_{\widehat r}(\mathcal{T}_c(t))$, we aim to restrict the support 
of $\eta_i(\cdot,t)$ to the region $B_{\widehat r}(\mathcal{T}_c(t)) \cap 
\big(W^c_{\mathcal{T}_i}(t) \cup W^c_{\Omega^\pm_v}(t)\big)$ for all $t\in [0,T]$. 
This will be done by means of the interpolation functions $\lambda^{\pm}_{c}$ 
of Lemma~\ref{lemma:existenceInterpolationFunction}. Recall in this context the convention that $\lambda^\pm_c(\cdot,t)$
was set equal to one on $\big(\partial W^c_{\Omega^\pm_v}(t) \cap \partial W^c_{\mathcal{T}_i}(t)\big)
\setminus \mathcal{T}_c(t)$ and set equal to zero on
 $\big(\partial W^c_{\Omega^\pm_v}(t) 
\cap \partial W^{\pm,c}_{\partial\Omega}(t)\big) \setminus \mathcal{T}_c(t)$ for all $t\in [0,T]$.
In particular, we may define in the interpolation wedges $W^c_{\Omega^\pm_v}$ 
\begin{align}
\label{eq:defCutoffTwoPhaseInterpolationWedge}
&\eta_i(x,t) := \lambda^\pm_c(x,t)(1-\zeta_c(x,t))\zeta_i(x,t),
\\& \nonumber
(x,t) \in \bigcup_{t'\in [0,T]} \big(B_{\widehat r}\big(\mathcal{T}_c(t')\big) 
\cap  W^c_{\Omega^\pm_v}(t)\big) {\times} \{t'\}. 
\end{align}
Again, this is well-defined because of the choice of $\widehat r$ and the fact that
\begin{align*}
B_{r_c}(\mathcal{T}_c(t)) \cap W^c_{\Omega^\pm_v}(t) \subset
\Psi_{\mathcal{T}_{i}(t)}(\mathcal{T}_{i}(t){\times}\{t\}{\times}(-2 r_c,2 r_c))
\end{align*}
for all $t\in [0,T]$ due to Definition~\ref{def:locRadius} of the 
admissible localization radius $r_c$.

Outside of the space-time domains appearing in the
definitions~\eqref{eq:defCutoffTwoPhaseAwayContactPoints}--\eqref{eq:defCutoffTwoPhaseInterpolationWedge},
we simply set $\eta_i$ equal to zero. 

In view of the definitions~\eqref{eq:quadraticProfileBuildingBlock}--\eqref{eq:cutOffProfileContact}
and the definitions~\eqref{eq:defCutoffTwoPhaseAwayContactPoints}--\eqref{eq:defCutoffTwoPhaseInterpolationWedge},
it now suffices to choose $\delta\in (0,1]$ sufficiently small such that~\eqref{eq:supportCutoffTwoPhase} holds true,
and in case there exists $c\in\mathcal{C}$ such that $i\sim c$ one may on top achieve
\begin{align}
\label{eq:auxLocalizationSupportTwoPhase}
\supp \eta_i(\cdot,t) \cap  B_{\widehat r}(\mathcal{T}_c(t))
\subset B_{\widehat r}(\mathcal{T}_c(t)) \cap 
\big(W^c_{\mathcal{T}_i}(t) \cup W^c_{\Omega^\pm_v}(t)\big)
\end{align}
for all $t\in [0,T]$. Moreover, in light of~\eqref{eq:supportCutoffTwoPhase}
and~\eqref{eq:locPropertyCutoff1} we also obtain~\eqref{eq:supportCutoffDistinctTwoPhase}.

\textit{Step 4: (Construction of the partition of unity, part II)}
We proceed with the construction of the cutoffs $\eta_c$
for contact points $c\in\mathcal{C}$. To this end,
let $i\in\mathcal{I}$ be the unique two-phase interface such that $i\sim c$.
In the wedge $W^c_{\mathcal{T}_i}$
containing the two-phase interface $\mathcal{T}_i\subset I_v$ we set
\begin{align}
\label{eq:defCutoffContactInterfaceWedge}
&\eta_c(x,t) := \zeta_c(x,t)\zeta_i(x,t),
\,
(x,t) \in \bigcup_{t'\in [0,T]} \big(B_{\widehat r}\big(\mathcal{T}_c(t')\big) 
\cap  W^c_{\mathcal{T}_i}(t')\big) {\times} \{t'\},
\end{align}
which is well-defined based on the same reason as for~\eqref{eq:defCutoffTwoPhaseInterfaceWedge}.

Moreover, in the interpolation wedges $W^c_{\Omega^\pm_v}$ we define
\begin{align}
\label{eq:defCutoffContactInterpolationWedge}
&\eta_c(x,t) := \lambda^\pm_c(x,t)\zeta_c(x,t)\zeta_i(x,t)
								+ (1-\lambda^\pm_c(x,t))\zeta_c(x,t),
\\& \nonumber
(x,t) \in \bigcup_{t'\in [0,T]} \big(B_{\widehat r}\big(\mathcal{T}_c(t')\big) 
\cap  W^c_{\Omega^\pm_v}(t)\big) {\times} \{t'\}. 
\end{align}
By the same argument as for~\eqref{eq:defCutoffTwoPhaseInterpolationWedge},
this is again well-defined.

Outside of the space-time domains appearing in the previous two definitions
we simply set $\eta_c := \zeta_c$. In particular, we register for reference purposes that
\begin{align}
\label{eq:defCutoffContactBoundaryWedge}
&\eta_c(x,t) := \zeta_c(x,t),
\, 
(x,t) \in \bigcup_{t'\in [0,T]} \big(B_{\hat r}\big(\mathcal{T}_c(t')\big) 
\setminus \big(W^c_{\mathcal{T}_i}(t') \cup W^c_{\Omega^\pm_v}(t)\big)\big) {\times} \{t'\}.
\end{align}

It now immediately follows from the definition~\eqref{eq:cutOffProfileContact} 
that~\eqref{eq:supportCutoffContact} is satisfied.
In particular, for pairs $i\in\mathcal{I}$ and $c\in\mathcal{C}$
such that $i\sim c$, $\supp\eta_i \cap \supp\eta_c \neq \emptyset$ and we obtain~\eqref{eq:LocalizationSupportTwoPhaseContact}
as an update of~\eqref{eq:auxLocalizationSupportTwoPhase}. 
Moreover, by~\eqref{eq:supportCutoffContact} and~\eqref{eq:locPropertyCutoff4}
we deduce the validity of~\eqref{eq:supportCutoffDistinctContact}.
In the case of pairs $i\in\mathcal{I}$ and $c\in\mathcal{C}$ with $i\not\sim c$, 
due to~\eqref{eq:locPropertyCutoff2}, \eqref{eq:supportCutoffTwoPhase} and~\eqref{eq:supportCutoffContact}, 
we can conclude that $\supp\eta_i \cap \supp\eta_c = \emptyset$.

\textit{Step 5: (Partition of unity property along the interface)} Fix $t\in [0,T]$,
and consider first the case of $x\in I_v(t)\setminus \bigcup_{c\in\mathcal{C}} B_{\widehat r}(\mathcal{T}_c(t))$.
The combination of the support properties~\eqref{eq:supportCutoffTwoPhase} and
\eqref{eq:supportCutoffContact} with the localization property~\eqref{eq:locPropertyCutoff1}
implies there exists a unique two-phase interface $i_*=i_*(x)\in\mathcal{I}$ such that
$\sum_{n=1}^N \eta_n(x,t) = \eta_{i_*}(x,t)$. Hence, we may deduce 
from~\eqref{eq:defCutoffTwoPhaseAwayContactPoints} that $\sum_{n=1}^N \eta_n(x,t)=1$
for all $t\in [0,T]$ and all $x\in I_v(t)\setminus \bigcup_{c\in\mathcal{C}} B_{\widehat r}(\mathcal{T}_c(t))$.

Fix a contact point $c\in\mathcal{C}$ and a point $x\in I_v(t) \cap B_{\widehat r}(\mathcal{T}_c(t))$.
Let $i\in\mathcal{I}$ be the unique two-phase interface such that $i\sim c$.
By the support properties~\eqref{eq:supportCutoffTwoPhase} and~\eqref{eq:supportCutoffContact}
in combination with the localization properties~\eqref{eq:locPropertyCutoff1}--\eqref{eq:locPropertyCutoff4}
it follows that $\sum_{n=1}^N \eta_n(x,t) = \eta_c(x,t) + \eta_i(x,t)$.
In particular $\sum_{n=1}^N \eta_n(x,t) = 1$ due to the definitions~\eqref{eq:defCutoffTwoPhaseInterfaceWedge}
and~\eqref{eq:defCutoffContactInterfaceWedge}.
The two discussed cases thus imply that
\begin{align}
\label{eq:partitionOfUnityInterface}
\sum_{n=1}^N \eta_n(x,t) = 1,
\quad (x,t) \in \bigcup_{t'\in [0,T]} I_v(t') \times \{t'\}.
\end{align}

\textit{Step 6: (Regularity)} Outside of interpolation wedges,
the required regularity
is an immediate consequence of the uniform regularity of the interface $I_v$
and the definitions~\eqref{eq:defCutoffTwoPhaseAwayContactPoints}, 
\eqref{eq:defCutoffTwoPhaseInterfaceWedge}, \eqref{eq:defCutoffContactInterfaceWedge}
and~\eqref{eq:defCutoffContactInterpolationWedge}. 

In interpolation wedges, one has to argue based on the definitions~\eqref{eq:defCutoffTwoPhaseInterpolationWedge}
and~\eqref{eq:defCutoffContactInterpolationWedge}. In terms of regularity, the critical
cases originating from an application of the product rule consist of those when derivatives 
hit the interpolation parameter. However, the by~\eqref{eq:LambdaFirstDeriv}--\eqref{eq:LambdaSecondDeriv} 
controlled blow-up of the derivatives of the
interpolation parameter is always counteracted by the presence of the term $1-\zeta_c$
(cf.\ \eqref{eq:defCutoffTwoPhaseInterpolationWedge}
and~\eqref{eq:defCutoffContactInterpolationWedge}) which is of second
order in the distance to the contact point due to~\eqref{eq:quadraticProfileBuildingBlock} 
and~\eqref{eq:cutOffProfileContact}. In other words, the required regularity
also holds true within interpolation wedges.

The two considered cases taken together entail the asserted regularity.

\textit{Step 7: (Estimate for the bulk cutoff)} In the course of 
establishing the desired coercivity estimates~\eqref{eq:coercivityBulkCutoff}
and~\eqref{eq:coercivityGradientBulkCutoff},
we also convince ourselves of the fact that 
\begin{align}
\label{eq:partitionOfUnityBulk}
\eta_{\mathrm{bulk}} = 1 - \sum_{n=1}^N \eta_n \in [0,1]
\end{align}
throughout $\R^2\times [0,T]$. By the support 
properties~\eqref{eq:supportCutoffTwoPhase} and~\eqref{eq:supportCutoffContact},
in both cases it suffices to argue for points contained in 
$\Psi_{\mathcal{T}_i}\big(\mathcal{T}_i(t){\times}\{t\}{\times}[-\widehat r,\widehat r]\big)
\setminus \bigcup_{c\in\mathcal{C}} B_{\widehat r}\big(\mathcal{T}_c(t)\big)$
or $B_{\widehat r}(\mathcal{T}_c(t))$ for all $i\in\mathcal{I}$, all $c\in\mathcal{C}$ and all $t\in [0,T]$.

We start with the latter and fix $i\in\mathcal{I}$ as well as $t\in [0,T]$. Due to the localization 
property~\eqref{eq:locPropertyCutoff1} and subsequently plugging in~\eqref{eq:defCutoffTwoPhaseAwayContactPoints},
we get
\begin{align}
\label{eq:representationBulkCutoff1}
\hspace*{-1.1ex}
\eta_{\mathrm{bulk}}(\cdot,t) &= 1 {-} \eta_i(\cdot,t)
= 1 {-} \zeta_i(\cdot,t) 
\text{ in } \Psi_{\mathcal{T}_i}\big(\mathcal{T}_i(t){\times}\{t\}{\times}[-\widehat r,\widehat r]\big)
\setminus \bigcup_{c\in\mathcal{C}} B_{\widehat r}\big(\mathcal{T}_c(t)\big).
\end{align}
The validity of~\eqref{eq:coercivityBulkCutoff}, \eqref{eq:coercivityGradientBulkCutoff} 
and~\eqref{eq:partitionOfUnityBulk}
in $\Psi_{\mathcal{T}_i}\big(\mathcal{T}_i(t){\times}\{t\}{\times}[-\widehat r,\widehat r]\big)
\setminus \bigcup_{c\in\mathcal{C}} B_{\widehat r}\big(\mathcal{T}_c(t)\big)$
thus follows immediately from definition~\eqref{eq:cutOffProfileTwoPhase}. 

Fix $c\in\mathcal{C}$, and let $i\in\mathcal{I}$ be the unique two-phase interface
with $i\sim c$. Due to~\eqref{eq:supportCutoffTwoPhase}, \eqref{eq:supportCutoffContact}
as well as~\eqref{eq:locPropertyCutoff1}--\eqref{eq:locPropertyCutoff4} we have
\begin{align}
\label{eq:representationBulkCutoff5}
\eta_{\mathrm{bulk}}(\cdot,t) &= 1 - \eta_c(\cdot,t) - \eta_i(\cdot,t)
\quad\text{in } B_{\widehat r}\big(\mathcal{T}_c(t)\big) 
\cap  \big(W^c_{\mathcal{T}_i}(t) \cup W^c_{\Omega^\pm_v}(t)\big). 
\end{align}
Plugging in~\eqref{eq:defCutoffTwoPhaseInterfaceWedge} and~\eqref{eq:defCutoffContactInterfaceWedge}
or~\eqref{eq:defCutoffTwoPhaseInterpolationWedge} and~\eqref{eq:defCutoffContactInterpolationWedge}, respectively,
yields
\begin{align}
\label{eq:representationBulkCutoff2}
\eta_{\mathrm{bulk}}(\cdot,t) &= 1 - \zeta_i(\cdot,t) 
\quad\text{in } B_{\widehat r}\big(\mathcal{T}_c(t)\big) 
\cap W^c_{\mathcal{T}_i}(t), 
\end{align}
as well as
\begin{align}
\label{eq:representationBulkCutoff3}
\eta_{\mathrm{bulk}}(\cdot,t) &= \lambda^\pm_c(\cdot,t)(1{-}\zeta_i(\cdot,t)) 
																	+ (1{-}\lambda^\pm_c(\cdot,t))(1{-}\zeta_c(\cdot,t))
&&\text{in } B_{\widehat r}\big(\mathcal{T}_c(t)\big)
\cap W^c_{\Omega^\pm_v}(t).
\end{align}
Hence, we can infer by means of~\eqref{eq:cutOffProfileTwoPhase} and~\eqref{eq:cutOffProfileContact}
that~\eqref{eq:coercivityBulkCutoff}, \eqref{eq:coercivityGradientBulkCutoff} and~\eqref{eq:partitionOfUnityBulk} 
hold true in the domain~$B_{\widehat r}\big(\mathcal{T}_c(t)\big) 
\cap \big(W^c_{\mathcal{T}_i}(t) \cup W^c_{\Omega^\pm_v}(t)\big)$.
Finally, we have
\begin{align}
\label{eq:representationBulkCutoff4}
\eta_{\mathrm{bulk}}(\cdot,t) &= 1 - \eta_c(\cdot,t)
= 1 - \zeta_c(\cdot,t) 
\quad\text{in } B_{\widehat r}\big(\mathcal{T}_c(t)\big) 
\setminus \big(W^c_{\mathcal{T}_i}(t) \cup W^c_{\Omega^\pm_v}(t)\big)
\end{align}
as a consequence of~\eqref{eq:supportCutoffTwoPhase}, \eqref{eq:supportCutoffContact},
\eqref{eq:locPropertyCutoff1}--\eqref{eq:locPropertyCutoff4} and~\eqref{eq:defCutoffContactBoundaryWedge}.
The previous display in turn implies~\eqref{eq:coercivityBulkCutoff},
\eqref{eq:coercivityGradientBulkCutoff} and~\eqref{eq:partitionOfUnityBulk}
in $B_{\widehat r}\big(\mathcal{T}_c(t)\big) 
\setminus \big(W^c_{\mathcal{T}_i}(t) \cup W^c_{\Omega^\pm_v}(t)\big)$ 
because of~\eqref{eq:cutOffProfileContact}.
This eventually concludes the proof of Lemma~\ref{lem:partitionOfUnity}.
\end{proof}

\begin{construction}[From local to global extensions]
\label{const:globalCalibrations}
Let $d=2$, and let $\Omega\subset\R^2$ be a bounded domain with orientable and smooth boundary.
Let $(\chi_v,v)$ be a strong solution to the incompressible
Navier--Stokes equation for two fluids in the sense of Definition~\ref{Def_strongsol}
on a time interval $[0,T]$. Let $(\eta_1,\ldots,\eta_N)$ be a partition
of unity along the interface~$I_v$ as given by the proof of Lemma~\ref{lem:partitionOfUnity}.
For each two-phase interface~$i\in\mathcal{I}$ denote by~$\xi^{i}$ the
bulk extension of Proposition~\ref{prop:CalibrationInterface}, and for
each contact point~$c\in\mathcal{C}$ let~$\xi^{c}$ be the contact point
extension of Proposition~\ref{prop:CalibrationContactPoint}. 

We then define a vector field~$\xi\colon \overline{\Omega} \times [0,T] \to \R^2$
with regularity
\begin{align}
\label{eq:regXi}
\xi \in \big(C^0_tC^2_x\cap C^1_tC^0_x \big)\big(\overline{\Omega{\times}[0,T]}
			\setminus(I_v\cap(\partial\Omega{\times}[0,T]))\big)
\end{align}
by means of the formula
\begin{align}
\label{eq:defGlobalCalibration}
\xi := \sum_{n=1}^N \eta_n \xi^n.
\end{align}
\end{construction}

Before we proceed on with a proof of Proposition~\ref{prop:existenceCalibration},
we first deduce that the bulk cutoff~$\eta_{\mathrm{bulk}}$ 
of Lemma~\ref{lem:partitionOfUnity} is transported by the fluid velocity~$v$ 
up to an admissible error in the distance to the interface of the strong solution.

\begin{lemma}[Transport equation for bulk cutoff]
\label{lem:transportCutOffs}
Let $d=2$, and let $\Omega\subset\R^2$ be a bounded domain with orientable and smooth boundary.
Let $(\chi_v,v)$ be a strong solution to the incompressible
Navier--Stokes equation for two fluids in the sense of Definition~\ref{Def_strongsol}
on a time interval $[0,T]$. Let $(\eta_1,\ldots,\eta_N)$ be a partition
of unity along the interface $I_v$ as given by the proof of Lemma~\ref{lem:partitionOfUnity}.

The bulk cutoff $\eta_{\mathrm{bulk}} = 1 - \sum_{n=1}^N \eta_n$ is
then transported by the fluid velocity~$v$ to second order in form of
\begin{align}
\label{eq:transportBulkCutoff}
|\partial_t \eta_{\mathrm{bulk}} + (v\cdot\nabla)\eta_{\mathrm{bulk}}|
\leq C(1 \wedge \dist^2(\cdot,I_v))
\quad\text{in } \Omega\times [0,T].
\end{align}
\end{lemma}

\begin{proof}
Let $\widehat r\in (0,\frac{1}{2}]$ be the localization radius of Lemma~\ref{lem:partitionOfUnity}.
In view of the regularity estimate~\eqref{eq:regularityCutoff}
and the fact that
\begin{align*}
\Omega \setminus \bigg(\bigcup_{c\in\mathcal{C}} B_{\widehat r}(\mathcal{T}_c(t))
\cup \bigcup_{i\in\mathcal{I}} \mathrm{im}(\Psi_{\mathcal{T}_i})\bigg)
\subset \Omega\cap\big\{x\in\R^2\colon \dist(x,I_v(t)) > \widehat r\big\}
\end{align*}
for all $t\in [0,T]$,
it suffices to establish~\eqref{eq:transportBulkCutoff} within
$\Omega\cap\Psi_{\mathcal{T}_i}\big(\mathcal{T}_i(t){\times}\{t\}{\times}[-\widehat r,\widehat r]\big)
\setminus \bigcup_{c\in\mathcal{C}} B_{\widehat r}\big(\mathcal{T}_c(t)\big)$
or $\Omega\cap B_{\widehat r}(\mathcal{T}_c(t))$ for all $i\in\mathcal{I}$, all $c\in\mathcal{C}$
and all $t\in [0,T]$.

\textit{Step 1: (Estimate near the interface but away from contact points)}
Fix a two-phase interface $i\in\mathcal{I}$.
As a consequence of the two identities in~\eqref{eq:representationBulkCutoff1}, 
we may compute
\begin{align}
\label{eq:transportBulkCutoffAux1}
\partial_t \eta_{\mathrm{bulk}} + (v\cdot\nabla)\eta_{\mathrm{bulk}}
= -\big(\partial_t \zeta_i + (v\cdot\nabla)\zeta_i\big)
+ \eta_{\mathrm{bulk}}(v\cdot\nabla)\zeta_i
\end{align}
in $\Omega\cap\Psi_{\mathcal{T}_i}\big(\mathcal{T}_i(t){\times}\{t\}{\times}[-\widehat r,\widehat r]\big)
\setminus \bigcup_{c\in\mathcal{C}} B_{\widehat r}\big(\mathcal{T}_c(t)\big)$
for all $t\in [0,T]$. Recall that the signed distance
to the two-phase interface~$\mathcal{T}_i\subset I_v$ is transported to
first order by the fluid velocity~$v$, and that the profile $\zeta$ from~\eqref{eq:quadraticProfileBuildingBlock}
is quadratic around the origin. Hence, by the chain rule and the
definition~\eqref{eq:cutOffProfileTwoPhase} we obtain
\begin{align}
\label{eq:transportProfileTwoPhase}
\big|\partial_t \zeta_i + (v\cdot\nabla)\zeta_i\big|
\leq C\dist^2(\cdot,I_v)
\quad\text{in } \Omega\cap\Psi_{\mathcal{T}_i}
\big(\mathcal{T}_i(t){\times}\{t\}{\times}[-\widehat r,\widehat r]\big)
\end{align}
for all $t\in [0,T]$. Since we also have the coercivity estimate~\eqref{eq:coercivityBulkCutoff}
for the bulk cutoff at our disposal, we may thus upgrade~\eqref{eq:transportBulkCutoffAux1}
to~\eqref{eq:transportBulkCutoff} in $\Omega\cap\Psi_{\mathcal{T}_i}
\big(\mathcal{T}_i(t){\times}\{t\}{\times}[-\widehat r,\widehat r]\big)
\setminus \bigcup_{c\in\mathcal{C}} B_{\widehat r}\big(\mathcal{T}_c(t)\big)$
for all $t\in [0,T]$. 

\textit{Step 2: (Estimate near contact points, part I)}
Fix $c\in\mathcal{C}$, and denote by $i\in\mathcal{I}$
the unique two-phase interface such that $i\sim c$.
This step is devoted to the proof of~\eqref{eq:transportBulkCutoff}
in the wedge $\Omega\cap B_{\widehat r}(\mathcal{T}_c(t)) \cap W^c_{\mathcal{T}_i}(t)$
containing the interface $\mathcal{T}_i(t)\subset I_v(t)$, $t\in [0,T]$.
Because of~\eqref{eq:representationBulkCutoff5}, \eqref{eq:representationBulkCutoff2} 
and~\eqref{eq:defGlobalCalibration} we have
\begin{equation}
\begin{aligned}
\label{eq:transportBulkCutoffAux2}
&\partial_t \eta_{\mathrm{bulk}} + (v\cdot\nabla)\eta_{\mathrm{bulk}}
= - \big(\partial_t \zeta_i + (v\cdot\nabla)\zeta_i\big)
+ \eta_{\mathrm{bulk}} (v\cdot\nabla)\zeta_i
\end{aligned}
\end{equation}
in $\Omega\cap B_{\widehat r}(\mathcal{T}_c(t)) \cap W^c_{\mathcal{T}_i}(t)$ for all $t\in [0,T]$.
Due to Definition~\ref{def:locRadius} of the admissible localization radius $r_c$
and $\widehat r \leq r_c$ by Lemma~\ref{lem:partitionOfUnity}, it holds
$B_{\widehat r}(\mathcal{T}_c(t)) \cap W^c_{\mathcal{T}_i}(t) \subset 
\Psi_{\mathcal{T}_i}\big(\mathcal{T}_i(t){\times}\{t\}{\times}[-\widehat r,\widehat r]\big)$
for all $t\in [0,T]$. In particular, the estimate~\eqref{eq:transportProfileTwoPhase}
is applicable in $\Omega\cap B_{\widehat r}(\mathcal{T}_c(t)) \cap W^c_{\mathcal{T}_i}(t)$ for all $t\in [0,T]$.
Hence, the estimate~\eqref{eq:transportProfileTwoPhase} 
in combination with the coercivity 
estimate~\eqref{eq:coercivityBulkCutoff} for the bulk cutoff allow to
deduce~\eqref{eq:transportBulkCutoff} from~\eqref{eq:transportBulkCutoffAux2}
in $\Omega\cap B_{\widehat r}(\mathcal{T}_c(t)) \cap W^c_{\mathcal{T}_i}(t)$ for all $t\in [0,T]$.

\textit{Step 3: (Estimate near contact points, part II)}
Fix a contact point $c\in\mathcal{C}$. The goal of this step
is to prove~\eqref{eq:transportBulkCutoff} in the wedges
$\Omega \cap B_{\widehat r}(\mathcal{T}_c(t)) \cap W^{\pm,c}_{\partial\Omega}(t)$
containing the boundary $\partial\Omega$ for all $t \in [0,T]$.
To this end, it follows from~\eqref{eq:representationBulkCutoff4}
and~\eqref{eq:defGlobalCalibration} that
\begin{align}
\label{eq:transportBulkCutoffAux3}
&\partial_t \eta_{\mathrm{bulk}} + (v\cdot\nabla)\eta_{\mathrm{bulk}}
= -\big(\partial_t \zeta_c + (v\cdot\nabla)\zeta_c\big)
+ \eta_{\mathrm{bulk}}(v\cdot\nabla)\zeta_c
\end{align}
in $\Omega \cap B_{\widehat r}(\mathcal{T}_c(t)) \cap W^{\pm,c}_{\partial\Omega}(t)$ for all $t \in [0,T]$.
Note that because of~\eqref{eq:quadraticProfileBuildingBlock} 
one can view the profile $\zeta_c$ from~\eqref{eq:cutOffProfileContact} 
as a smooth function of the contact point $\mathcal{T}_c$. 
Performing a slight
yet convenient abuse of notation $\mathcal{T}_c(t)=\{c(t)\}$, we obtain as a consequence
of $\frac{\mathrm{d}}{\mathrm{d}t} c(t) = v(c(t),t)$ and 
an application of the chain rule that $\partial_t\zeta_c(\cdot,t) 
+ \big(v(c(t),t)\cdot\nabla\big)\zeta_c(\cdot,t) = 0$ 
at $c(t)$ for all $t\in [0,T]$. 
Furthermore, proceeding similarly as done in the proof of \cite[Lemma 11]{Fischer2020}, we can also deduce that 
$\partial_t\zeta_c(\cdot,t) 
+ \big(v(c(t),t)\cdot\nabla\big)\zeta_c(\cdot,t) = 0$ 
in $\Omega\cap B_{\widehat r}(\mathcal{T}_c(t))$ for all $t\in [0,T]$. 
By the regularity of the fluid
velocity~$v$, this in turn implies by adding zero (and exploiting the 
quadratic behaviour of the profile $\zeta$ from~\eqref{eq:quadraticProfileBuildingBlock} 
around the origin) that
\begin{align}
\label{eq:transportProfileContactAux}
\big|\partial_t \zeta_c + (v\cdot\nabla)\zeta_c\big|
\leq C\dist^2(\cdot,\mathcal{T}_c)
\quad\text{in } \Omega\cap B_{\widehat r}(\mathcal{T}_c(t))
\end{align}
for all $t\in [0,T]$. Since $\widehat r \leq r_c$ by Lemma~\ref{lem:partitionOfUnity},
we can infer from Definition~\ref{def:locRadius} of the
admissible localization radius $r_c$ that $\dist(\cdot,\mathcal{T}_c)$
is dominated by $\dist(\cdot,I_v)$ in $B_{\widehat r}(\mathcal{T}_c(t)) 
\cap \big(W^{\pm,c}_{\partial\Omega}(t) \cup W^c_{\Omega^\pm_v}(t)\big)$
for all $t\in [0,T]$. Hence, we deduce from~\eqref{eq:transportProfileContactAux} that
\begin{align}
\label{eq:transportProfileContact}
\big|\partial_t \zeta_c + (v\cdot\nabla)\zeta_c\big|
\leq C\dist^2(\cdot,I_v)
\,\text{in } \Omega\cap B_{\widehat r}(\mathcal{T}_c(t)) 
\cap \big(W^{\pm,c}_{\partial\Omega}(t) \cup W^c_{\Omega^\pm_v}(t)\big)
\end{align}
for all $t\in [0,T]$. Inserting the estimate~\eqref{eq:transportProfileContact}
and the coercivity estimate~\eqref{eq:coercivityBulkCutoff}
for the bulk cutoff into~\eqref{eq:transportBulkCutoffAux3} thus
yields~\eqref{eq:transportBulkCutoff} 
in $\Omega \cap B_{\widehat r}(\mathcal{T}_c(t)) \cap W^{\pm,c}_{\partial\Omega}(t)$ for all $t \in [0,T]$.

\textit{Step 4: (Estimate near contact points, part III)}
Fix $c\in\mathcal{C}$, and denote by $i\in\mathcal{I}$
the unique two-phase interface such that $i\sim c$.
We aim to verify~\eqref{eq:transportBulkCutoff} in
the interpolation wedges $\Omega\cap B_{\widehat r}(\mathcal{T}_c(t)) \cap W^c_{\Omega^\pm_v}(t)$
for all $t\in [0,T]$. To this end, we may employ~\eqref{eq:representationBulkCutoff5}, 
\eqref{eq:representationBulkCutoff3} and~\eqref{eq:defGlobalCalibration} to argue that
\begin{equation}
\begin{aligned}
\label{eq:transportBulkCutoffAux4}
&\partial_t \eta_{\mathrm{bulk}} + (v\cdot\nabla)\eta_{\mathrm{bulk}}
\\&
= - \lambda^\pm_c \Big\{\big(\partial_t \zeta_i + (v\cdot\nabla)\zeta_i\big)
- \eta_{\mathrm{bulk}} (v\cdot\nabla)\zeta_i\Big\}
\\&~~~
- (1{-}\lambda^\pm_c) \Big\{\big(\partial_t \zeta_c + (v\cdot\nabla)\zeta_c\big)
- \eta_{\mathrm{bulk}} (v\cdot\nabla)\zeta_c\Big\}
\\&~~~
+ \big(\partial_t\lambda^\pm_c + (v\cdot\nabla)\lambda^\pm_c\big) (\zeta_c - \zeta_i)
\end{aligned}
\end{equation}
in $\Omega\cap B_{\widehat r}(\mathcal{T}_c(t)) \cap W^c_{\Omega^\pm_v}(t)$ for all $t\in [0,T]$.
Due to Definition~\ref{def:locRadius} of the admissible localization radius $r_c$
and $\widehat r \leq r_c$ by Lemma~\ref{lem:partitionOfUnity}, it holds
$B_{\widehat r}(\mathcal{T}_c(t)) \cap W^c_{\Omega^\pm_v}(t) \subset 
\Psi_{\mathcal{T}_i}\big(\mathcal{T}_i(t){\times}\{t\}{\times}[-\widehat r,\widehat r]\big)$
for all $t\in [0,T]$. The estimates~\eqref{eq:transportProfileTwoPhase} 
and~\eqref{eq:coercivityBulkCutoff} therefore imply that the first term on the right hand side
of~\eqref{eq:transportBulkCutoffAux4} is of required order. For the second term on the right hand side
of~\eqref{eq:transportBulkCutoffAux4}, we may instead rely on the estimates~\eqref{eq:transportProfileContact}
and~\eqref{eq:coercivityBulkCutoff}.

Note that in view of the definitions~\eqref{eq:quadraticProfileBuildingBlock}--\eqref{eq:cutOffProfileContact},
the auxiliary cutoffs $\zeta_i$ and $\zeta_c$ are compatible to second order in the sense that
$|\zeta_i - \zeta_c| \leq C\dist^2(\cdot,\mathcal{T}_c)$
in $\Omega\cap B_{\widehat r}(\mathcal{T}_c(t)) \cap W^c_{\Omega^\pm_v}(t)$ for all $t\in [0,T]$. 
Recall from the previous step that $\dist(\cdot,\mathcal{T}_c)$
is dominated by $\dist(\cdot,I_v)$ in $B_{\widehat r}(\mathcal{T}_c(t)) 
\cap \big(W^{\pm,c}_{\partial\Omega}(t) \cup W^c_{\Omega^\pm_v}(t)\big)$
for all $t\in [0,T]$. Hence,
\begin{align}
\label{eq:compCutoffProfiles}
|\zeta_i - \zeta_c| \leq C\dist^2(\cdot,I_v)
\end{align}
in $\Omega\cap B_{\widehat r}(\mathcal{T}_c(t)) \cap W^c_{\Omega^\pm_v}(t)$ for all $t\in [0,T]$.
In particular, together with~\eqref{eq:AdvectiveDerivativeLambda} the bound~\eqref{eq:compCutoffProfiles}
allows to upgrade~\eqref{eq:transportBulkCutoffAux4}
to the desired estimate~\eqref{eq:transportBulkCutoff} 
in $\Omega\cap B_{\widehat r}(\mathcal{T}_c(t)) \cap W^c_{\Omega^\pm_v}(t)$ for all $t\in [0,T]$.

\textit{Step 5: (Conclusion)} Recall from Definition~\ref{def:locRadius} of the admissible localization radius $r_c$
that for all $t\in[0,T]$ the set $\Omega \cap B_{r_c}(\mathcal{T}_c(t))$ is decomposed
by means of the five pairwise disjoint open wedges $W^{\pm,c}_{\partial\Omega}(t), W^c_{\mathcal{T}_i}(t),
W^c_{\Omega^\pm_v}(t)\subset\R^2$. Hence, the previous three steps entail
the validity of~\eqref{eq:transportBulkCutoff} in $\Omega \cap B_{r_c}(\mathcal{T}_c(t))$ for all $t\in[0,T]$.
In particular, based on the discussion at the beginning of this proof and the argument in 
the vicinity of the interface but away from contact points (see \textit{Step 1}), we may
conclude the proof of Lemma~\ref{lem:partitionOfUnity}.
\end{proof}

\subsection{Proof of Proposition~\ref{prop:existenceCalibration}}
\label{eq:proofExistenceCalibration}
All ingredients are in place to proceed with the proof of the main result of this section,
i.e., that the vector field~$\xi$ of Construction~\ref{const:globalCalibrations}
gives rise to a boundary adapted extension of the interface unit normal
for two-phase fluid flow in the sense of Definition~\ref{def:calibrationTwoPhaseFluidFlow}
with respect to~$(\chi_v,v)$.

\begin{proof}[Proof of~\eqref{eq:coercivityByModulationOfLength}]
This is an easy consequence of the lower bound in the coercivity 
estimate~\eqref{eq:coercivityBulkCutoff} for the bulk cutoff,
the definition~\eqref{eq:defGlobalCalibration} of the global vector field $\xi$,
the fact that the local vector fields $(\xi^n)_{n\in\{1,\ldots,N\}}$
as provided by Proposition~\ref{prop:CalibrationInterface} and 
Proposition~\ref{prop:CalibrationContactPoint} are of unit length, 
and the triangle inequality in form of
$|\xi| = | \sum_{n=1}^N \eta_n \xi_n |
\leq \sum_{n=1}^N \eta_n |\xi^n| 
= \sum_{n=1}^N \eta_n = 1 - \eta_{\mathrm{bulk}}$
in $\Omega\times [0,T]$.
\end{proof}

\begin{proof}[Proof of~\eqref{eq:boundaryValueXi}]
By definition~\eqref{eq:defGlobalCalibration} of the candidate extension~$\xi$
and the localization properties~\eqref{eq:supportCutoffTwoPhase}--\eqref{eq:LocalizationSupportTwoPhaseContact}
of the partition of unity $(\eta_1,\ldots,\eta_N)$ from Lemma~\ref{lem:partitionOfUnity},
it suffices to verify~\eqref{eq:boundaryValueXi}
in terms of $\xi = \eta_c\xi^c$
in the associated region $B_{\widehat r}(\mathcal{T}_c(t))\cap\partial\Omega$ 
for all contact points $c\in\mathcal{C}$ and all $t\in [0,T]$.
However, this in turn is an immediate consequence of Proposition~\ref{prop:CalibrationContactPoint}.  
\end{proof}

\begin{proof}[Proof of~\eqref{eq:divConstraintXi}]
For a proof of~\eqref{eq:divConstraintXi}, we start computing
based on the definition~\eqref{eq:defGlobalCalibration} of the global vector field $\xi$
that $\nabla \cdot \xi = \sum_{n=1}^N \eta_n\nabla\cdot\xi^n 
+ \sum_{n=1}^N (\xi^n\cdot\nabla)\eta_n$.
As a consequence of the corresponding local versions of~\eqref{eq:divConstraintXi}
from Proposition~\ref{prop:CalibrationInterface} and 
Proposition~\ref{prop:CalibrationContactPoint}, and the fact that
$(\eta_1,\ldots,\eta_n)$ is a partition of unity along the interface $I_v$
by Lemma~\ref{lem:partitionOfUnity} we obtain
$\sum_{n=1}^N \eta_n\nabla\cdot\xi^n = - H_{I_v}
\quad\text{along } I_v \cap \Omega$.
Moreover, by adding zero and subsequently relying on 
the definition~\eqref{eq:defGlobalCalibration} of the global vector field $\xi$,
the localization properties~\eqref{eq:supportCutoffTwoPhase}--\eqref{eq:LocalizationSupportTwoPhaseContact}
of the partition of unity $(\eta_1,\ldots,\eta_N)$ from Lemma~\ref{lem:partitionOfUnity},
the compatibility estimate~\eqref{eq:compBoundLocalCalibrations1} and the
estimates~\eqref{eq:coercivityBulkCutoff} and~\eqref{eq:coercivityGradientBulkCutoff} 
for the bulk cutoff we may infer that
\begin{align*}
\sum_{n=1}^N (\xi^n\cdot\nabla)\eta_n 
& = - (\xi\cdot\nabla)\eta_{\mathrm{bulk}} - \sum_{n=1}^N((\xi - \xi^n)\cdot\nabla)\eta_n  \\
&= - (\xi\cdot\nabla)\eta_{\mathrm{bulk}}
+ \eta_{\mathrm{bulk}}\sum_{n=1}^N(\xi^n\cdot\nabla)\eta_n
\\&~~~
+ \sum_{i\in\mathcal{I}}\sum_{c\in\mathcal{C},i\sim c}
\eta_c\big((\xi^{i}{-}\xi^c)\cdot\nabla\eta_i)
+ \sum_{c\in\mathcal{C}}\sum_{i\in\mathcal{I},i\sim c}
\eta_i\big((\xi^{c}{-}\xi^i)\cdot\nabla\eta_c)
\\&
= O(1 \wedge \dist(\cdot,I_v)) \quad\text{in } \Omega\times [0,T].
\end{align*}
In summary, we thus obtain~\eqref{eq:divConstraintXi}. 
\end{proof}

\begin{proof}[Proof of~\eqref{eq:timeEvolutionXi}]
For a proof of~\eqref{eq:timeEvolutionXi}, we start estimating
based on the definition~\eqref{eq:defGlobalCalibration} of the global vector field $\xi$
as well as the corresponding local versions of~\eqref{eq:timeEvolutionXi}
from Proposition~\ref{prop:CalibrationInterface} and 
Proposition~\ref{prop:CalibrationContactPoint}
\begin{align}
\nonumber
\partial_t \xi &= \sum_{n=1}^N \eta_n \partial_t\xi^n 
+ \sum_{n=1}^N \xi^n \partial_t\eta_n
\\& \label{eq:transportXiProof}
= - \sum_{n=1}^N \eta_n (v\cdot\nabla)\xi^n
+ \sum_{n=1}^N \xi^n \partial_t\eta_n
\\&~~~ \nonumber
- \sum_{n=1}^N \eta_n (\mathrm{Id}{-}\xi^n\otimes\xi^n)(\nabla v)^\mathsf{T}\xi^n 
+ O(1 \wedge \dist(\cdot,I_v))
\quad\text{in } \Omega\times [0,T].
\end{align}
Adding zero twice and applying the product rule, we may further rewrite
based on the definition~\eqref{eq:defGlobalCalibration} of the candidate extension~$\xi$ and
the localization properties~\eqref{eq:supportCutoffTwoPhase}--\eqref{eq:LocalizationSupportTwoPhaseContact}
of the partition of unity $(\eta_1,\ldots,\eta_N)$ from Lemma~\ref{lem:partitionOfUnity}
\begin{align*}
&-\sum_{n=1}^N \eta_n (v\cdot\nabla)\xi^n + \sum_{n=1}^N \xi^n \partial_t\eta_n
\\&
= - (v\cdot\nabla)\xi
+ \sum_{n=1}^N \xi^n \big(\partial_t \eta_n + (v\cdot\nabla)\eta_n\big)
\\&
= - (v\cdot\nabla)\xi 
- \xi\big(\partial_t \eta_{\mathrm{bulk}} + (v\cdot\nabla)\eta_{\mathrm{bulk}}\big)
+ \sum_{n=1}^N (\xi^n {-} \xi) \big(\partial_t \eta_n + (v\cdot\nabla)\eta_n\big)
\\&
= - (v\cdot\nabla)\xi 
- \xi\big(\partial_t \eta_{\mathrm{bulk}} + (v\cdot\nabla)\eta_{\mathrm{bulk}}\big)
+ \eta_{\mathrm{bulk}} \sum_{n=1}^N \xi^n \big(\partial_t \eta_n + (v\cdot\nabla)\eta_n\big)
\\&~~~
+ \sum_{i\in\mathcal{I}}\sum_{c\in\mathcal{C},i\sim c}
\eta_c(\xi^i {-} \xi^c) \big(\partial_t \eta_i {+} (v\cdot\nabla)\eta_i\big)
+ \sum_{c\in\mathcal{C}}\sum_{i\in\mathcal{I},i\sim c}
\eta_i(\xi^c {-} \xi^i) \big(\partial_t \eta_c {+} (v\cdot\nabla)\eta_c\big)
\end{align*}
in $\Omega\times [0,T]$. 
Hence, estimating based on the compatibility estimate~\eqref{eq:compBoundLocalCalibrations1} 
as well as the 
estimates~\eqref{eq:coercivityBulkCutoff} and~\eqref{eq:transportBulkCutoff} 
for the bulk cutoff yields the bound
\begin{equation}
\label{eq:transportXiProof1}
\begin{aligned}
&-\sum_{n=1}^N \eta_n (v\cdot\nabla)\xi^n + \sum_{n=1}^N \xi^n \partial_t\eta_n
= - (v\cdot\nabla)\xi + O(1 \wedge \dist(\cdot,I_v)) \text{ in } \Omega\times [0,T].
\end{aligned}
\end{equation}
Adding zero twice and making use
of the definition~\eqref{eq:defGlobalCalibration} of the candidate extension~$\xi$ together with
the localization properties~\eqref{eq:supportCutoffTwoPhase}--\eqref{eq:LocalizationSupportTwoPhaseContact}
of the partition of unity $(\eta_1,\ldots,\eta_N)$ from Lemma~\ref{lem:partitionOfUnity},
we next compute
\begin{align} \label{eq_xinTensProd}
&\mathds{1}_{\supp\eta_n} \xi^n\otimes\xi^n
 \\&
= \mathds{1}_{\supp\eta_n} \xi\otimes\xi
+ \mathds{1}_{\supp\eta_n} (\xi^n{-}\xi)\otimes\xi^n
+ \mathds{1}_{\supp\eta_n} \xi\otimes(\xi^n{-}\xi) 
\nonumber\\&
= \mathds{1}_{\supp\eta_n} \xi\otimes\xi
\nonumber\\&~~~
+ \mathds{1}_{\supp\eta_n} \eta_{\mathrm{bulk}} \xi^n\otimes\xi^n
+ \mathds{1}_{\supp\eta_n} \eta_{\mathrm{bulk}} \xi\otimes\xi^n
\nonumber\\&~~~
+ \mathds{1}_{n=i\in\mathcal{I}} \mathds{1}_{\supp\eta_i} 
\sum_{c\in\mathcal{C},i\sim c} \eta_c(\xi^i{-}\xi^c)\otimes \xi^i
+ \mathds{1}_{n=c\in\mathcal{C}} \mathds{1}_{\supp\eta_c}
\sum_{i\in\mathcal{I},i\sim c} \eta_i(\xi^c{-}\xi^i)\otimes \xi^c
\nonumber\\&~~~
+ \mathds{1}_{n=i\in\mathcal{I}} \mathds{1}_{\supp\eta_i} 
\sum_{c\in\mathcal{C},i\sim c} \eta_c \xi \otimes (\xi^i{-}\xi^c)
+ \mathds{1}_{n=c\in\mathcal{C}} \mathds{1}_{\supp\eta_c}
\sum_{i\in\mathcal{I},i\sim c} \eta_i \xi \otimes (\xi^c{-}\xi^i) \nonumber
\end{align}
in $\Omega\times [0,T]$.
Relying on the same ingredients as for the previous computation we also have
\begin{align*}
- \sum_{n=1}^N \eta_n (\nabla v)^\mathsf{T}\xi^n
&= - (\nabla v)^\mathsf{T}\xi 
- \sum_{n=1}^N \eta_n (\nabla v)^\mathsf{T}(\xi^n {-} \xi)
+ \eta_{\mathrm{bulk}} (\nabla v)^\mathsf{T}\xi
\\&
= - (\nabla v)^\mathsf{T}\xi 
+ \eta_{\mathrm{bulk}} (\nabla v)^\mathsf{T}\xi
- \eta_{\mathrm{bulk}} \sum_{n=1}^N \eta_n (\nabla v)^\mathsf{T}\xi^n
\\&~~~
- \sum_{i\in\mathcal{I}}\sum_{c\in\mathcal{C},i\sim c}
\eta_i\eta_c (\nabla v)^\mathsf{T}(\xi^i {-} \xi^c)
- \sum_{c\in\mathcal{C}}\sum_{i\in\mathcal{I},i\sim c}
\eta_c\eta_i (\nabla v)^\mathsf{T}(\xi^c {-} \xi^i)
\end{align*}
in $\Omega\times [0,T]$. 
The compatibility estimate~\eqref{eq:compBoundLocalCalibrations1} 
as well as the 
estimates~\eqref{eq:coercivityBulkCutoff} and~\eqref{eq:transportBulkCutoff} 
therefore imply in view of the previous two displays that
\begin{equation}
\label{eq:transportXiProof2}
\begin{aligned}
&- \sum_{n=1}^N \eta_n (\mathrm{Id}{-}\xi^n\otimes\xi^n)(\nabla v)^\mathsf{T}\xi^n
\\&
= - (\mathrm{Id}{-}\xi\otimes\xi)(\nabla v)^\mathsf{T}\xi 
+ O(1 \wedge \dist(\cdot,I_v)) \quad\text{in } \Omega\times [0,T]. 
\end{aligned}
\end{equation}
The combination of the bounds~\eqref{eq:transportXiProof}--\eqref{eq:transportXiProof2} 
now immediately entails the desired estimate~\eqref{eq:timeEvolutionXi} on the time 
evolution of the global vector field $\xi$.
\end{proof}

\begin{proof}[Proof of~\eqref{eq:timeEvolutionLengthXi}]
We get as a consequence of the product rule and inserting the local
versions of~\eqref{eq:timeEvolutionLengthXi} from Proposition~\ref{prop:CalibrationInterface} and 
Proposition~\ref{prop:CalibrationContactPoint}
\begin{align*}
\xi\cdot\partial_t \xi 
&= \sum_{n=1}^N \eta_n \xi \cdot \partial_t\xi^n
+ \sum_{n=1}^N (\xi \cdot \xi^n) \partial_t\eta_n
\\&
= - \sum_{n=1}^N \eta_n \xi^n \cdot (v \cdot \nabla) \xi^n
+ \sum_{n=1}^N \eta_n (\xi {-} \xi^n) \cdot \partial_t\xi^n
\\&~~~
+ \sum_{n=1}^N (\xi \cdot \xi^n) \partial_t\eta_n
+ O(\dist(\cdot,I_v)^2 \wedge 1)
\quad\text{in } \Omega\times [0,T].
\end{align*}
Adding zero to produce the left hand sides of the local versions 
of~\eqref{eq:timeEvolutionXi} from Proposition~\ref{prop:CalibrationInterface} and 
Proposition~\ref{prop:CalibrationContactPoint} 
further updates the previous display to 
\begin{align*}
\xi\cdot\partial_t \xi 
&= - \sum_{n=1}^N \eta_n \xi \cdot (v \cdot \nabla) \xi^n + \sum_{n=1}^N (\xi \cdot \xi^n) \partial_t\eta_n
\\&~~~
- \sum_{n=1}^N \eta_n (\xi {-} \xi^n) \cdot (\Id - \xi^n \otimes \xi^n)(\nabla v)^\mathsf{T}\xi^n
\\&~~~
+ \sum_{n=1}^N \eta_n (\xi {-} \xi^n) \cdot 
\big(\partial_t\xi^n {+} (v\cdot\nabla)\xi^n {+} (\Id - \xi^n \otimes \xi^n)(\nabla v)^\mathsf{T}\xi^n\big)
\\&~~~
+ O(\dist(\cdot,I_v)^2 \wedge 1)
\quad\text{in } \Omega\times [0,T].
\end{align*}
We then continue with adding zeros to obtain
\begin{equation}
\label{eq:auxComputationGlobalLenghtXi}
\begin{aligned}
\xi\cdot\partial_t \xi 
&= - \xi \cdot (v\cdot\nabla) \xi
\\&~~~
+ \sum_{n=1}^N \big(\xi \cdot (\xi^n {-} \xi)\big) \big(\partial_t\eta_n {+} (v\cdot\nabla)\eta_n\big)
- |\xi|^2 \big(\partial_t\eta_{\mathrm{bulk}} {+} (v\cdot\nabla)\eta_{\mathrm{bulk}}\big)
\\&~~~
- \sum_{n=1}^N \eta_n (\xi {-} \xi^n) \cdot (\xi \otimes \xi  - \xi^n \otimes \xi^n)(\nabla v)^\mathsf{T}\xi^n
\\&~~~
- \sum_{n=1}^N \eta_n (\xi {-} \xi^n) \cdot (\Id - \xi \otimes \xi)(\nabla v)^\mathsf{T}(\xi^n - \xi)
\\&~~~
+ \sum_{n=1}^N \eta_n (\xi {-} \xi^n) \cdot 
\big(\partial_t\xi^n {+} (v\cdot\nabla)\xi^n {+} (\Id - \xi^n \otimes \xi^n)(\nabla v)^\mathsf{T}\xi^n\big)
\\&~~~
+ O(\dist(\cdot,I_v)^2 \wedge 1)
\quad\text{in } \Omega\times [0,T].
\end{aligned}
\end{equation}
As it is by now routine, we may employ the localization 
properties~\eqref{eq:supportCutoffTwoPhase}--\eqref{eq:LocalizationSupportTwoPhaseContact}
of the partition of unity $(\eta_1,\ldots,\eta_N)$ from Lemma~\ref{lem:partitionOfUnity}
and the estimates~\eqref{eq:coercivityBulkCutoff} and~\eqref{eq:transportBulkCutoff}
for the bulk cutoff to reduce the task of estimating the right hand side terms of~\eqref{eq:auxComputationGlobalLenghtXi}
to an application of the compatibility 
estimates~\eqref{eq:compBoundLocalCalibrations1}--\eqref{eq:compBoundLocalCalibrations2}.
More precisely, we obtain by straightforward applications of these two ingredients that
\begin{align}
\nonumber
&\sum_{n=1}^N \big(\xi \cdot (\xi {-} \xi^n)\big) \big(\partial_t\eta_n {+} (v\cdot\nabla)\eta_n\big)
\\& \label{eq:auxComputationGlobalLenghtXi2}
= \sum_{i\in\mathcal{I}} \sum_{c\in\mathcal{C},i\sim c} \eta_c^2
\big((\xi^c {-} \xi^i) \cdot (\xi^c {-} \xi^i)\big) \big(\partial_t\eta_i {+} (v\cdot\nabla)\eta_i\big)
\\&~~~~~ \nonumber
+ \sum_{c\in\mathcal{C}} \sum_{i\in\mathcal{I},i\sim c} \eta_c \eta_i
\big((\xi^c - \xi^i) \cdot (\xi^i {-} \xi^c)\big) \big(\partial_t\eta_c {+} (v\cdot\nabla)\eta_c\big)
\\&~~~~~ \nonumber
+\sum_{i\in\mathcal{I}} \sum_{c\in\mathcal{C},i\sim c} \eta_c^2
\big( \xi^i \cdot (\xi^c {-} \xi^i)\big) \big(\partial_t\eta_i {+} (v\cdot\nabla)\eta_i\big)
\\&~~~~~ \nonumber
+ \sum_{c\in\mathcal{C}} \sum_{i\in\mathcal{I},i\sim c} \eta_c \eta_i
\big(\xi^i \cdot (\xi^i {-} \xi^c)\big) \big(\partial_t\eta_c {+} (v\cdot\nabla)\eta_c\big)
\\&~~~~~ \nonumber
+ \sum_{i\in\mathcal{I}} \sum_{c\in\mathcal{C},i\sim c} \eta_i \eta_c
\big(\xi^i \cdot (\xi^c {-} \xi^i)\big) \big(\partial_t\eta_i {+} (v\cdot\nabla)\eta_i\big)
\\&~~~~~ \nonumber
+ \sum_{c\in\mathcal{C}} \sum_{i\in\mathcal{I},i\sim c} \eta_i^2
\big(\xi^i \cdot (\xi^i {-} \xi^c)\big) \big(\partial_t\eta_c {+} (v\cdot\nabla)\eta_c\big)
\\&~~~~~ \nonumber
+ O(\dist(\cdot,I_v)^2 \wedge 1)
\quad\text{in } \Omega\times [0,T],
\\
\nonumber
&\sum_{n=1}^N \eta_n (\xi {-} \xi^n) \cdot (\Id - \xi \otimes \xi )(\nabla v)^\mathsf{T}(\xi {-} \xi^n)
\\& \label{eq:auxComputationGlobalLenghtXi3}
= \sum_{i\in\mathcal{I}} \sum_{c\in\mathcal{C},i\sim c} \eta_i\eta_c^2 
(\xi^c {-} \xi^i) \cdot (\Id - \xi \otimes \xi )(\nabla v)^\mathsf{T}(\xi^c {-} \xi^i)
\\&~~~~~ \nonumber
+ \sum_{c\in\mathcal{C}} \sum_{i\in\mathcal{I},i\sim c} \eta_c\eta_i^2
(\xi^i {-} \xi^c) \cdot (\Id - \xi \otimes \xi )(\nabla v)^\mathsf{T}(\xi^i {-} \xi^c)
\\&~~~~~ \nonumber
+ O(\dist(\cdot,I_v)^2 \wedge 1) \quad\text{in } \Omega\times [0,T],
\\
\nonumber
&\sum_{n=1}^N \eta_n (\xi {-} \xi^n) \cdot 
\big(\partial_t\xi^n {+} (v\cdot\nabla)\xi^n {+} (\Id - \xi^n \otimes \xi^n )(\nabla v)^\mathsf{T}\xi^n\big)
\\& \label{eq:auxComputationGlobalLenghtXi4}
= \sum_{i\in\mathcal{I}} \sum_{c\in\mathcal{C}} \eta_i\eta_c (\xi^c {-} \xi^i) \cdot
\big(\partial_t\xi^i {+} (v\cdot\nabla)\xi^i {+} (\Id - \xi^i \otimes \xi^i )(\nabla v)^\mathsf{T}\xi^i\big)
\\&~~~~~ \nonumber
+ \sum_{c\in\mathcal{C}} \sum_{i\in\mathcal{I},i\sim c} \eta_c\eta_i (\xi^i {-} \xi^c) \cdot
\big(\partial_t\xi^c {+} (v\cdot\nabla)\xi^c {+} (\Id - \xi^c \otimes \xi^c )(\nabla v)^\mathsf{T}\xi^c\big)
\\&~~~~~ \nonumber
+ O(\dist(\cdot,I_v)^2 \wedge 1) \quad\text{in } \Omega\times [0,T],
\end{align}
and finally
\begin{align}
\nonumber
&\sum_{n=1}^N \eta_n (\xi {-} \xi^n) \cdot (\xi \otimes \xi - \xi^n \otimes \xi^n)( \nabla v)^\mathsf{T}\xi^n
\\& \label{eq:auxComputationGlobalLenghtXi6}
= \sum_{i\in\mathcal{I}} \sum_{c\in\mathcal{C},i\sim c} \eta_c (\xi^c {-} \xi^i) 
\cdot  (\xi \otimes \xi - \xi^i \otimes \xi^i)( \nabla v)^\mathsf{T}\xi^i
\\&~~~ \nonumber
+ \sum_{c\in\mathcal{C}} \sum_{i\in\mathcal{I},i\sim c} 
\eta_i(\xi^i {-} \xi^c) \cdot  (\xi \otimes \xi - \xi^c \otimes \xi^c)(\nabla v)^\mathsf{T}\xi^c
\\&~~~ \nonumber
+ O(\dist(\cdot,I_v)^2 \wedge 1) \quad\text{in } \Omega\times [0,T].
\end{align}
We then exploit  the compatibility 
estimates~\eqref{eq:compBoundLocalCalibrations1} and~\eqref{eq:compBoundLocalCalibrations2}
for an estimate of~\eqref{eq:auxComputationGlobalLenghtXi2}, the compatibility
estimate~\eqref{eq:compBoundLocalCalibrations1} for an estimate of~\eqref{eq:auxComputationGlobalLenghtXi3},
the local versions of~\eqref{eq:timeEvolutionXi} from Proposition~\ref{prop:CalibrationInterface} and 
Proposition~\ref{prop:CalibrationContactPoint}
in combination with the compatibility estimate~\eqref{eq:compBoundLocalCalibrations1} 
for an estimate of~\eqref{eq:auxComputationGlobalLenghtXi4}, 
and finally~\eqref{eq_xinTensProd} together with the estimate for the bulk cutoff~\eqref{eq:coercivityBulkCutoff} 
and the compatibility estimate~\eqref{eq:compBoundLocalCalibrations1} to 
estimate~\eqref{eq:auxComputationGlobalLenghtXi6}. In summary,
using also the bound on the advection derivative~\eqref{eq:transportBulkCutoff}
as well as the coercivity estimate~\eqref{eq:coercivityBulkCutoff},
we may upgrade~\eqref{eq:auxComputationGlobalLenghtXi} to the desired estimate~\eqref{eq:timeEvolutionLengthXi}. 
\end{proof}

\section{Existence of transported weights: Proof of Lemma~\ref{lem:existenceTransportedWeight}}
\label{sec:constructionWeight}
We decompose the argument for the construction of a transported
weight $\vartheta$ in the sense of Definition~\ref{def:transportedWeight} in several steps.

\textit{Step 1: (Choice of suitable profiles)}
Let $\bar\vartheta\colon\R\to\R$ be chosen such that it represents a smooth
truncation of the identity in the sense that $\bar\vartheta(r)=r$ for $|r|\leq\frac{1}{2}$, 
$\bar\vartheta(r)= -1$ for $r\leq -1$, $\bar\vartheta(r)=1$ for $r\geq 1$,  
$0\leq\bar\vartheta'\leq 2$ as well as $|\bar\vartheta''|\leq C$.

For each two-phase interface $i\in\mathcal{I}$ present in the interface $I_v$
of the strong solution, we then define an auxiliary weight
\begin{align}
\label{eq:weightProfileTwoPhase}
&\bar\vartheta_i(x,t) := -\bar\vartheta\Big(\frac{\sdist(x,\mathcal{T}_i(t))}{\delta \widehat r}\Big),
\quad
(x,t) \in \mathrm{im}(\Psi_{\mathcal{T}_i})
\end{align}
where the change of variables $\Psi_{\mathcal{T}_i}$ 
and the associated signed distance $\sdist(\cdot,\mathcal{T}_i)$
are the ones from Definition~\ref{def:locRadiusInterface} of the admissible localization radius $r_i$.
Moreover, $\widehat r$ represents the localization scale of Lemma~\ref{lem:partitionOfUnity}
and $\delta \in (0,1]$ denotes a constant to be chosen in the course of the proof.

Recalling also from Definition~\ref{def:locRadius} of the admissible localization radii $(r_c)_{c\in\mathcal{C}}$
the definition of the change of variables $\Psi_{\partial\Omega}$ 
with associated signed distance $\sdist(\cdot,\partial\Omega)$ we define another two auxiliary weights
by means of
\begin{align}
\label{eq:weightProfileBoundary}
&\bar\vartheta_{\partial\Omega}^\pm(x,t) := \mp\bar\vartheta
\Big(\frac{\sdist(x,\partial\Omega)}{\delta \widehat r}\Big),
\\& \nonumber
(x,t) \in \bigcup_{t'\in [0,T]} \big(\Omega^\pm_v(t') \cap
\Psi_{\partial\Omega}\big(\partial\Omega{\times}(-2\widehat r,2\widehat r)\big)\big) {\times} \{t'\}.
\end{align}

\textit{Step 2: (Construction of the transported weight)}
Away from contact points and the interface but in the vicinity of the domain boundary, 
we introduce the following notational shorthand
\begin{align}
\label{eq:contactInterfaceExcluded}
\mathcal{U}_{\widehat r}(t) := \bigcup_{i\in\mathcal{I}}\Psi_{\mathcal{T}_i}
\big(\mathcal{T}_i(t){\times}\{t\}{\times}[-\widehat r,\widehat r]\big) \cup
\bigcup_{c\in\mathcal{C}} B_{\widehat r}\big(\mathcal{T}_c(t)\big),\quad t\in [0,T],
\end{align}
and then define
\begin{align}
\label{eq:defWeightNearBoundary}
&\vartheta(x,t) := \bar\vartheta_{\partial\Omega}^\pm(x,t),
\\& \nonumber
(x,t) \in \bigcup_{t'\in [0,T]} \big(\Omega^\pm_v(t') \cap
\Psi_{\partial\Omega}\big(\partial\Omega{\times}[-\widehat r,\widehat r]\big)
\setminus \mathcal{U}_{\widehat r}(t')  \big) {\times} \{t'\}.
\end{align}

Fix next a two-phase interface $i\in\mathcal{I}$.
Away from contact points but in the vicinity of the interface, we then define
\begin{align}
\label{eq:defWeightAwayContactPoints}
&\vartheta(x,t) := \bar\vartheta_i(x,t),
\\& \nonumber
(x,t) \in \bigcup_{t'\in [0,T]} \bigg(\Omega \cap \Psi_{\mathcal{T}_i}
\big(\mathcal{T}_i(t'){\times}\{t'\}{\times}[-\widehat r,\widehat r]\big)
\setminus \bigcup_{c\in\mathcal{C}} B_{\widehat r}\big(\mathcal{T}_c(t')\big) \bigg) {\times} \{t'\}.
\end{align}

Let now a contact point $c\in\mathcal{C}$ be fixed, and denote by $i\in\mathcal{I}$
the unique two-phase interface with $i\sim c$.
Recall from Definition~\ref{def:locRadius} of the admissible localization radius~$r_c$
that for all $t\in[0,T]$ we decomposed $\Omega \cap B_{r_c}(\mathcal{T}_c(t))$
by means of five pairwise disjoint open wedges $W^{\pm,c}_{\partial\Omega}(t), W^c_{\mathcal{T}_i}(t),
W^c_{\Omega^\pm_v}(t)\subset\R^2$. In the wedge $W^c_{\mathcal{T}_i}$
containing the two-phase interface $\mathcal{T}_i\subset I_v$, we still define
\begin{align}
\label{eq:defWeightInterfaceWedge}
&\vartheta(x,t) := \bar\vartheta_i(x,t),
\quad 
(x,t) \in \bigcup_{t'\in [0,T]} \big(\Omega \cap B_{\widehat r}\big(\mathcal{T}_c(t')\big) 
\cap W^c_{\mathcal{T}_i}(t') \big) {\times} \{t'\}.
\end{align}
In the wedges $W^{\pm,c}_{\partial\Omega}$ containing the domain boundary $\partial\Omega$,
we instead set
\begin{align}
\label{eq:defWeightBoundaryWedge}
&\vartheta(x,t) := \bar\vartheta_{\partial\Omega}^\pm(x,t),
\quad 
(x,t) \in \bigcup_{t'\in [0,T]} \big(\Omega \cap B_{\widehat r}\big(\mathcal{T}_c(t')\big) 
\cap W^{\pm,c}_{\partial\Omega}(t') \big) {\times} \{t'\}.
\end{align}
In the interpolation wedges $W^c_{\Omega^\pm_v}$, we make use of the interpolation
parameter~$\lambda^\pm_c$ of Lemma~\ref{lemma:existenceInterpolationFunction} to interpolate between
the two constructions near the interface~\eqref{eq:defWeightInterfaceWedge}
and near the domain boundary~\eqref{eq:defWeightBoundaryWedge}. 
Recall in this context the convention that $\lambda^\pm_c(\cdot,t)$
was set equal to one on $\big(\partial W^c_{\Omega^\pm_v}(t) \cap \partial W^c_{\mathcal{T}_i}(t)\big)
\setminus \mathcal{T}_c(t)$ and set equal to zero on
 $\big(\partial W^c_{\Omega^\pm_v}(t) 
\cap \partial W^{\pm,c}_{\partial\Omega}(t)\big) \setminus \mathcal{T}_c(t)$ for all $t\in [0,T]$.
With this notation in place, we define on the interpolation wedges
\begin{align}
\label{eq:defWeightInterpolationWedge}
&\vartheta(x,t) := \lambda^\pm_c(x,t)\bar\vartheta_i(x,t)
										+ (1{-}\lambda^\pm_c(x,t))\bar\vartheta_{\partial\Omega}^\pm(x,t),
\\& \nonumber
(x,t) \in \bigcup_{t'\in [0,T]} \big(\Omega \cap B_{\widehat r}\big(\mathcal{T}_c(t')\big) 
\cap W^c_{\Omega^\pm_v}(t') \big) {\times} \{t'\}.
\end{align}
 
Finally, choosing $\delta$ small enough in the definition~\eqref{eq:weightProfileTwoPhase}
of the auxiliary weights $(\vartheta_i)_{i\in\mathcal{I}}$ and recalling the
localization properties~\eqref{eq:locPropertyCutoff1}--\eqref{eq:locPropertyCutoff4}
of the scale $\widehat r$, it is safe to define in the space-time domain not captured
by the definitions~\eqref{eq:defWeightNearBoundary}--\eqref{eq:defWeightInterpolationWedge}
\begin{align}
\label{eq:defWeightBulk}
&\vartheta(x,t) := \mp 1,
\\& \nonumber
(x,t) \in \bigcup_{t'\in [0,T]} \big(\Omega^\pm_v(t') \setminus 
\big(\mathcal{U}_{\widehat r}(t') \cup \Psi_{\partial\Omega}(\partial\Omega{\times}[-\widehat r,\widehat r]
)\big) \big) {\times} \{t'\}.
\end{align}
Recall for this definition also the notation~\eqref{eq:contactInterfaceExcluded}.

\textit{Step 3: (Regularity and coercivity)}
The validity of the asserted sign conditions in Definition~\ref{def:transportedWeight}
are immediate from~\eqref{eq:defWeightNearBoundary}--\eqref{eq:defWeightBulk}.
Since the first-order derivatives of the interpolation
parameter $\lambda^\pm_c$ feature controlled blow-up~\eqref{eq:LambdaFirstDeriv}, it is also a direct
consequence of the definitions~\eqref{eq:defWeightNearBoundary}--\eqref{eq:defWeightBulk}
that $\vartheta\in W^{1,\infty}_{x,t}(\Omega\times [0,T])$ as asserted. 

In view of the definition~\eqref{eq:defWeightBulk}
of the weight in the bulk it suffices to establish~\eqref{eq:lowerBoundTransportedWeight} in the regions
$\Omega \cap \Psi_{\partial\Omega}\big(\partial\Omega{\times}[-\widehat r,\widehat r]\big)
\setminus \mathcal{U}_{\widehat r}(t)$,
$\Omega\cap\Psi_{\mathcal{T}_i}\big(\mathcal{T}_i(t){\times}\{t\}{\times}[-\widehat r,\widehat r]\big)
\setminus \bigcup_{c\in\mathcal{C}} B_{\widehat r}\big(\mathcal{T}_c(t)\big)$
and $\Omega\cap B_{\widehat r}(\mathcal{T}_c(t))$ for all $i\in\mathcal{I}$, all $c\in\mathcal{C}$
and all $t\in [0,T]$. However, in these regions the asserted 
estimate~\eqref{eq:lowerBoundTransportedWeight} is immediately
implied by the properties of the truncation of unity~$\bar\vartheta$ from \textit{Step 1} of this
proof and the definitions~\eqref{eq:defWeightNearBoundary}--\eqref{eq:defWeightInterpolationWedge}.

\textit{Step 4: (Advection equation)} Because of the definition~\eqref{eq:defWeightBulk}
of the weight $\vartheta$ in the bulk, it suffices to establish~\eqref{eq:advDerivTransportedWeight}
in the regions $\Omega \cap \Psi_{\partial\Omega}\big(\partial\Omega{\times}[-\widehat r,\widehat r]\big)
\setminus \mathcal{U}_{\widehat r}(t)$,
$\Omega\cap\Psi_{\mathcal{T}_i}\big(\mathcal{T}_i(t){\times}\{t\}{\times}[-\widehat r,\widehat r]\big)
\setminus \bigcup_{c\in\mathcal{C}} B_{\widehat r}\big(\mathcal{T}_c(t)\big)$
and $\Omega\cap B_{\widehat r}(\mathcal{T}_c(t))$ for all $i\in\mathcal{I}$, all $c\in\mathcal{C}$
and all $t\in [0,T]$. 

Observe first that it follows from the definitions~\eqref{eq:weightProfileBoundary},
\eqref{eq:defWeightNearBoundary} and~\eqref{eq:defWeightBoundaryWedge}
as well as the boundary condition for the fluid velocity $(v\cdot n_{\partial\Omega})|_{\partial\Omega} = 0$
that 
\begin{align}
\label{eq:transportAuxWeightBoundary}
\partial_t\vartheta + (v\cdot\nabla)\vartheta = 0
\quad\text{along } \partial\Omega\setminus\bigcup_{c\in\mathcal{C}} \mathcal{T}_c(t)
\end{align}
for all $t\in [0,T]$. By a Lipschitz estimate together with the coercivity 
estimate~\eqref {eq:lowerBoundTransportedWeight}, the desired estimate~\eqref{eq:advDerivTransportedWeight}
follows in $\Omega \cap \Psi_{\partial\Omega}\big(\partial\Omega{\times}[-\widehat r,\widehat r]\big)
\setminus \mathcal{U}_{\widehat r}(t)$ for all $t\in [0,T]$.

Fix next a two-phase interface $i\in\mathcal{I}$. We then claim that
\begin{align}
\label{eq:transportAuxWeight}
\big|\partial_t \bar\vartheta_i + (v\cdot\nabla) \bar\vartheta_i\big|
\leq C\dist(\cdot,I_v)
\quad\text{in } \Omega\cap\Psi_{\mathcal{T}_i(t)}\big(\mathcal{T}_i(t){\times}[-\widehat r,\widehat r]\big)
\end{align}
for all $t\in [0,T]$. Indeed, one only needs to recall that
the signed distance to the two-phase interface $\mathcal{T}_i\subset I_v$
is transported by the fluid velocity~$v$ to first order in the distance to the interface.
In particular, combining~\eqref{eq:transportAuxWeight} with the definition~\eqref{eq:defWeightAwayContactPoints}
and the coercivity estimate~\eqref{eq:lowerBoundTransportedWeight} entails~\eqref{eq:advDerivTransportedWeight}
in  $\Omega\cap\Psi_{\mathcal{T}_i}\big(\mathcal{T}_i(t){\times}\{t\}{\times}[-\widehat r,\widehat r]\big)
\setminus \bigcup_{c\in\mathcal{C}} B_{\widehat r}\big(\mathcal{T}_c(t)\big)$ for all $t\in [0,T]$.

Let now a contact point $c\in\mathcal{C}$ be given, and let $i\in\mathcal{I}$
be the unique two-phase interface such that $i\sim c$. The desired estimate~\eqref{eq:advDerivTransportedWeight}
follows immediately from~\eqref{eq:transportAuxWeight} and~\eqref{eq:defWeightInterfaceWedge}
in the wedge $\Omega \cap  B_{\widehat r}\big(\mathcal{T}_c(t)\big) \cap W^c_{\mathcal{T}_i}(t)$
for all $t\in [0,T]$.	For the wedges containing the domain boundary $\partial\Omega$,
the estimate~\eqref{eq:advDerivTransportedWeight} in form of 
\begin{align}
\label{eq:transportAuxWeight2}
\hspace*{-1.3ex}
\big|\partial_t \bar\vartheta_{\partial\Omega}^\pm + (v\cdot\nabla)\bar\vartheta_{\partial\Omega}^\pm\big|
\leq C\dist(\cdot,\partial\Omega)
\text{ in } \Omega\cap B_{\widehat r}\big(\mathcal{T}_c(t)\big)
\cap \big(W^c_{\Omega^\pm_v}(t)\cup W^{\pm,c}_{\partial\Omega}(t)\big)
\end{align}
for all $t\in [0,T]$, is satisfied because of the analogue of~\eqref{eq:transportAuxWeightBoundary} 
and a Lipschitz estimate. Finally, in the interpolation wedges one may estimate
\begin{align*}
|\partial_t\vartheta {+} (v\cdot\nabla)\vartheta|
&\leq |\bar\vartheta_i - \bar\vartheta^\pm_{\partial\Omega}|
|\partial_t\lambda^\pm_c {+} (v\cdot\nabla)\lambda^\pm_c|
\\&~~~
+ \lambda^\pm_c |\partial_t\bar\vartheta_i {+} (v\cdot\nabla)\bar\vartheta_i|
+ (1{-}\lambda^\pm_c)|\partial_t\bar\vartheta^\pm_{\partial\Omega} {+} (v\cdot\nabla)\bar\vartheta^\pm_{\partial\Omega}|.
\end{align*}
The desired bound thus follows from the estimate~\eqref{eq:AdvectiveDerivativeLambda} for the advective 
derivative of the interpolation parameter $\lambda^\pm_c$,
the estimates~\eqref{eq:transportAuxWeight} and~\eqref{eq:transportAuxWeight2}, and the fact that the auxiliary weights 
from~\eqref{eq:weightProfileTwoPhase} and~\eqref{eq:weightProfileBoundary} are compatible in the sense
\begin{align*}
|\bar\vartheta_i - \bar\vartheta^\pm_{\partial\Omega}| \leq 
C(\dist(\cdot,\partial\Omega) \wedge \dist(\cdot,I_v))
\end{align*} 
in $\Omega \cap  B_{\widehat r}\big(\mathcal{T}_c(t)) \cap W^c_{\Omega^\pm_v}(t)$
for all $t\in [0,T]$. 
This concludes the proof of Lemma~\ref{lem:existenceTransportedWeight}. \qed

\appendix
\section{Existence of varifold solutions to two-phase fluid flow with surface tension} \label{appendix}
The aim of this Appendix is to give a sketch of a proof regarding
existence of varifold solutions to two-phase fluid flow with surface tension and with ninety degree contact angle (see Definition \ref{Def_varsol}). Note that this is not treated by the work
of Abels~\cite{Abels} in which the existence of a varifold solution in the 
presence of surface tension is only established in a full space setting.
However, in principle it still suggests itself to follow, where possible, the structure of the proof
for the case of an unbounded domain by Abels~\cite{Abels}. 
In this regard, we first discuss two tools which are needed due 
to the different setting of the present work,
i.e., geometric evolution with a ninety degree contact angle condition and
the associated boundary conditions for the solenoidal fluid velocity.
These tools concern an existence result for weak solutions to the
required transport equation (for sufficiently regular transport velocities) 
and elliptic regularity estimates for the 
Helmholtz decomposition associated with the bounded and smooth domain~$\Omega$. 
In a second step, we present the corresponding approximate problem, 
focusing again on the key steps of the proof which differ with respect to the 
case of an unbounded domain studied by Abels~\cite{Abels}. 
Note that analogous to the existence theory of~\cite{Abels},
we will assume some regularity for the geometry of the initial data
and, for simplicity, that the densities of the two fluids
coincide and are normalized to~$1$.

{\it Transport equation.}
In order to construct approximate solutions of the two-phase flow with surface tension and with ninety degree contact angle, one first needs an existence result for weak 
solutions to the transport equation in a bounded domain.
In particular, it suffices to motivate the validity of \cite[Lemma 2.3, $\Omega \equiv \R^d$]{Abels} 
in case of a smooth and bounded domain $\Omega\subset\R^d,\,d\in\{2,3\}$. 

To this aim, let the open subset $\Omega_0^+ \subset \Omega$ be subject to the regularity conditions in Definition \ref{Def_smdom}, let $\chi_0:= \chi_{\Omega_0^+} \in \BV(\Omega;\{0,1\})$,
let $T \in (0,\infty)$, and consider a sufficiently regular fluid velocity 
$v \in C([0,T];C_b^2(\Omega)) \cap C(\overline{\Omega} {\times} [0,T])$ such that  
$\operatorname{div }v = 0$ in~$\Omega$ and $(\no \cdot v )|_{\partial \Omega} = 0$. 
Consider any $C([0,T];C^2_b(\R^d))$ extension of~${v}$ which we 
denote by $\widetilde v$. Then, a solution $\widetilde\chi$
to the transport equation associated with~$\widetilde v$
can be constructed on~$\R^d$ by the usual method of characteristics
(see, e.g., \cite[Proof of Lemma~2.3]{Abels}).
The associated flow map is a $C^1$-diffeomorphism at any time $t \in [0,T]$. 
However, note that it maps $\partial \Omega$ onto itself, 
due to $v|_{\partial\Omega}=\widetilde v|_{\partial\Omega}$ being tangential along~$\partial \Omega$. 
Moreover, since the flow map is a global diffeomorphism (and since continuous images of connected sets are connected), it also maps~$\Omega$ onto itself.
Then, one can conclude by means of the same computations as in the proof of \cite[Lemma 2.3]{Abels} --- using 
in the process the fact that $\operatorname{div} v = 0$ in $\Omega$ ---
that the restriction $\chi:=\widetilde\chi|_{\Omega{\times}[0,T]}
\in L^\infty (0, T ; \BV (\Omega ; \{0, 1\}))$ is a weak solution 
of the transport equation associated with~$v$ in the sense of
	\begin{align}
	\label{eq:transportEquBoundedDomain}
	\int_0^T \int_{\Omega} \chi\left(\partial_{t} \varphi
	+   {v} \cdot \nabla \varphi\right) \mathrm{d}x  \mathrm{d}t
	+		\int_{\Omega} \chi_{0} \varphi(x, 0) \mathrm{d} x=0
	\end{align}
	for any $\varphi \in C^1_c([0,T);C(\overline{\Om})) \cap C_c([0,T);C^1(\overline{\Om}))$.
	Moreover, we have
	\begin{align}\|\chi\|_{L^{\infty}\left(0, T ; B V\left(\Omega\right)\right)} & \leqslant M\left(\|v\|_{C\left([0, T] ; C_{b}^{2}\left(\Omega\right)\right)}\right)\left\|\chi_{0}\right\|_{B V\left(\Omega\right)},
	\\ \frac{\mathrm{d}}{\mathrm{d} t}|\nabla \chi(\cdot,t)|\left(\Omega\right) &=-\left\langle H_{\chi(\cdot,t)}, v(\cdot,t)\right\rangle \quad \text { for all } t \in(0, T) \label{eq:transport} \end{align}
	for some continuous function $M$. Note that the latter holds because the $90$~degree contact angle condition is preserved by sufficiently regular transport velocities
	(see, e.g., the remark after Definition~\ref{Def_strongsol}).

{\it Helmholtz decomposition associated with  bounded domains.}
We recall properties of the Helmholtz projection $P_\Omega$ associated with the smooth bounded domain $\Omega$, referring the reader to \cite[Corollaries 7.4.4-5]{Pruess2016} (see also \cite{Simader1992}).

Define $W_p(\Omega) := \{ g \in W^{1,p}(\Omega;\R^d) : \operatorname{div} g = 0, (g \cdot \no)|_{\partial \Om}=0\}$. 
Given $f \in  W^{1,p}(\Omega;\R^d)$, $2 \leq  p < \infty$, there are unique functions $\phi \in  W^{2,p}(\Omega)$ and $w \in W_p(\Omega) $ such that $f = \nabla \phi + w$.
The bounded linear operator $P_\Om \in \mathcal{B}(W^{1,p}(\Om;\R^d),W_p(\Omega))$ defined by $P_\Om f := w$ is a projection, which is the Helmholtz projection associated with the smooth bounded domain $\Omega$.
Moreover, if $f \in  W^{2,p}(\Omega;\R^d)$ it holds $\phi \in  W^{3,p}(\Omega)$ and 
\begin{align} \label{eq:Helmregularity}
\| P_\Omega f \|_{W^{2,p}(\Omega;\R^d)} \leq C \| f\|_{W^{2,p}(\Omega;\R^d)},
\end{align}
and if $f \in W^{k,2}(\Omega;\R^d)$, $k \geq 2$, then $\phi \in  W^{k,2}(\Omega)$ and
\begin{align} \label{eq:Helmregularity2}
\| P_\Omega f \|_{W^{k,2}(\Omega;\R^d)} \leq C \| f\|_{W^{k,2}(\Omega;\R^d)}.
\end{align}
This follows from existence and regularity theory of the
associated Neumann problem (see for the case $p>2$ the result of \cite[Corollary 7.4.5]{Pruess2016})
\begin{align*}
	\Delta \phi &= \operatorname{div} f && \text{ in } \Omega,\\
 (\no \cdot  \nabla )\phi  &= f \cdot \no  && \text{ on } \partial \Omega.
\end{align*}

{\it Solutions to approximate two-phase fluid flow.}
In order to formulate the approximate equations, let $\psi$ be a standard mollifier, for every~$k\in\mathbb{N}$ we denote by $\psi_k:=k^d\psi(k\cdot)$ its usual rescaling, and by~$P_{\Omega}$ the Helmholtz projection associated with the smooth domain~$\Omega$.
Moreover, let $\Psi_k \cdot = P_\Omega(\Psi_k \ast \cdot )$. 
Consider the initial data $v_0 \in L^2(\Omega)$ with $\operatorname{div} v_0 = 0$ and $(\no \cdot v_0 )|_{\partial \Omega} = 0$, and let $\chi_0:= \chi_{\Omega_0^+} \in \BV(\Omega;\{0,1\})$, where $\Omega_0^+ \subset \Omega$ is subject to the regularity conditions in Definition~\ref{Def_smdom}.
Let $\mu,\sigma > 0$. 
Then, we consider an approximate two-phase flow on $(0, T_w )$, $T_w \in (0,\infty)$. 
This is a pair~$(v_k,\chi_k)$ consisting on one side of a fluid velocity field
$v_k  \in L^\infty([0,T_w];L^2(\Omega))\cap L^2([0,T_w];W_2(\Omega))$ solving
	\begin{align} 
& \into  v_k(\cdot, T) \cdot \eta (\cdot, T) \dx - \into  v_0 \cdot \eta (\cdot, 0) \dx
- \int_{0}^{T} \into  v_k \cdot \partial_t \eta \dx \dt \nonumber \\
&- \int_{0}^{T}   \into \Psi_{k} v_{k} \otimes \psi_{k} * v_{k}  : \nabla (\psi_{k} *\eta) \dx \dt  
+ \int_{0}^{T}  \into \mu (\nabla v_k + \nabla v_k^\mathsf{T}) : \nabla \eta \dx \dt \nonumber \\
&= \sigma \int_{0}^{T} \int_{\partial^*\{\chi_k=1\}\cap\Omega} \mathrm{H}_{\chi_k} \cdot  \Psi_{k}\eta  \dS \dt \label{eq:approxeq}
\end{align}
for a.e.\ \(T \in  [0, T_w)\) and every
$\eta \in C^\infty ([0, T_w); C^1(\overline{\Omega};\R^d ) \cap\bigcap_{p\geq 2}W^{2,p}(\Omega;\R^d) )$ with $\operatorname{div} \eta = 0$ and $(\no \cdot \eta )_{\partial \Omega} = 0$, 
and on the other side an evolving phase indicator
$\chi_k \in L^\infty ([0, T_w] ; \BV (\Omega ; \{0, 1\}))$ 
which is the unique weak solution --- in the sense of~\eqref{eq:transportEquBoundedDomain} --- to 
the transport equation
\begin{align*}
	\partial_{t} \chi_{k}+\left(\Psi_{k} v_{k}\right) \cdot \nabla \chi_{k} &= 0 
	&&  \text { in } (0, T_w) \times \Omega,\\ 
	\chi_{k}|_{t=0} &= \chi_{0} &&  \text { in } \Omega.
\end{align*}
The existence of approximate solutions $(v_k, \chi_k)$ satisfying the energy equality 
\begin{align}
\nonumber
&\frac{1}{2}\|v_k(\cdot,T)\|^2_{L^2(\Omega)}
+ \sigma|\nabla\chi_k(\cdot,T)|(\Omega)
+ \frac{\mu}{2}\|\nabla v_k\|^2_{L^2(\Omega{\times}(0,T))}
\\& \label{eq:energyEqualityApproximateFlow}
= \frac{1}{2}\|v_0\|^2_{L^2(\Omega)}
+ \sigma|\nabla\chi_0|(\Omega),
\quad T \in (0,T_w),
\end{align}
and satisfying
\begin{align}
\label{eq:helly}
\text{the map} \quad (0,T_w) \ni t \mapsto |\nabla\chi_k(\cdot,t)|(\Omega)
\quad\text{is absolutely continuous},
\end{align}
can then be proved by means of a fixed-point argument as done in~\cite[Proof of Theorem~4.2]{Abels}, 
relying in the process on the above two ingredients corresponding to 
the different setting of the present work: the existence result for weak solutions to the transport 
equation~\eqref{eq:transportEquBoundedDomain}
with sufficiently regular transport velocity, and the
elliptic regularity estimates~\eqref{eq:Helmregularity2} for the Helmholtz projection
associated with~$\Omega$. In particular, one obtains uniform bounds
\begin{align}
\label{eq:uniformBoundsApproximateVelocity}
\sup_{k \in \mathbb{N}} \sup_{t \in (0,T_w)} \|v_k(\cdot,t)\|^2_{L^2(\Omega)}
+ \sup_{k \in \mathbb{N}} \|\nabla v_k\|^2_{L^2(\Omega{\times}(0,T_w))} &< \infty,
\\ \label{eq:uniformBoundsApproximatePhase}
\sup_{k \in \mathbb{N}} \sup_{t \in (0,T_w)} |\nabla\chi_k(\cdot,t)|(\Omega) &< \infty.
\end{align}

{\it Limit passage in the approximation scheme to a varifold solution.}
As for the passage to the limit, we only discuss the surface tension term on the right hand side of the approximate problem~\eqref{eq:approxeq} as well as the validity of 
the energy inequality~\eqref{endisineq}. The other terms 
as well as the passage to the limit in the transport equation can be treated as in~\cite{Abels}.
First, we define a varifold $V_k \in \mathcal{M}((0, T_w)\times \overline{\Omega} \times \Ss)$ by
\begin{align}
\label{eq:disintTimeApproxVarifold}
	{V}_k := \mathcal{L}^{1}\llcorner(0, T_w) \otimes \left( V_k(t)  \right)_{t \in (0, T_w)},
\end{align}
where
\[
  V_k(t) := |\nabla \chi_k (\cdot, t) | \llcorner \Omega \otimes 
	\big(\delta_{\frac{\nabla\chi_k(\cdot,t)}{|\nabla\chi_k(\cdot,t)|}}\big)_{x \in \Omega}
	\in \mathcal{M}(\overline{\Omega}{\times}\Ss)\quad \text{ for any } t \in (0, T_w).
\]
	Since $\chi_k \in L^\infty([0,T_w]; \BV(\Omega; \{0,1\}))$ is uniformly bounded
	in the sense of~\eqref{eq:uniformBoundsApproximatePhase}, there then exists 
	$\chi \in L^\infty([0,T_w]; \BV(\Omega; \{0,1\}))$ such that, up to taking a subsequence,
	\begin{align}
	\label{eq:convPhase}
	\chi_k &\rightharpoonup^* \chi
	&&\text{ in } L^\infty(\Omega{\times}(0,T_w)),
	\\ \label{eq:convGradientPhase}
	\nabla \chi_k &\rightharpoonup^* \nabla \chi 
	&&\text{ in } L^\infty([0,T_w]; \mathcal{M}(\Omega)).
	\end{align}
	Moreover, we have $\sup_k \| {V}_k\|_\mathcal{M} < \infty$ due to~\eqref{eq:uniformBoundsApproximatePhase}
	and the definition of~$V_k$. In particular,
	there exists ${V} \in \mathcal{M}((0, T_w)\times \overline{\Omega} \times \Ss)$ 
	such that, up to taking a subsequence,
	\begin{align}
	\label{eq:convergenceVarifold}
	{V}_k &\rightharpoonup^*  {V} 
	&&\text{in }\mathcal{M}((0, T_w)\times \overline{\Omega} \times \Ss).
	\end{align}
	Note that the compatibility condition~\eqref{compcond} then
	simply follows from exploiting~\eqref{eq:convGradientPhase} and~\eqref{eq:convergenceVarifold}.
	As a preparation for the remaining arguments, note also that
	thanks to the condition~\eqref{eq:helly}
	a careful inspection of the argument of~\cite[Lemma~2]{Hensel2021e} reveals
	that one may disintegrate the limit varifold~$V$ in form of
	\begin{align}
	\label{eq:disintTimeLimitVarifold}
	{V}= \mathcal{L}^{1}\llcorner(0, T_w) \otimes \left( V_t\right)_{t \in (0, T_w)},
	\quad V_t \in \mathcal{M}( \overline{\Omega} {\times} \Ss),\,t \in (0,T_w),
	\end{align}
  and that the limit interface energy satisfies
	\begin{align}
	\label{eq:energyInterfaceLimit}
	|V_t|_{\Ss}(\overline\Omega) \leq \liminf_k |\nabla\chi_k(\cdot,t)|(\Omega)
	\quad \text{for a.e.\ } t \in [0, T_w).
	\end{align}

	For any $\eta \in C^\infty ([0, T_w); C^1(\overline{\Omega};\R^d ) \cap\bigcap_{p\geq 2}W^{2,p}(\Omega;\R^d) )$ such that $\operatorname{div} \eta = 0$ and $(\eta \cdot \no)|_{\partial \Omega} =0 $, we discuss the limit of 
	\begin{align*}
	\int_{0}^{T}	\int_{  \Om} \left(\Id -  \frac{\nabla \chi_k}{|\nabla \chi_k|}\otimes  \frac{\nabla \chi_k}{|\nabla \chi_k|} \right): \nabla (\Psi_k \eta) \, \mathrm{d} |\nabla \chi_k|  \dt \quad \text{ for } k \rightarrow \infty,
	\end{align*}
	for almost every \(T \in  [0, T_w)\).
	By adding a zero, we obtain 
	\begin{align*}
	&\int_{0}^{T} \int_{  \Om} \left(\Id -  \frac{\nabla \chi_k}{|\nabla \chi_k|}\otimes  \frac{\nabla \chi_k}{|\nabla \chi_k|} \right): \nabla (\Psi_k \eta - \eta ) \,  \mathrm{d} |\nabla \chi_k|  \dt \\
	&+ \int_{0}^{T} \int_{ \overline{\Omega} \times \Ss } \left(\Id -  s \otimes  s \right): \nabla \eta  \,    \mathrm{d} V_k(t,x,s) \,,
	\end{align*}
   where the second term converges to 
	$
	\int_{0}^{T} 	\int_{ \overline{\Omega} \times \Ss } \left(\Id -  s \otimes  s\right): \nabla \eta  \, \mathrm{d} V_t(x,s)$ for $k \rightarrow \infty$ for any $ \eta \in C^\infty_{0} ([0, T_w); C^1(\overline{\Omega};\R^d ) \cap\bigcap_{p\geq 2}W^{2,p}(\Omega;\R^d) )$. 
	Indeed, the latter guarantees $(\mathrm{Id}{-}s\otimes s):\nabla\eta \in C_0((0,T_w)
	{\times}\overline{\Omega}{\times}\Ss)$ so that one may use~\eqref{eq:convergenceVarifold}
	for such~$\eta$. However, the additional support assumption on the time variable
	can be removed by means of a standard truncation argument relying on the disintegration
	formulas~\eqref{eq:disintTimeApproxVarifold} and~\eqref{eq:disintTimeLimitVarifold}, respectively,
	and the uniform bound $\sup_k \| {V}_k\|_\mathcal{M} < \infty$.
	As for the first term, we exploit the regularity properties of the Helmholtz 
	projection. More precisely, we may estimate
	for any $p>3$ based on~\eqref{eq:Helmregularity}
	and the Sobolev embedding $W^{1,p}({\Om}) 
	\hookrightarrow C(\overline{\Om})$, $d \in \{2,3\}$,
	\begin{align*}
	& \left| \int_{0}^{T} 	\int_{  \Om} \left(\Id -  \frac{\nabla \chi_k}{|\nabla \chi_k|}\otimes  \frac{\nabla \chi_k}{|\nabla \chi_k|} \right): \nabla (\Psi_k \eta - \eta ) \, \mathrm{d} |\nabla \chi_k|  \dt  \right| \\
	&\leq C \int_{0}^{T}	\left\|	 \nabla (\Psi_k \eta - \eta )   \right\|_{C(\overline{\Om}; \R^{d \times d })} \dt \\
	&\leq C	  \int_{0}^{T}\left\|	 \nabla P_\Omega(\psi_k \ast  \eta - \eta )   \right\|_{C(\overline{\Om} ; \R^{d \times d })} \dt \\
	& \leq C  \int_{0}^{T} \left\|	 \psi_k \ast  \eta - \eta    \right\|_{W^{2,p}({\Om};   \R^{d })}  \dt.
	\end{align*}
The right hand side obviously goes to zero by letting $k \rightarrow \infty$. In summary,
we obtain as desired
\begin{align*}
&\int_{0}^{T}	\int_{  \Om} \left(\Id -  \frac{\nabla \chi_k}{|\nabla \chi_k|}\otimes  \frac{\nabla \chi_k}{|\nabla \chi_k|} \right): \nabla (\Psi_k \eta) \, \mathrm{d} |\nabla \chi_k|  \dt
\\&
\to \int_{0}^{T} 	\int_{ \overline{\Omega} \times \Ss } 
\left(\Id -  s \otimes  s\right): \nabla \eta  \, \mathrm{d} V_t(x,s)
\quad \text{for } k \to \infty,
\end{align*}
for almost every $T \in [0,T_w)$
and all $\eta \in C^\infty ([0, T_w); C^1(\overline{\Omega};\R^d ) \cap\bigcap_{p\geq 2}W^{2,p}(\Omega;\R^d) )$ such that $\operatorname{div} \eta = 0$ and $(\eta \cdot \no)|_{\partial \Omega} =0 $.

	At last, we comment how to recover the energy inequality~\eqref{endisineq}.
	This can be obtained from combining the energy equality~\eqref{eq:energyEqualityApproximateFlow}
	with the lower-semicontinuity property~\eqref{eq:energyInterfaceLimit}
	and the convergence properties of~$v_k$ to its limit~$v$
	(i.e., up to a subsequence, $v_k \rightharpoonup v$ in $L^2(0,T_w;H^1(\Omega))$
	and $v_k \rightharpoonup^* v$ in $L^\infty(0,T_w;L^2(\Omega))$
	due to the uniform bound~\eqref{eq:uniformBoundsApproximateVelocity}).

\section*{Acknowledgments}
The authors warmly thank their former resp.\ current PhD advisor Julian Fischer 
for the suggestion of this problem and for valuable initial discussions on the subjects of this paper. 
This project has received funding from the European Research Council 
(ERC) under the European Union's Horizon 2020 research and innovation 
programme (grant agreement No 948819)
\begin{tabular}{@{}c@{}}\includegraphics[width=8ex]{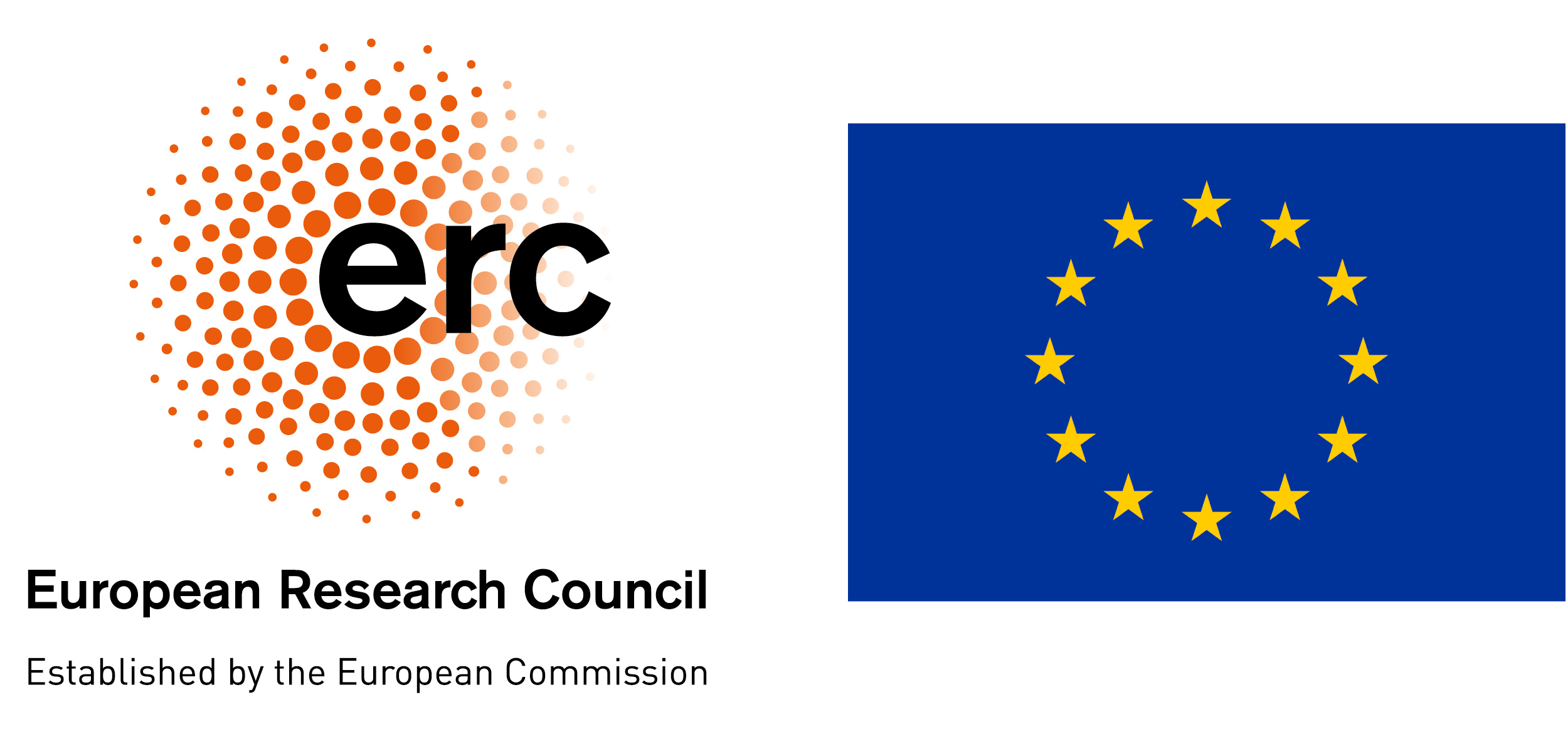}\end{tabular}, and 
from the Deutsche Forschungsgemeinschaft (DFG, German Research Foundation) under 
Germany's Excellence Strategy -- EXC-2047/1 -- 390685813.

\bibliographystyle{abbrv}
\bibliography{two_fluid_contact}

\end{document}